\documentclass[12pt,twoside]{amsart}
\usepackage[latin1]{inputenc}
\usepackage{amsmath, amsthm, amscd, amsfonts, amssymb, graphicx}
\usepackage[bookmarksnumbered, plainpages]{hyperref}

\newtheorem{thm}{Theorem}[section]

\newtheorem{prop}[thm]{Proposition}
\newtheorem{defn}[thm]{Definition}

\numberwithin{equation}{section}

\begin{document}

\title{\bf Affine connections on singular multiply warped products and singular twisted products}
\author{Siyao Liu  \hskip 0.4 true cm Tong Wu \hskip 0.4 true cm  Yong Wang$^{*}$}

\thanks{{\scriptsize
\hskip -0.4 true cm \textit{2010 Mathematics Subject Classification:}
53C40; 53C42.
\newline \textit{$^{*}$Corresponding author}}}

\maketitle

\begin{abstract}
 \indent In this paper, we generalize the results in [Y. Wang: Affine connections on singular warped products. Int. J. Geom. Methods Mod. Phys. 18(5), 2150076, (2021).] to singular multiply warped products and singular twisted products.
We study singular multiply warped products and singular twisted products and their curvature with the semi-symmetric metric connection and the semi-symmetric non-metric connection. We also discuss Koszul forms associated with the almost product structure and their curvature of singular multiply warped products and singular twisted products.
Finally, several examples are presented to demonstrate the theoretical results.

\end{abstract}

\vskip 0.2 true cm


\pagestyle{myheadings}
\markboth{\rightline {\scriptsize Liu}}
         {\leftline{\scriptsize  Affine connections on singular multiply warped products and singular twisted products}}

\bigskip
\bigskip


\section{ Introduction}
Bishop and O'Neill presented the concept of warped products \cite{RL}.
Multiply warped products are natural generalization of warped products, for example, generalized Robertson-Walker space-times.
The warped products provides a way to construct new semi-Riemannian manifolds from known ones.
This construction are useful for analyzing and interpreting General Relativity, cosmological models and black holes.
Due to the wide applicability, the topics on warped products have received increased attentions from geometers, for instance, some researchers using warped product spaces found Einstein manifolds and manifolds with constant scalar curvature \cite{DF,OC5,Wang4}.

Stoica extended the definition of the warped product between two non-degenerate semi-Riemannian manifolds to two singular semi-Riemannian manifolds in \cite{OC2}.
If the metric becomes degenerate, the Levi-Civita connection and the Riemann curvature no longer work, because they are based on the inverse of the metric, and on related operations.
For the Friedmann-Lemaitre-Robertson-Walker model, the metric of the product manifold becomes degenerate, and the Levi-Civita connection and Riemann curvature, as usually defined, become singular or undefined.
For this purpose, in \cite{OC1}, Stoica gave the properties of the Koszul form, which are similar to those of the Levi-Civita connection.
Stoica also obtained an useful formula that the Riemann curvature tensor was expressed directly in terms of the Koszul form.
Wang generalized multiply warped products to singular multiply warped products in \cite{Wang1}.

By changing the domain of the warping function, we can get the definition of the twisted product.
Wang studied the multiply twisted products in \cite{Wang2}, as generalizations of multiply warped products and twisted products.\par

In \cite{HA}, the notion of a semi-symmetric metric connection on a Riemannian manifold was introduced.
N.S. Agashe and M.R. Chafle showed some properties of submanifolds of a Riemannian manifold with a semi-symmetric non-metric connection \cite{NS1, NS2}.
Semi-symmetric metric connection and semi-symmetric non-metric connection on a Riemannian manifold, have been studied extensively by, among others,  by Demirbag et al \cite{Wang3, OC4, SD, AG}.
In \cite{Wang1}, the author introduced semi-symmetric metric Koszul forms and semi-symmetric non-metric Koszul forms on singular semi-Riemannian manifolds.
Wang also introduced Koszul forms associated with the almost product structure on singular almost product semi-Riemannian manifolds.\par

It is natural to try to relate singular multiply warped products to the semi-symmetric metric Koszul forms, semi-symmetric non-metric Koszul forms and Koszul forms associated with the almost product structure, respectively.
To make up the gaps in the research, the current study gives the definitions of singular twisted products, similarly we can study the singular twisted products and their curvature with the semi-symmetric metric connection and semi-symmetric non-metric connection. 
We also discuss Koszul forms associated with the almost product structure and their curvature of the singular twisted products.\par

A brief description of the organization of this paper is as follows.
In Section 2, this paper will firstly introduce the definitions of singular multiply warped products, then give the formulae of semi-symmetric metric Koszul forms. Via some simple calculations,
we generalize the results in \cite{Wang1} to singular multiply warped products.
In the third section, we briefly sketch semi-symmetric non-metric Koszul forms, hereby get the Riemann curvature tensor of the semi-symmetric non-metric covariant derivative on a semi-symmetric non-metric semi-regular semi-Riemannian manifold.
In Section 4, we extend the results of Theorem 4.17 in \cite{Wang1} to singular multiply warped products.
In the next Section, we generalize the twisted product to singular semi-Riemannian manifolds.
Firstly, we compute the Koszul form on a degenerate twisted product, and secondly we study the Riemann curvature tensor of the covariant derivative on a semi-regular semi-Riemannian manifold. Then we give the semi-symmetric metric Koszul forms on degenerate twisted products and the Riemann curvature tensor of the semi-symmetric metric covariant derivative on a semi-symmetric metric semi-regular semi-Riemannian manifold.
Section 6 is devoted to the study of semi-symmetric non-metric Koszul forms and their curvature of singular twisted products.
In Section 7, it is shown that the Koszul forms associated with the almost product structure on a singular almost product semi-Riemannian manifold. We also develop the Theorem 4.8 to singular twisted products.
In Section 8, we give four examples to illustrate the results in this paper.

\vskip 1 true cm
\section{ Semi-Symmetric Metric Koszul Forms and Their Curvature of Singular Multiply Warped Products}

We begin by recalling the notion of multiply warped products of singular semi-Riemannian manifolds.

\begin{defn}{\rm\cite{Wang1} }
Let $(B, g_{B})$ and $(F_{j}, g_{F_j}),~ 1\leq j\leq m,$  be singular semi-Riemannian manifolds, and $b_{j}: B\rightarrow \mathbb{R},~ 1\leq j\leq m,$ smooth functions. The multiply warped product of $B$ and $F_{j}$ with warping function $b_{j},~ 1\leq j\leq m,$ is the semi-Riemannian manifold
\begin{equation}
B\times_{b_{1}}F_{1}\times\cdots\times_{b_{m}}F_{m}:=(B\times F_{1}\times\cdots\times F_{m}, \pi^{*}_{B}(g_{B})+\sum_{j=1}^{m}(b_{j}\circ\pi_{B})^{2}\pi^{*}_{F_{j}}(g_{F_j})),
\end{equation}
where $\pi_{B}: B\times F_{1}\times\cdots\times F_{m}\rightarrow B$ and $\pi_{F_{j}}: B\times F_{1}\times\cdots\times F_{m}\rightarrow F_{j}$ are the canonical projections.
\end{defn}

The inner product on $B\times_{b_{1}}F_{1}\times\cdots\times_{b_{m}}F_{m}$ takes, for any point $p \in B\times F_{1}\times\cdots\times F_{m}$ and for any pair of tangent vectors $x, y \in T_{p}(B\times F_{1}\times\cdots\times F_{m}),$ the explicit form
\begin{equation}
g(x, y)=g_{B}(d\pi_{B}(x), d\pi_{B}(y))+\sum_{j=1}^{m}b^{2}_{j}(p)g_{F_{j}}(d\pi_{F_{j}}(x), d\pi_{F_{j}}(y)).
\end{equation}

Let us look back to the definitions of the Koszul form and semi-symmetric metric Koszul forms.

\begin{defn}{\rm\cite{OC1} }
Let $(M, g)$ be a singular semi-Riemannian manifold. The Koszul form is defined as $\mathcal{K}: \Gamma(TM)^{3}\rightarrow C^{\infty}(M),$
\begin{align}
\mathcal{K}(X, Y, Z):=&\frac{1}{2}\{X(g(Y, Z))+Y(g(Z, X))-Z(g(X, Y))\\
&-g(X, [Y, Z])+g(Y, [Z, X])+g(Z, [X, Y])\}.\nonumber
\end{align}
\end{defn}

\begin{defn}{\rm\cite{OC1} }
Let $X, Y\in\Gamma(TM).$ The lower covariant derivative of $Y$ in the direction of $X$ is defined as the differential $1$-form $\nabla^{\flat}_{X}Y\in A^{1}(M)$
\begin{equation}
(\nabla^{\flat}_{X}Y)(Z):=\mathcal{K}(X, Y, Z),
\end{equation}
for any $Z\in\Gamma(TM).$
\end{defn}

\begin{defn}{\rm\cite{Wang1} }
Let $P$ be a vector field on $M.$ Semi-symmetric metric Koszul forms $\overline{\mathcal{K}_{p}}: \Gamma(TM)^{3}\rightarrow C^{\infty}(M)$ on $(M, g)$ is defined as
\begin{equation}
\overline{\mathcal{K}_{p}}(X, Y, Z):=\mathcal{K}(X, Y, Z)+g(Y, P)g(X, Z)-g(X, Y)g(P, Z).
\end{equation}
Usually, we write it $\overline{\mathcal{K}}$ for brevity.
\end{defn}

\begin{defn}{\rm\cite{Wang1} }
Let $X, Y\in\Gamma(TM).$ The semi-symmetric metric lower covariant derivative of $Y$ in the direction of $X$ as the differential $1$-form $\overline{\nabla}^{\flat}_{X}Y\in A^{1}(M)$
\begin{equation}
(\overline{\nabla}^{\flat}_{X}Y)(Z):=\overline{\mathcal{K}}(X, Y, Z),
\end{equation}
for any $Z\in\Gamma(TM).$
\end{defn}

Similar to Proposition 2.15 and Proposition 2.16 in \cite{Wang1}, we have the following statements.
\begin{prop}
Let $B\times_{b_{1}}F_{1}\times\cdots\times_{b_{m}}F_{m}$ be a degenerate multiply warped product and let the vector fields $X, Y, Z\in\Gamma(TB)$ and $U_{j}, V_{j}, W_{j}\in\Gamma(TF_{j}).$ Let $\overline{\mathcal{K}}$ be the semi-symmetric metric Koszul form on $B\times_{b_{1}}F_{1}\times\cdots\times_{b_{m}}F_{m}$ and $\overline{\mathcal{K}}_{B}$, $\overline{\mathcal{K}}_{F_{j}}$ be the lifts of the semi-symmetric metric
Koszul form on $B,$ $F_{j},$ respectively. Let $P\in\Gamma(TB),$ then
\begin{align}
&(1)\overline{\mathcal{K}}(X, Y, Z)=\overline{\mathcal{K}}_{B}(X, Y, Z);\nonumber\\
&(2)\overline{\mathcal{K}}(X, Y, W_{j})=\overline{\mathcal{K}}(X, W_{j}, Y)=\overline{\mathcal{K}}(W_{j}, X, Y)=0;\nonumber\\
&(3)\overline{\mathcal{K}}(X, V_{i}, W_{j})=b_{j}X(b_{j})g_{F_{j}}(V_{j}, W_{j}),~if~i=j;\nonumber\\
&(4)\overline{\mathcal{K}}(V_{i}, X, W_{j})=-\overline{\mathcal{K}}(V_{i}, W_{j}, X)=b_{j}X(b_{j})g_{F_{j}}(V_{j}, W_{j})+b^{2}_{j}g_{B}(X, P)g_{F_{j}}(V_{j}, W_{j}),~if~i=j;\nonumber\\
&(5)\overline{\mathcal{K}}(X, V_{i}, W_{j})=\overline{\mathcal{K}}(V_{i}, X, W_{j})=\overline{\mathcal{K}}(V_{i}, W_{j}, X)=0,~if~i\neq j;\nonumber\\
&(6)\overline{\mathcal{K}}(U_{i}, V_{j}, W_{k})=b^{2}_{j}\mathcal{K}_{F_{j}}(U_{j}, V_{j}, W_{j}),~if~i=j=k;\nonumber\\
&(7)\overline{\mathcal{K}}(U_{i}, V_{j}, W_{k})=0,~other~cases.\nonumber
\end{align}
\end{prop}

\begin{prop}
Let $B\times_{b_{1}}F_{1}\times\cdots\times_{b_{m}}F_{m}$ be a degenerate multiply warped product and let the vector fields $X, Y, Z\in\Gamma(TB)$ and $U_{j}, V_{j}, W_{j}\in\Gamma(TF_{j}).$ Let $\overline{\mathcal{K}}$ be the semi-symmetric metric Koszul form on $B\times_{b_{1}}F_{1}\times\cdots\times_{b_{m}}F_{m}$ and $\overline{\mathcal{K}}_{B}$, $\overline{\mathcal{K}}_{F_{j}}$ be the lifts of the semi-symmetric metric
Koszul form on $B,$ $F_{j},$ respectively. Let $P\in\Gamma(TF_{l}),$ then
\begin{align}
&(1)\overline{\mathcal{K}}(X, Y, Z)=\mathcal{K}_{B}(X, Y, Z);\nonumber\\
&(2)\overline{\mathcal{K}}(X, Y, W_{j})=-\overline{\mathcal{K}}(X, W_{j}, Y)=-b^{2}_{j}g_{B}(X, Y)g_{F_{j}}(W_{j}, P),~if~j=l;\nonumber\\
&(3)\overline{\mathcal{K}}(X, Y, W_{j})=\overline{\mathcal{K}}(X, W_{j}, Y)=0,~if~j\neq l;\nonumber\\
&(4)\overline{\mathcal{K}}(W_{j}, X, Y)=0;\nonumber\\
&(5)\overline{\mathcal{K}}(X, V_{i}, W_{j})=\overline{\mathcal{K}}(V_{i}, X, W_{j})=-\overline{\mathcal{K}}(V_{i}, W_{j}, X)=b_{j}X(b_{j})g_{F_{j}}(V_{j}, W_{j}),~if~i=j;\nonumber\\
&(6)\overline{\mathcal{K}}(X, V_{i}, W_{j})=\overline{\mathcal{K}}(V_{i}, X, W_{j})=\overline{\mathcal{K}}(V_{i}, W_{j}, X)=0,~if~i\neq j;\nonumber\\
&(7)\overline{\mathcal{K}}(U_{i}, V_{j}, W_{k})=b^{2}_{j}\mathcal{K}_{F_{j}}(U_{j}, V_{j}, W_{j})+b^{4}_{j}g_{F_{j}}(U_{j}, W_{j})g_{F_{j}}(V_{j}, P)-b^{4}_{j}g_{F_{j}}(U_{j}, V_{j})g_{F_{j}}(W_{j}, P),\nonumber\\
&~if~i=j=k=l;\nonumber\\
&(8)\overline{\mathcal{K}}(U_{i}, V_{j}, W_{k})=-\overline{\mathcal{K}}(U_{i}, W_{k}, V_{j})=-b^{2}_{j}b^{2}_{k}g_{F_{j}}(U_{j}, V_{j})g_{F_{k}}(W_{k}, P),~if~i=j\neq k=l;\nonumber\\
&(9)\overline{\mathcal{K}}(U_{i}, V_{j}, W_{k})=0,~other~cases.\nonumber
\end{align}
\end{prop}

On a semi-regular semi-Riemannian manifold, we define the Riemann curvature tensor of the covariant derivative
\begin{equation}
R(X, Y, Z, T):=(\nabla_{X}\nabla^{\flat}_{Y}Z)(T)-(\nabla_{Y}\nabla^{\flat}_{X}Z)(T)-(\nabla^{\flat}_{[X, Y]}Z)(T).
\end{equation}
Likewise, on a semi-symmetric metric semi-regular semi-Riemannian manifold, we define the Riemann curvature tensor of the semi-symmetric metric covariant derivative
\begin{equation}
\overline{R}(X, Y, Z, T):=(\overline{\nabla}_{X}\overline{\nabla}^{\flat}_{Y}Z)(T)-(\overline{\nabla}_{Y}\overline{\nabla}^{\flat}_{X}Z)(T)-(\overline{\nabla}^{\flat}_{[X, Y]}Z)(T).
\end{equation}

Via some simple calculations, we can get
\begin{prop}{\rm\cite{Wang1} }
For any vector fields $X, Y, Z, T\in\Gamma(TM)$ on a semi-symmetric metric semi-regular semi-Riemannian manifold $(M, g),$
\begin{align}
\overline{R}(X, Y, Z, T)&=R(X, Y, Z, T)-g(Y, P)g(X, Z)g(P, T)+g(X, P)g(Y, Z)g(P, T)\\
&+\overline{\mathcal{K}}(X, P, Z)g(Y, T)-\overline{\mathcal{K}}(Y, P, Z)g(X, T)\nonumber\\
&-\mathcal{K}(X, P, T)g(Y, Z)+\mathcal{K}(Y, P, T)g(X, Z)\nonumber.
\end{align}
\end{prop}

And we have
\begin{equation}
\overline{R}(X, Y, Z, T)=-\overline{R}(Y, X, Z, T),~~~~~~~ \overline{R}(X, Y, Z, T)=-\overline{R}(Y, X, T, Z).
\end{equation}

According to the detailed descriptions in \cite{Wang1}, we know that
\begin{thm}
Let $(B, g_{B})$ be a non-degenerate semi-Riemannian manifold and $(F_{j}, g_{F_j})$ be semi-regular semi-Riemannian manifolds and $b_{j}\in C^{\infty}(B),~ 1\leq j\leq m.$ Then the multiply warped product manifold $B\times_{b_{1}}F_{1}\times\cdots\times_{b_{m}}F_{m}$ is a semi-regular semi-Riemannian manifold.
\end{thm}

\begin{prop}
A singular semi-Riemannian manifold $(M, g)$ is a semi-regular semi-Riemannian manifold if and only if it is a semi-symmetric metric semi-regular semi-Riemannian manifold.
\end{prop}

Summarizing, we have
\begin{thm}
Let $(B, g_{B})$ be a non-degenerate semi-Riemannian manifold, $(F_{j}, g_{F_j})$ be semi-regular semi-Riemannian manifolds and $b_{j}\in C^{\infty}(B),~ 1\leq j\leq m.$
Let the vector fields $X, Y, Z, T\in\Gamma(TB)$ and $U_{j}, V_{j}, W_{j}, Q_{j}\in\Gamma(TF_{j}).$  Let $P\in\Gamma(TB),$ then
\begin{align}
&(1)\overline{R}(X, Y, Z, T)=\overline{R}_{B}(X, Y, Z, T);\nonumber\\
&(2)\overline{R}(X, Y, Z, W_{j})=\overline{R}(Z, W_{j}, X, Y)=0;\nonumber\\
&(3)\overline{R}(X, Y,  V_{i}, W_{j})=\overline{R}(V_{i}, W_{j}, X, Y)=0,~if~i=j;\nonumber\\
&(4)\overline{R}(X, V_{i}, W_{j}, Y)=b_{j}g^{*}_{B}(\nabla^{\flat}_{X}Y, db_{j})g_{F_{j}}(V_{j}, W_{j})-b_{j}XY(b_{j})g_{F_{j}}(V_{j}, W_{j})-b^{2}_{j}g_{F_{j}}(V_{j}, W_{j})\nonumber\\
&\times \mathcal{K}_{B}(X, P, Y)-b_{j}P(b_{j})g_{B}(X, Y)g_{F_{j}}(V_{j}, W_{j})+b^{2}_{j}g_{B}(X, P)g_{B}(Y, P)g_{F_{j}}(V_{j}, W_{j})\nonumber\\
&-b^{2}_{j}g_{B}(X, Y)g_{B}(P, P)g_{F_{j}}(V_{j}, W_{j}),~if~i=j;\nonumber\\
&(5)\overline{R}(X, Y,  V_{i}, W_{j})=\overline{R}(V_{i}, W_{j}, X, Y)=\overline{R}(X, V_{i}, W_{j}, Y)=0,~if~i\neq j;\nonumber\\
&(6)\overline{R}(X, V_{i}, W_{j}, U_{k})=\overline{R}(W_{j}, U_{k}, X, V_{i})=0,~if~i=j=k;\nonumber\\
&(7)\overline{R}(X, V_{i}, W_{j}, U_{k})=\overline{R}(X,  U_{k}, V_{i}, W_{j})=\overline{R}(W_{j}, U_{k}, X, V_{i})=\overline{R}(V_{i}, W_{j}, X,  U_{k})=0,\nonumber\\
&~if~i=j\neq k;\nonumber\\
&(8)\overline{R}(X, V_{i}, W_{j}, U_{k})=\overline{R}(W_{j}, U_{k}, X, V_{i})=0,~where~i, j, k~are~different;\nonumber
\end{align}
\begin{align}
&(9)\overline{R}(U_{k}, V_{i}, W_{j}, Q_{s})=b^{2}_{j}R_{F_{j}}(U_{j}, V_{j}, W_{j}, Q_{j})+(b^{2}_{j}g^{*}_{B}(db_{j}, db_{j})+2b^{3}_{j}P(b_{j})+b^{4}_{j}g_{B}(P, P))\nonumber\\
&\times(g_{F_{j}}(U_{j}, W_{j})g_{F_{j}}(V_{j}, Q_{j})-g_{F_{j}}(V_{j}, W_{j})g_{F_{j}}(U_{j}, Q_{j})),~if~i=j=k=s;\nonumber\\
&(10)\overline{R}(U_{k}, V_{i}, W_{j}, Q_{s})=\overline{R}(W_{j}, Q_{s}, U_{k}, V_{i},)=0,~if~i=j=s\neq k;\nonumber\\
&(11)\overline{R}(U_{k}, V_{i}, W_{j}, Q_{s})=(b_{i}b_{j}g^{*}_{B}(db_{i}, db_{j})+b_{i}b^{2}_{j}P(b_{i})+b_{j}b^{2}_{i}P(b_{j})+b^{2}_{i}b^{2}_{j}g_{B}(P, P))\nonumber\\
&\times g_{F_{i}}(V_{i}, Q_{i})g_{F_{j}}(U_{j}, W_{j}),~if~j=k\neq i=s;\nonumber\\
&(12)\overline{R}(U_{k}, V_{i}, W_{j}, Q_{s})=0,~other~cases.\nonumber
\end{align}
\end{thm}

\begin{proof}
\begin{align}(1)
\overline{R}(X, Y, Z, T)&=R(X, Y, Z, T)-g(Y, P)g(X, Z)g(P, T)+g(X, P)g(Y, Z)g(P, T)\nonumber\\
&+\overline{\mathcal{K}}(X, P, Z)g(Y, T)-\overline{\mathcal{K}}(Y, P, Z)g(X, T)\nonumber\\
&-\mathcal{K}(X, P, T)g(Y, Z)+\mathcal{K}(Y, P, T)g(X, Z)\nonumber\\
&=R_{B}(X, Y, Z, T)-g_{B}(Y, P)g_{B}(X, Z)g_{B}(P, T)+g_{B}(X, P)g_{B}(Y, Z)g_{B}(P, T)\nonumber\\
&+\overline{\mathcal{K}}_{B}(X, P, Z)g_{B}(Y, T)-\overline{\mathcal{K}}_{B}(Y, P, Z)g_{B}(X, T)\nonumber\\
&-\mathcal{K}_{B}(X, P, T)g_{B}(Y, Z)+\mathcal{K}_{B}(Y, P, T)g_{B}(X, Z)\nonumber\\
&=\overline{R}_{B}(X, Y, Z, T),\nonumber
\end{align}
where we applied (1) from Proposition 5.2 in \cite{Wang1}, (1) from Theorem 5.5 in \cite{Wang1} and (1) from Proposition 2.6.\par
(4) Follows from the Proposition 5.2 and Theorem 5.5 in \cite{Wang1}, we have
\begin{align}
&R(X, V_{j}, W_{j}, Y)=-b_{j}(XY(b_{j})-g^{*}_{B}(\nabla^{\flat}_{X}Y, db_{j}))g_{F_{j}}(V_{j}, W_{j}),\nonumber\\
&\mathcal{K}(X, Y, Z)=\mathcal{K}_{B}(X, Y, Z).\nonumber
\end{align}
A simple calculation shows that
\begin{align}
\overline{R}(X, V_{j}, W_{j}, Y)&=R(X, V_{j}, W_{j}, Y)-g(V_{j}, P)g(X, W_{j})g(P, Y)+g(X, P)g(V_{j}, W_{j})g(P, Y)\nonumber\\
&+\overline{\mathcal{K}}(X, P, W_{j})g(V_{j}, Y)-\overline{\mathcal{K}}(V_{j}, P, W_{j})g(X, Y)\nonumber\\
&-\mathcal{K}(X, P, Y)g(V_{j}, W_{j})+\mathcal{K}(V_{j}, P, Y)g(X, W_{j})\nonumber\\
&=R(X, V_{j}, W_{j}, Y)+g(X, P)g(V_{j}, W_{j})g(P, Y)-\overline{\mathcal{K}}(V_{j}, P, W_{j})g(X, Y)\nonumber\\
&-\mathcal{K}(X, P, Y)g(V_{j}, W_{j})\nonumber\\
&=-b_{j}XY(b_{j})g_{F_{j}}(V_{j}, W_{j})+b_{j}g^{*}_{B}(\nabla^{\flat}_{X}Y, db_{j})g_{F_{j}}(V_{j}, W_{j})\nonumber\\
&+b^{2}_{j}g_{B}(X, P)g_{B}(Y, P)g_{F_{j}}(V_{j}, W_{j})-b_{j}P(b_{j})g_{B}(X, Y)g_{F_{j}}(V_{j}, W_{j})\nonumber\\
&-b^{2}_{j}g_{B}(X, Y)g_{B}(P, P)g_{F_{j}}(V_{j}, W_{j})-b^{2}_{j}g_{F_{j}}(V_{j}, W_{j})\mathcal{K}_{B}(X, P, Y).\nonumber
\end{align}
(9) From (7) of the Theorem 5.5 in \cite{Wang1} we obtain that
\begin{align}
R(U_{j}, V_{j}, W_{j}, Q_{j})=&b^{2}_{j}R_{F_{j}}(U_{j}, V_{j}, W_{j}, Q_{j})+b^{2}_{j}g^{*}_{B}(db_{j}, db_{j})\nonumber\\
&\times(g_{F_{j}}(U_{j}, W_{j})g_{F_{j}}(V_{j}, Q_{j})-g_{F_{j}}(V_{j}, W_{j})g_{F_{j}}(U_{j}, Q_{j})).\nonumber
\end{align}
From (3) of the Proposition 5.2 in \cite{Wang1} we obtain that
\begin{align}
\mathcal{K}(X, Y, Z)= b_{j}X(b_{j})g_{F_{j}}(V_{j}, W_{j}).\nonumber
\end{align}
\begin{align}
\overline{R}(U_{j}, V_{j}, W_{j}, Q_{j})&=R(U_{j}, V_{j}, W_{j}, Q_{j})-g(V_{j}, P)g(U_{j}, W_{j})g(P, Q_{j})+g(U_{j}, P)g(V_{j}, W_{j})g(P, Q_{j})\nonumber\\
&+\overline{\mathcal{K}}(U_{j}, P, W_{j})g(V_{j}, Q_{j})-\overline{\mathcal{K}}(V_{j}, P, W_{j})g(U_{j}, Q_{j})\nonumber\\
&-\mathcal{K}(U_{j}, P, Q_{j})g(V_{j}, W_{j})+\mathcal{K}(V_{j}, P, Q_{j})g(U_{j}, W_{j})\nonumber\\
&=R(U_{j}, V_{j}, W_{j}, Q_{j})+\overline{\mathcal{K}}(U_{j}, P, W_{j})g(V_{j}, Q_{j})-\overline{\mathcal{K}}(V_{j}, P, W_{j})g(U_{j}, Q_{j})\nonumber\\
&-\mathcal{K}(U_{j}, P, Q_{j})g(V_{j}, W_{j})+\mathcal{K}(V_{j}, P, Q_{j})g(U_{j}, W_{j})\nonumber\\
&=b^{2}_{j}R_{F_{j}}(U_{j}, V_{j}, W_{j}, Q_{j})+(b^{2}_{j}g^{*}_{B}(db_{j}, db_{j})+2b^{3}_{j}P(b_{j})+b^{4}_{j}g_{B}(P, P))\nonumber\\
&\times(g_{F_{j}}(U_{j}, W_{j})g_{F_{j}}(V_{j}, Q_{j})-g_{F_{j}}(V_{j}, W_{j})g_{F_{j}}(U_{j}, Q_{j})).\nonumber
\end{align}
(11)By (9) from Theorem 5.5 in \cite{Wang1} and (3) from Proposition 5.2 in \cite{Wang1} and (4) from Proposition 2.15, we have
\begin{align}
\overline{R}(U_{j}, V_{i}, W_{j}, Q_{i})&=R(U_{j}, V_{i}, W_{j}, Q_{i})-g(V_{i}, P)g(U_{j}, W_{j})g(P, Q_{i})+g(U_{j}, P)g(V_{i}, W_{j})g(P, Q_{i})\nonumber\\
&+\overline{\mathcal{K}}(U_{j}, P, W_{j})g(V_{i}, Q_{i})-\overline{\mathcal{K}}(V_{i}, P, W_{j})g(U_{j}, Q_{i})\nonumber\\
&-\mathcal{K}(U_{j}, P, Q_{i})g(V_{i}, W_{j})+\mathcal{K}(V_{i}, P, Q_{i})g(U_{j}, W_{j})\nonumber\\
&=R(U_{j}, V_{i}, W_{j}, Q_{i})+\overline{\mathcal{K}}(U_{j}, P, W_{j})g(V_{i}, Q_{i})+\mathcal{K}(V_{i}, P, Q_{i})g(U_{j}, W_{j})\nonumber\\
&=(b_{i}b_{j}g^{*}_{B}(db_{i}, db_{j})+b_{i}b^{2}_{j}P(b_{i})+b_{j}b^{2}_{i}P(b_{j})+b^{2}_{i}b^{2}_{j}g_{B}(P, P))\nonumber\\
&\times g_{F_{i}}(V_{i}, Q_{i})g_{F_{j}}(U_{j}, W_{j}).\nonumber
\end{align}
\end{proof}

\begin{thm}
Let $(B, g_{B})$ be a non-degenerate semi-Riemannian manifold, $(F_{j}, g_{F_j})$ be semi-regular semi-Riemannian manifolds and $b_{j}\in C^{\infty}(B),~ 1\leq j\leq m.$
Let the vector fields $X, Y, Z, T\in\Gamma(TB)$ and $U_{j}, V_{j}, W_{j}, Q_{j}\in\Gamma(TF_{j}).$  Let $P\in\Gamma(TF_{l}),$ then
\begin{align}
&(1)\overline{R}(X, Y, Z, T)=R_{B}(X, Y, Z, T)+b^{2}_{l}g_{B}(X, Z)g_{B}(Y, T)g_{F_{l}}(P, P)-b^{2}_{l}g_{B}(X, T)\nonumber\\
&\times g_{B}(Y, Z)g_{F_{l}}(P, P);\nonumber\\
&(2)\overline{R}(X, Y, Z, W_{j})=-\overline{R}(Z, W_{j}, X, Y)=-b_{j}X(b_{j})g_{B}(Y, Z)g_{F_{j}}(W_{j}, P)+b_{j}Y(b_{j})\nonumber\\
&\times g_{B}(X, Z)g_{F_{j}}(W_{j}, P),~if~j=l;\nonumber\\
&(3)\overline{R}(X, Y, Z, W_{j})=\overline{R}(Z, W_{j}, X, Y)=0,~if~j\neq l;\nonumber\\
&(4)\overline{R}(X, Y,  V_{i}, W_{j})=\overline{R}(V_{i}, W_{j}, X, Y)=0;\nonumber\\
&(5)\overline{R}(X, V_{i}, W_{j}, Y)=b_{j}g^{*}_{B}(\nabla^{\flat}_{X}Y, db_{j})g_{F_{j}}(V_{j}, W_{j})-b_{j}XY(b_{j})g_{F_{j}}(V_{j}, W_{j})-b^{2}_{j}g_{B}(X, Y)\nonumber\\
&\times(\mathcal{K}_{F_{j}}(V_{j}, P, W_{j})+b^{2}_{j}g_{F_{j}}(V_{j}, W_{j})g_{F_{j}}(P, P)-b^{2}_{j}g_{F_{j}}(V_{j}, P)g_{F_{j}}(W_{j}, P)),~if~i=j=l;\nonumber
\end{align}
\begin{align}
&(6)\overline{R}(X, V_{i}, W_{j}, Y)=b_{j}g^{*}_{B}(\nabla^{\flat}_{X}Y, db_{j})g_{F_{j}}(V_{j}, W_{j})-b_{j}XY(b_{j})g_{F_{j}}(V_{j}, W_{j})\nonumber\\
&-b^{2}_{j}b^{2}_{l}g_{B}(X, Y)g_{F_{j}}(V_{j}, W_{j})g_{F_{j}}(P, P),~if~i=j\neq l;\nonumber\\
&(7)\overline{R}(X, V_{i}, W_{j}, Y)=0,~other~cases;\nonumber\\
&(8)\overline{R}(X, V_{i}, W_{j}, U_{k})=-\overline{R}(W_{j}, U_{k}, X, V_{i})=b^{3}_{j}X(b_{j})(g_{F_{j}}(V_{j}, U_{j})g_{F_{j}}(W_{j}, P)-g_{F_{j}}(V_{j}, W_{j})\nonumber\\
&\times g_{F_{j}}(U_{j}, P)),~if~i=j=k=l;\nonumber\\
&(9)\overline{R}(X, V_{i}, W_{j}, U_{k})=-\overline{R}(W_{j}, U_{k}, X, V_{i})=-b_{l}b^{2}_{j}X(b_{l})g_{F_{l}}(U_{l}, P)g_{F_{j}}(V_{j}, W_{j}),~if~i=j\neq k=l;\nonumber\\
&(10)\overline{R}(X, V_{i}, W_{j}, U_{k})=0,~other~cases;\nonumber\\
&(11)\overline{R}(U_{k}, V_{i}, W_{j}, Q_{s})=b^{2}_{j}R_{F_{j}}(U_{j}, V_{j}, W_{j}, Q_{j})+(b^{2}_{j}g^{*}_{B}(db_{j}, db_{j})+b^{6}_{j}g_{F_{j}}(P, P))\nonumber\\
&\times(g_{F_{j}}(U_{j}, W_{j})g_{F_{j}}(V_{j}, Q_{j})-g_{F_{j}}(V_{j}, W_{j})g_{F_{j}}(U_{j}, Q_{j}))+b^{4}_{j}g_{F_{j}}(V_{j}, Q_{j})\mathcal{K}_{F_{j}}(U_{j}, P, W_{j})\nonumber\\
&-b^{4}_{j}g_{F_{j}}(U_{j}, Q_{j})\mathcal{K}_{F_{j}}(V_{j}, P, W_{j})-b^{4}_{j}g_{F_{j}}(V_{j}, W_{j})\mathcal{K}_{F_{j}}(U_{j}, P, Q_{j})+b^{4}_{j}g_{F_{j}}(U_{j}, W_{j})\mathcal{K}_{F_{j}}(V_{j}, P, Q_{j})\nonumber\\
&+b^{6}_{j}g_{F_{j}}(U_{j}, P)(g_{F_{j}}(V_{j}, W_{j})g_{F_{j}}(Q_{j}, P)-g_{F_{j}}(V_{j}, Q_{j})g_{F_{j}}(W_{j}, P))-b^{6}_{j}g_{F_{j}}(V_{j}, P)(g_{F_{j}}(U_{j}, W_{j})\nonumber\\
&\times g_{F_{j}}(Q_{j}, P)-g_{F_{j}}(U_{j}, Q_{j})g_{F_{j}}(W_{j}, P)),~if~i=j=k=s=l;\nonumber\\
&(12)\overline{R}(U_{k}, V_{i}, W_{j}, Q_{s})=b^{2}_{j}R_{F_{j}}(U_{j}, V_{j}, W_{j}, Q_{j})+(b^{2}_{j}g^{*}_{B}(db_{j}, db_{j})+b^{2}_{l}b^{4}_{j}g_{F_{l}}(P, P))\nonumber\\
&\times(g_{F_{j}}(U_{j}, W_{j})g_{F_{j}}(V_{j}, Q_{j})-g_{F_{j}}(V_{j}, W_{j})g_{F_{j}}(U_{j}, Q_{j}))~if~i=j=k=s\neq l;\nonumber\\
&(13)\overline{R}(U_{k}, V_{i}, W_{j}, Q_{s})=b_{i}b_{j}g^{*}_{B}(db_{i}, db_{j})g_{F_{i}}(V_{i}, Q_{i})g_{F_{j}}(U_{j}, W_{j})+b^{2}_{i}b^{2}_{j}g_{F_{j}}(U_{j}, W_{j})(\mathcal{K}_{F_{i}}(V_{i}, P, Q_{i})\nonumber\\
&+b^{2}_{i}g_{F_{i}}(P, P)g_{F_{i}}(V_{i}, Q_{i})-b^{2}_{i}g_{F_{i}}(V_{i}, P)g_{F_{i}}(Q_{i}, P)),~if~j=k\neq i=s=l;\nonumber\\
&(14)\overline{R}(U_{k}, V_{i}, W_{j}, Q_{s})=b_{i}b_{j}g^{*}_{B}(db_{i}, db_{j})g_{F_{i}}(V_{i}, Q_{i})g_{F_{j}}(U_{j}, W_{j})+b^{2}_{i}b^{2}_{j}g_{F_{i}}(V_{i}, Q_{i})(\mathcal{K}_{F_{j}}(U_{j}, P, W_{j})\nonumber\\
&+b^{2}_{j}g_{F_{j}}(P, P)g_{F_{j}}(U_{j}, W_{j})-b^{2}_{j}g_{F_{j}}(U_{j}, P)g_{F_{j}}(W_{j}, P)),~if~l=j=k\neq i=s;\nonumber\\
&(15)\overline{R}(U_{k}, V_{i}, W_{j}, Q_{s})=b_{i}b_{j}g^{*}_{B}(db_{i}, db_{j})g_{F_{i}}(V_{i}, Q_{i})g_{F_{j}}(U_{j}, W_{j})+b^{2}_{i}b^{2}_{j}b^{2}_{l}g_{F_{i}}(V_{i}, Q_{i})\nonumber\\
&\times g_{F_{j}}(U_{j}, W_{j})g_{F_{l}}(P, P),~where~k=j, i=s, l~are~different;\nonumber\\
&(16)\overline{R}(U_{k}, V_{i}, W_{j}, Q_{s})=0,~other~cases.\nonumber
\end{align}
\end{thm}

\begin{proof}
The proof is similar to Theorem 2.11, so that we omit it.
\end{proof}

\vskip 1 true cm
\section{ Semi-Symmetric Non-Metric Koszul Forms and Their Curvature of Singular Multiply Warped Products}
We now prepare some basic facts of semi-symmetric non-metric Koszul forms.

\begin{defn}{\rm\cite{Wang1} }
Semi-symmetric non-metric Koszul forms $\widehat{\mathcal{K}_{p}}: \Gamma(TM)^{3}\rightarrow C^{\infty}(M)$ on $(M, g)$ is defined as
\begin{equation}
\widehat{\mathcal{K}_{p}}(X, Y, Z):=\mathcal{K}(X, Y, Z)+g(Y, P)g(X, Z).
\end{equation}
Usually, we write $\widehat{\mathcal{K}}$ instead of $\widehat{\mathcal{K}_{p}}.$
\end{defn}

\begin{defn}{\rm\cite{Wang1} }
Let $X, Y\in\Gamma(TM).$ The semi-symmetric non-metric lower covariant derivative of $Y$ in the direction of $X$ as the differential $1$-form $\widehat{\nabla}^{\flat}_{X}Y\in A^{1}(M)$
\begin{equation}
(\widehat{\nabla}^{\flat}_{X}Y)(Z):=\widehat{\mathcal{K}}(X, Y, Z),
\end{equation}
for any $Z\in\Gamma(TM).$
\end{defn}

So we can claim the following propositions.
\begin{prop}
Let $B\times_{b_{1}}F_{1}\times\cdots\times_{b_{m}}F_{m}$ be a degenerate multiply warped product and let the vector fields $X, Y, Z\in\Gamma(TB)$ and $U_{j}, V_{j}, W_{j}\in\Gamma(TF_{j}).$ Let $\widehat{\mathcal{K}}$ be the semi-symmetric non-metric Koszul form on $B\times_{b_{1}}F_{1}\times\cdots\times_{b_{m}}F_{m}$ and $\widehat{\mathcal{K}}_{B}$, $\widehat{\mathcal{K}}_{F_{j}}$ be the lifts of the semi-symmetric non-metric Koszul form on $B,$ $F_{j},$ respectively. Let $P\in\Gamma(TB),$ then
\begin{align}
&(1)\widehat{\mathcal{K}}(X, Y, Z)=\widehat{\mathcal{K}}_{B}(X, Y, Z);\nonumber\\
&(2)\widehat{\mathcal{K}}(X, Y, W_{j})=\widehat{\mathcal{K}}(X, W_{j}, Y)=\widehat{\mathcal{K}}(W_{j}, X, Y)=0;\nonumber\\
&(3)\widehat{\mathcal{K}}(X, V_{i}, W_{j})=-\widehat{\mathcal{K}}(V_{i}, W_{j}, X)=b_{j}X(b_{j})g_{F_{j}}(V_{j}, W_{j}),~if~i=j;\nonumber\\
&(4)\widehat{\mathcal{K}}(V_{i}, X, W_{j})=b_{j}X(b_{j})g_{F_{j}}(V_{j}, W_{j})+b^{2}_{j}g_{B}(X, P)g_{F_{j}}(V_{j}, W_{j}),~if~i=j;\nonumber\\
&(5)\widehat{\mathcal{K}}(X, V_{i}, W_{j})=\widehat{\mathcal{K}}(V_{i}, X, W_{j})=\widehat{\mathcal{K}}(V_{i}, W_{j}, X)=0,~if~i\neq j;\nonumber\\
&(6)\widehat{\mathcal{K}}(U_{i}, V_{j}, W_{k})=b^{2}_{j}\mathcal{K}_{F_{j}}(U_{j}, V_{j}, W_{j}),~if~i=j=k;\nonumber\\
&(7)\widehat{\mathcal{K}}(U_{i}, V_{j}, W_{k})=0,~other~cases.\nonumber
\end{align}
\end{prop}

\begin{prop}
Let $B\times_{b_{1}}F_{1}\times\cdots\times_{b_{m}}F_{m}$ be a degenerate multiply warped product and let the vector fields $X, Y, Z\in\Gamma(TB)$ and $U_{j}, V_{j}, W_{j}\in\Gamma(TF_{j}).$ Let $\widehat{\mathcal{K}}$ be the semi-symmetric non-metric Koszul form on $B\times_{b_{1}}F_{1}\times\cdots\times_{b_{m}}F_{m}$ and $\widehat{\mathcal{K}}_{B}$, $\widehat{\mathcal{K}}_{F_{j}}$ be the lifts of the semi-symmetric non-metric Koszul form on $B,$ $F_{j},$ respectively. Let $P\in\Gamma(TF_{l}),$ then
\begin{align}
&(1)\widehat{\mathcal{K}}(X, Y, Z)=\mathcal{K}_{B}(X, Y, Z);\nonumber\\
&(2)\widehat{\mathcal{K}}(X, Y, W_{j})=\widehat{\mathcal{K}}(W_{j}, X, Y)=0,~if~j=l;\nonumber\\
&(3)\widehat{\mathcal{K}}(X, W_{j}, Y)=b^{2}_{j}g_{B}(X, Y)g_{F_{j}}(W_{j}, P),~if~j=l;\nonumber\\
&(4)\widehat{\mathcal{K}}(X, Y, W_{j})=\widehat{\mathcal{K}}(X, W_{j}, Y)=\widehat{\mathcal{K}}(W_{j}, X, Y)=0,~if~j\neq l;\nonumber\\
&(5)\widehat{\mathcal{K}}(X, V_{i}, W_{j})=\widehat{\mathcal{K}}(V_{i}, X, W_{j})=-\widehat{\mathcal{K}}(V_{i}, W_{j}, X)=b_{j}X(b_{j})g_{F_{j}}(V_{j}, W_{j}),~if~i=j;\nonumber\\
&(6)\widehat{\mathcal{K}}(X, V_{i}, W_{j})=\widehat{\mathcal{K}}(V_{i}, X, W_{j})=\widehat{\mathcal{K}}(V_{i}, W_{j}, X)=0,~if~i\neq j;\nonumber\\
&(7)\widehat{\mathcal{K}}(U_{i}, V_{j}, W_{k})=b^{2}_{j}\mathcal{K}_{F_{j}}(U_{j}, V_{j}, W_{j})+b^{4}_{j}g_{F_{j}}(U_{j}, W_{j})g_{F_{j}}(V_{j}, P),~if~i=j=k=l;\nonumber\\
&(8)\widehat{\mathcal{K}}(U_{i}, V_{j}, W_{k})=b^{2}_{j}\mathcal{K}_{F_{j}}(U_{j}, V_{j}, W_{j}),~if~i=j= k\neq l;\nonumber\\
&(9)\widehat{\mathcal{K}}(U_{i}, W_{k}, V_{j})=b^{2}_{j}b^{2}_{k}g_{F_{j}}(U_{j}, V_{j})g_{F_{k}}(W_{k}, P),~if~i=j\neq k=l;\nonumber\\
&(10)\widehat{\mathcal{K}}(U_{i}, V_{j}, W_{k})=0,~other~cases.\nonumber
\end{align}
\end{prop}
On a semi-symmetric non-metric semi-regular semi-Riemannian manifold, we define the Riemann curvature tensor of the semi-symmetric non-metric covariant derivative
\begin{equation}
\widehat{R}(X, Y, Z, T):=(\widehat{\nabla}_{X}\widehat{\nabla}^{\flat}_{Y}Z)(T)-(\widehat{\nabla}_{Y}\widehat{\nabla}^{\flat}_{X}Z)(T)-(\widehat{\nabla}^{\flat}_{[X, Y]}Z)(T).
\end{equation}

By some computations, we can get
\begin{prop}{\rm\cite{Wang1} }
For any vector fields $X, Y, Z, T\in\Gamma(TM)$ on a semi-symmetric non-metric semi-regular semi-Riemannian manifold $(M, g)$
\begin{align}
\widehat{R}(X, Y, Z, T)&=R(X, Y, Z, T)+X(g(Z, P))g(Y, T)-Y(g(Z, P))g(X, T)\\
&-\widehat{\mathcal{K}}(Y, Z, X)g(P, T)+\widehat{\mathcal{K}}(X, Z, Y)g(P, T).\nonumber
\end{align}
\end{prop}

It is easy to check that
\begin{align}
\widehat{R}(X, Y, Z, T)=-\widehat{R}(Y, X, Z, T),~~~~~~~ \widehat{R}(X, Y, Z, T)\neq-\widehat{R}(X, Y, T, Z).
\end{align}

We review the following proposition in \cite{Wang1}.
\begin{prop}
A singular semi-Riemannian manifold $(M, g)$ is a semi-regular semi-Riemannian manifold if and only if it is a semi-symmetric non-metric semi-regular semi-Riemannian manifold.
\end{prop}

Then we can conclude the following facts.
\begin{thm}
Let $(B, g_{B})$ be a non-degenerate semi-Riemannian manifold, $(F_{j}, g_{F_j})$ be semi-regular semi-Riemannian manifolds and $b_{j}\in C^{\infty}(B),~ 1\leq j\leq m.$
Let the vector fields $X, Y, Z, T\in\Gamma(TB)$ and $U_{j}, V_{j}, W_{j}, Q_{j}\in\Gamma(TF_{j}).$  Let $P\in\Gamma(TB),$ then
\begin{align}
&(1)\widehat{R}(X, Y, Z, T)=\widehat{R}_{B}(X, Y, Z, T);\nonumber\\
&(2)\widehat{R}(X, Y, Z, W_{j})=\widehat{R}(X, Y, W_{j}, Z)=\widehat{R}(Z, W_{j}, X, Y)=0;\nonumber\\
&(3)\widehat{R}(X, Y,  V_{i}, W_{j})=\widehat{R}(V_{i}, W_{j}, X, Y)=0,~if~i=j;\nonumber\\
&(4)\widehat{R}(V_{i}, X, W_{j}, Y)=-\widehat{R}(X, V_{i}, W_{j}, Y)=b_{j}XY(b_{j})g_{F_{j}}(V_{j}, W_{j})-b_{j}g^{*}_{B}(\nabla^{\flat}_{X}Y, db_{j})g_{F_{j}}(V_{j}, W_{j})\nonumber\\
&-2b_{j}X(b_{j})g_{B}(Y, P)g_{F_{j}}(V_{j}, W_{j}),~if~i=j;\nonumber\\
&(5)\widehat{R}(V_{i}, X, Y, W_{j})=-\widehat{R}(X, V_{i}, Y, W_{j})=-b_{j}XY(b_{j})g_{F_{j}}(V_{j}, W_{j})+b_{j}g^{*}_{B}(\nabla^{\flat}_{X}Y, db_{j})g_{F_{j}}(V_{j}, W_{j})\nonumber\\
&-b^{2}_{j}X(g_{B}(Y, P))g_{F_{j}}(V_{j}, W_{j}),~if~i=j;\nonumber\\
&(6)\widehat{R}(X, Y,  V_{i}, W_{j})=\widehat{R}(V_{i}, W_{j}, X, Y)=\widehat{R}(X, V_{i}, W_{j}, Y)=\widehat{R}(X, V_{i}, Y, W_{j})=0,~if~i\neq j;\nonumber\\
&(7)\widehat{R}(X, V_{i}, W_{j}, U_{k})=\widehat{R}(W_{j}, U_{k}, X, V_{i})=0,~if~i=j=k;\nonumber\\
&(8)\widehat{R}(W_{j}, U_{k}, V_{i}, X)=b^{2}_{j}g_{B}(X, P)(\mathcal{K}_{F_{j}}(W_{j}, V_{j}, U_{j})-\mathcal{K}_{F_{j}}(U_{j}, V_{j}, W_{j})),~if~i=j=k;\nonumber\\
&(9)\widehat{R}(X, V_{i}, W_{j}, U_{k})=\widehat{R}(W_{j}, U_{k}, X, V_{i})=\widehat{R}(W_{j}, U_{k}, V_{i}, X)=0,~other~cases;\nonumber
\end{align}
\begin{align}
&(10)\widehat{R}(U_{k}, V_{i}, W_{j}, Q_{s})=b^{2}_{j}R_{F_{j}}(U_{j}, V_{j}, W_{j}, Q_{j})+b^{2}_{j}g^{*}_{B}(db_{j}, db_{j})(g_{F_{j}}(U_{j}, W_{j})g_{F_{j}}(V_{j}, Q_{j})\nonumber\\
&-g_{F_{j}}(V_{j}, W_{j})g_{F_{j}}(U_{j}, Q_{j})),~if~i=j=k=s;\nonumber\\
&(11)\widehat{R}(U_{k}, V_{i}, W_{j}, Q_{s})=b_{i}b_{j}g^{*}_{B}(db_{i}, db_{j})g_{F_{i}}(V_{i}, Q_{i})g_{F_{j}}(U_{j}, W_{j}),~if~k=j\neq i=s;\nonumber\\
&(12)\widehat{R}(U_{k}, V_{i}, W_{j}, Q_{s})=-b_{j}b_{k}g^{*}_{B}(db_{j}, db_{k})g_{F_{j}}(V_{j}, W_{j})g_{F_{k}}(U_{k}, Q_{k}),~if~k=s\neq j=i;\nonumber\\
&(13)\widehat{R}(U_{k}, V_{i}, W_{j}, Q_{s})=0,~other~cases.\nonumber
\end{align}
\end{thm}

\begin{thm}
Let $(B, g_{B})$ be a non-degenerate semi-Riemannian manifold, $(F_{j}, g_{F_j})$ be semi-regular semi-Riemannian manifolds and $b_{j}\in C^{\infty}(B),~ 1\leq j\leq m.$
Let the vector fields $X, Y, Z, T\in\Gamma(TB)$ and $U_{j}, V_{j}, W_{j}, Q_{j}\in\Gamma(TF_{j}).$  Let $P\in\Gamma(TF_{l}),$ then
\begin{align}
&(1)\widehat{R}(X, Y, Z, T)=R_{B}(X, Y, Z, T);\nonumber\\
&(2)\widehat{R}(X, Y, Z, W_{j})=b^{2}_{j}g_{F_{j}}(W_{j}, P)(\mathcal{K}_{B}(X, Z, Y)-\mathcal{K}_{B}(Y, Z, X)),~if~j=l;\nonumber\\
&(3)\widehat{R}(X, Y, W_{j}, Z)=2b_{j}g_{F_{j}}(W_{j}, P)(X(b_{j})g_{B}(Y, Z)-Y(b_{j})g_{B}(X, Z)),~if~j=l;\nonumber\\
&(4)\widehat{R}(X, W_{j}, Y, Z)=\widehat{R}(W_{j}, X, Y, Z)=0,~if~j=l;\nonumber\\
&(5)\widehat{R}(X, Y, Z, W_{j})=\widehat{R}(X, Y, W_{j}, Z)=\widehat{R}(X, W_{j}, Y, Z)=\widehat{R}(W_{j}, X, Y, Z)=0,~if~j\neq l;\nonumber\\
&(6)\widehat{R}(X, Y,  V_{i}, W_{j})=\widehat{R}(V_{i}, W_{j}, X, Y)=0;\nonumber\\
&(7)\widehat{R}(V_{i}, X, W_{j}, Y)=-\widehat{R}(X, V_{i}, W_{j}, Y)=b_{j}XY(b_{j})g_{F_{j}}(V_{j}, W_{j})-b_{j}g^{*}_{B}(\nabla^{\flat}_{X}Y, db_{j})g_{F_{j}}(V_{j}, W_{j})\nonumber\\
&+b^{2}_{j}V_{j}(g_{F_{j}}(W_{j}, P))g_{B}(X, Y),~if~i=j=l;\nonumber\\
&(8)\widehat{R}(V_{i}, X, W_{j}, Y)=-\widehat{R}(X, V_{i}, W_{j}, Y)=b_{j}XY(b_{j})g_{F_{j}}(V_{j}, W_{j})-b_{j}g^{*}_{B}(\nabla^{\flat}_{X}Y, db_{j})g_{F_{j}}(V_{j}, W_{j}),\nonumber\\
&~if~i=j\neq l;\nonumber\\
&(9)\widehat{R}(V_{i}, X, Y, W_{j})=-\widehat{R}(X, V_{i}, Y, W_{j})=-b_{j}XY(b_{j})g_{F_{j}}(V_{j}, W_{j})+b_{j}g^{*}_{B}(\nabla^{\flat}_{X}Y, db_{j})g_{F_{j}}(V_{j}, W_{j}),\nonumber\\
&~if~i=j=l~or~i=j\neq l;\nonumber\\
&(10)\widehat{R}(V_{i}, X, W_{j}, Y)=\widehat{R}(V_{i}, X, Y, W_{j})=0,~other~cases;\nonumber\\
&(11)\widehat{R}(X, V_{i}, W_{j}, U_{k})=2b^{3}_{j}X(b_{j})(g_{F_{j}}(V_{j}, U_{j})g_{F_{j}}(W_{j}, P)+g_{F_{j}}(V_{j}, W_{j})g_{F_{j}}(U_{j}, P)),\nonumber\\
&~if~i=j=k=l;\nonumber\\
&(12)\widehat{R}(X, V_{i}, W_{j}, U_{k})=2b_{j}b^{2}_{k}X(b_{j})g_{F_{j}}(V_{j}, W_{j})g_{F_{k}}(U_{k}, P),~if~i=j\neq k=l;\nonumber\\
&(13)\widehat{R}(X, V_{i}, U_{k}, W_{j})=2b_{k}b^{2}_{j}X(b_{k})g_{F_{j}}(V_{j}, W_{j})g_{F_{k}}(U_{k}, P),~if~i=j\neq k=l;\nonumber\\
&(14)\widehat{R}(X, V_{i}, W_{j}, U_{k})=\widehat{R}(W_{j}, U_{k}, X, V_{i})=\widehat{R}(W_{j}, U_{k}, V_{i}, X)=0,~other~cases;\nonumber\\
&(15)\widehat{R}(U_{k}, V_{i}, W_{j}, Q_{s})=b^{2}_{j}R_{F_{j}}(U_{j}, V_{j}, W_{j}, Q_{j})+b^{2}_{j}g^{*}_{B}(db_{j}, db_{j})(g_{F_{j}}(U_{j}, W_{j})g_{F_{j}}(V_{j}, Q_{j})\nonumber\\
&-g_{F_{j}}(V_{j}, W_{j})g_{F_{j}}(U_{j}, Q_{j}))+b^{4}_{j}g_{F_{j}}(Q_{j}, P)(\mathcal{K}_{F_{j}}(U_{j}, W_{j}, V_{j})-\mathcal{K}_{F_{j}}(V_{j}, W_{j}, U_{j}))\nonumber\\
&+b^{4}_{j}U_{j}(g_{F_{j}}(W_{j}, P))g_{F_{j}}(V_{j}, Q_{j})-b^{4}_{j}V_{j}(g_{F_{j}}(W_{j}, P))g_{F_{j}}(U_{j}, Q_{j}),~if~i=j=k=s=l;\nonumber
\end{align}
\begin{align}
&(16)\widehat{R}(U_{k}, V_{i}, W_{j}, Q_{s})=b^{2}_{j}R_{F_{j}}(U_{j}, V_{j}, W_{j}, Q_{j})+b^{2}_{j}g^{*}_{B}(db_{j}, db_{j})(g_{F_{j}}(U_{j}, W_{j})g_{F_{j}}(V_{j}, Q_{j})\nonumber\\
&-g_{F_{j}}(V_{j}, W_{j})g_{F_{j}}(U_{j}, Q_{j})),~if~i=j=k=s\neq l;\nonumber\\
&(17)\widehat{R}(U_{k}, V_{i}, W_{j}, Q_{s})=b_{i}b_{j}g^{*}_{B}(db_{i}, db_{j})g_{F_{i}}(V_{i}, Q_{i})g_{F_{j}}(U_{j}, W_{j}),~if~j=k\neq i=s=l\nonumber\\
&~or~k=j, i=s, l~are~different;\nonumber\\
&(18)\widehat{R}(U_{k}, V_{i}, W_{j}, Q_{s})=b_{i}b_{j}g^{*}_{B}(db_{i}, db_{j})g_{F_{i}}(V_{i}, Q_{i})g_{F_{j}}(U_{j}, W_{j})+b^{2}_{i}b^{2}_{j}U_{j}(g_{F_{j}}(W_{j}, P))\nonumber\\
&\times g_{F_{i}}(V_{i}, Q_{i}),~if~l=j=k\neq i=s;\nonumber\\
&(19)\widehat{R}(U_{k}, V_{i}, W_{j}, Q_{s})=-b_{j}b_{k}g^{*}_{B}(db_{j}, db_{k})g_{F_{j}}(V_{j}, W_{j})g_{F_{k}}(U_{k}, Q_{k})-b^{2}_{j}b^{2}_{k}V_{j}(g_{F_{j}}(W_{j}, P))\nonumber\\
&\times g_{F_{k}}(U_{k}, Q_{k}),~if~k=s\neq i=j=l;\nonumber\\
&(20)\widehat{R}(U_{k}, V_{i}, W_{j}, Q_{s})=-b_{j}b_{k}g^{*}_{B}(db_{j}, db_{k})g_{F_{j}}(V_{j}, W_{j})g_{F_{k}}(U_{k}, Q_{k}),~if~l=k=s\neq i=j\nonumber\\
&~or~k=s, i=j, l~are~different;\nonumber\\
&(21)\widehat{R}(V_{i}, W_{j}, Q_{s}, U_{k})=b^{2}_{j}b^{2}_{k}g_{F_{k}}(U_{k}, P)(\mathcal{K}_{F_{j}}(V_{j}, Q_{j}, W_{j})-\mathcal{K}_{F_{j}}(W_{j}, Q_{j}, V_{j})),\nonumber\\
&~if~l=k\neq i=j=s;\nonumber\\
&(22)\widehat{R}(U_{k}, V_{i}, W_{j}, Q_{s})=0,~other~cases.\nonumber
\end{align}
\end{thm}

\vskip 1 true cm
\section{ Koszul Forms Associated with the Almost Product Structure and Their Curvature of Singular Multiply Warped Products}
We review some of the standard facts on almost product Koszul forms, more details are found in the book \cite{Wang1}.

\begin{defn}{\rm\cite{Wang1} }
A singular semi-Riemannian manifold $(M, g)$ is called a almost product singular semi-Riemannian manifold if there exists a almost product structure $J : TM \rightarrow TM$ satisfying $J^{2}=\texttt{id}$ and $g(JX, JY)=g(X, Y)$ for any $X, Y\in \Gamma(TM).$
\end{defn}

\begin{defn}{\rm\cite{Wang1} }
Almost product Koszul forms $\widetilde{\mathcal{K}}: \Gamma(TM)^{3}\rightarrow C^{\infty}(M)$ on $(M, g, J)$ is defined as
\begin{equation}
\widetilde{\mathcal{K}}(X, Y, Z):=\frac{1}{2}[\mathcal{K}(X, Y, Z)+\mathcal{K}(X, JY, JZ)].
\end{equation}
\end{defn}

\begin{defn}{\rm\cite{Wang1} }
Let $X, Y\in\Gamma(TM).$ The almost product lower covariant derivative of $Y$ in the direction of $X$ as the differential $1$-form $\widetilde{\nabla}^{\flat}_{X}Y\in A^{1}(M)$
\begin{equation}
(\widetilde{\nabla}^{\flat}_{X}Y)(Z):=\widetilde{\mathcal{K}}(X, Y, Z),
\end{equation}
for any $Z\in\Gamma(TM).$
\end{defn}

We let $J = (J_{B}, J_{F_{j}})$ be the almost product structure on $B\times_{b_{1}}F_{1}\times\cdots\times_{b_{m}}F_{m},$ $J_{B}: TB\rightarrow TB,$ $J_{F_{j}}: TF_{j}\rightarrow TF_{j},$ $J^{2}_{B}=\texttt{id}_{TB}$ and $J^{2}_{F_{j}}=\texttt{id}_{TF_{j}}.$ One sees immediately that

\begin{prop}
Let $B\times_{b_{1}}F_{1}\times\cdots\times_{b_{m}}F_{m}$ be a degenerate multiply warped product and let the vector fields $X, Y, Z\in\Gamma(TB)$ and $U_{j}, V_{j}, W_{j}\in\Gamma(TF_{j}).$ Let $\widetilde{\mathcal{K}}$ be the almost product Koszul form on $B\times_{b_{1}}F_{1}\times\cdots\times_{b_{m}}F_{m}$ and $\widetilde{\mathcal{K}}_{B}$, $\widetilde{\mathcal{K}}_{F_{j}}$ be the lifts of the almost product Koszul form on $B,$ $F_{j},$ respectively. Then
\begin{align}
&(1)\widetilde{\mathcal{K}}(X, Y, Z)=\widetilde{\mathcal{K}}_{B}(X, Y, Z);\nonumber\\
&(2)\widetilde{\mathcal{K}}(X, Y, W_{j})=\widetilde{\mathcal{K}}(X, W_{j}, Y)=\widetilde{\mathcal{K}}(W_{j}, X, Y)=0;\nonumber\\
&(3)\widetilde{\mathcal{K}}(X, V_{i}, W_{j})=b_{j}X(b_{j})g_{F_{j}}(V_{j}, W_{j}),~if~i=j;\nonumber\\
&(4)\widetilde{\mathcal{K}}(V_{i}, X, W_{j})=-\widetilde{\mathcal{K}}(V_{i}, W_{j}, X)=\frac{1}{2}(b_{j}X(b_{j})g_{F_{j}}(V_{j}, W_{j})+b_{j}J_{B}X(b_{j})g_{F_{j}}(V_{j}, J_{F_{j}}W_{j})),\nonumber\\
&~if~i=j;\nonumber\\
&(5)\widetilde{\mathcal{K}}(X, V_{i}, W_{j})=\widetilde{\mathcal{K}}(V_{i}, X, W_{j})=\widetilde{\mathcal{K}}(V_{i}, W_{j}, X)=0,~if~i\neq j;\nonumber\\
&(6)\widetilde{\mathcal{K}}(U_{i}, V_{j}, W_{k})=b^{2}_{j}\widetilde{\mathcal{K}}_{F_{j}}(U_{j}, V_{j}, W_{j}),~if~i=j=k;\nonumber\\
&(7)\widetilde{\mathcal{K}}(U_{i}, V_{j}, W_{k})=\widetilde{\mathcal{K}}(U_{i}, W_{k}, V_{j})=\widetilde{\mathcal{K}}(W_{k}, U_{i}, V_{j})=0,~if~i=j\neq k;\nonumber\\
&(8)\widetilde{\mathcal{K}}(U_{i}, V_{j}, W_{k})=0,~where~i, j, k~are~different.\nonumber
\end{align}
\end{prop}

On a almost product semi-regular semi-Riemannian manifold, we define the Riemann curvature tensor of the almost product covariant derivative
\begin{equation}
\widetilde{R}(X, Y, Z, T):=(\widetilde{\nabla}_{X}\widetilde{\nabla}^{\flat}_{Y}Z)(T)-(\widetilde{\nabla}_{Y}\widetilde{\nabla}^{\flat}_{X}Z)(T)-(\widetilde{\nabla}^{\flat}_{[X, Y]}Z)(T).
\end{equation}

Similarly to Proposition 8.1 in \cite{OC1}, we have the following.
\begin{prop}{\rm\cite{Wang1} }
For any vector fields $X, Y, Z, T\in\Gamma(TM)$ on a almost product semi-regular semi-Riemannian manifold $(M, g, J)$
\begin{align}
\widetilde{R}(X, Y, Z, T)&=X(\widetilde{\mathcal{K}}(Y, Z, T))-Y(\widetilde{\mathcal{K}}(X, Z, T))-\widetilde{\mathcal{K}}([X, Y], Z, T)\\
&+\widetilde{\mathcal{K}}(X, Z, \bullet)\widetilde{\mathcal{K}}(Y, T, \bullet)-\widetilde{\mathcal{K}}(Y, Z, \bullet)\widetilde{\mathcal{K}}(X, T, \bullet).\nonumber
\end{align}
\end{prop}

From this
\begin{align}
\widetilde{R}(X, Y, Z, T)=-\widetilde{R}(Y, X, Z, T),~~~~~~~ \widetilde{R}(X, Y, Z, T)=-\widetilde{R}(X, Y, T, Z).
\end{align}

\begin{thm}
Let $(B, g_{B}, J_{B})$ and $(F_{j}, g_{F_{j}}, J_{F_{j}})$ be two almost product radical-stationary semi-Riemannian manifolds and $b_{j}\in C^{\infty}(B),~ 1\leq j\leq m,$ so that $db_{j}\in A^{\bullet}(B).$ Then the  multiply warped product manifold $(B\times_{b_{1}}F_{1}\times\cdots\times_{b_{m}}F_{m}, J)$ is a almost product radical-stationary semi-Riemannian manifold.
\end{thm}

\begin{proof}
By applying the Definition 3.1.3 shown in \cite{DK}, we only need to show that $\widetilde{\mathcal{K}}(X, Y, W_{p})\\=0$ for any $X, Y\in \Gamma(B\times_{b_{1}}F_{1}\times\cdots\times_{b_{m}}F_{m})$ and $W_{p}\in \Gamma_{0}(B\times_{b_{1}}F_{1}\times\cdots\times_{b_{m}}F_{m}).$
Since $db_{j}\in A^{\bullet}(B),$ for $X\in \Gamma_{0}(B),$ we see that $db_{j}(J_{B}X)=J_{B}X(db_{j})=0.$ The proof of Theorem 4.6 is analogous to the proof for Theorem 4.15 in \cite{Wang1}.
\end{proof}

\begin{thm}
Let $(B, g_{B}, J_{B})$ be a almost product non-degenerate semi-Riemannian manifold and $(F_{j}, g_{F_{j}}, J_{F_{j}})$ be almost product semi-regular semi-Riemannian manifolds and $b_{j}\in C^{\infty}(B),~ 1\leq j\leq m.$ Then the  multiply warped product manifold $(B\times_{b_{1}}F_{1}\times\cdots\times_{b_{m}}F_{m}, J)$ is a almost product semi-regular semi-Riemannian manifold.
\end{thm}

\begin{proof}
In the notation of \cite{OC1}, it is sufficient to prove that $\widetilde{\mathcal{K}}(X, Y, \bullet)\widetilde{\mathcal{K}}(Z, T, \bullet)\in C^{\infty}(B\times_{b_{1}}F_{1}\times\cdots\times_{b_{m}}F_{m}).$
Set $\bullet$ for the covariant contraction on $B\times_{b_{1}}F_{1}\times\cdots\times_{b_{m}}F_{m},$ $\bullet_{B}$ and $\bullet_{F_{j}}$ for the covariant contraction on $B,$ $F_{j},$ respectively.
Let $X, Y, Z, T\in\Gamma(TB)$ and $U_{j}, V_{j}, W_{j}, Q_{j}\in\Gamma(TF_{j}).$
From now on we tacitly assume that $b_{j}>0,$ we obtain
\begin{align}
&(1)\widetilde{\mathcal{K}}(X, Y, \bullet)\widetilde{\mathcal{K}}(Z, T, \bullet)=\widetilde{\mathcal{K}}(X, Y, \bullet_{B})\widetilde{\mathcal{K}}(Z, T, \bullet_{B})+\sum^{m}_{l=1}\frac{1}{b^{2}_{l}}\widetilde{\mathcal{K}}(X, Y, \bullet_{F_{l}})\widetilde{\mathcal{K}}(Z, T, \bullet_{F_{l}})\nonumber\\
&=\widetilde{\mathcal{K}}_{B}(X, Y, \bullet_{B})\widetilde{\mathcal{K}}_{B}(Z, T, \bullet_{B});\nonumber\\
&(2)\widetilde{\mathcal{K}}(X, Y, \bullet)\widetilde{\mathcal{K}}(Z, W_{j}, \bullet)=\widetilde{\mathcal{K}}(X, Y, \bullet)\widetilde{\mathcal{K}}(W_{j}, Z, \bullet)=0;\nonumber\\
&(3)\widetilde{\mathcal{K}}(X, Y, \bullet)\widetilde{\mathcal{K}}(V_{i}, W_{j}, \bullet)=-\frac{1}{2}b_{j}g^{*}_{B}(\widetilde{\nabla}^{\flat}_{X}Y, db_{j})g_{F_{j}}(W_{j}, V_{j})-\frac{1}{2}b_{j}g^{*}_{B}(\widetilde{\nabla}^{\flat}_{X}Y, db_{j}\circ J_{B})\nonumber\\
&\times g_{F_{j}}(W_{j}, J_{F_{j}}V_{j}),~for~i=j;\nonumber\\
&(4)\widetilde{\mathcal{K}}(X, V_{i}, \bullet)\widetilde{\mathcal{K}}(Y, W_{j}, \bullet)=X(b_{j})Y(b_{j})g_{F_{j}}(W_{j}, V_{j}),~for~i=j;\nonumber\\
&(5)\widetilde{\mathcal{K}}(X, V_{i}, \bullet)\widetilde{\mathcal{K}}(W_{j}, Y, \bullet)=\frac{1}{2}X(b_{j})Y(b_{j})g_{F_{j}}(W_{j}, V_{j})+\frac{1}{2}X(b_{j})J_{B}Y(b_{j})g_{F_{j}}(W_{j}, J_{F_{j}}V_{j}),\nonumber\\
&~for~i=j;\nonumber\\
&(6)\widetilde{\mathcal{K}}(V_{i}, X, \bullet)\widetilde{\mathcal{K}}(W_{j}, Y, \bullet)=\frac{1}{4}X(b_{j})Y(b_{j})g_{F_{j}}(W_{j}, V_{j})+\frac{1}{4}X(b_{j})J_{B}Y(b_{j})g_{F_{j}}(W_{j}, J_{F_{j}}V_{j})\nonumber\\
&+\frac{1}{4}J_{B}X(b_{j})Y(b_{j})g_{F_{j}}(W_{j}, J_{F_{j}}V_{j})+\frac{1}{4}J_{B}X(b_{j})J_{B}Y(b_{j})g_{F_{j}}(W_{j}, V_{j}),~for~i=j;\nonumber\\
&(7)\widetilde{\mathcal{K}}(X, Y, \bullet)\widetilde{\mathcal{K}}(V_{i}, W_{j}, \bullet)=\widetilde{\mathcal{K}}(X, V_{i}, \bullet)\widetilde{\mathcal{K}}(Y, W_{j}, \bullet)=\widetilde{\mathcal{K}}(X, V_{i}, \bullet)\widetilde{\mathcal{K}}(W_{j}, Y, \bullet)\nonumber\\
&=\widetilde{\mathcal{K}}(V_{i}, X, \bullet)\widetilde{\mathcal{K}}(Y, W_{j}, \bullet)=\widetilde{\mathcal{K}}(V_{i}, X, \bullet)\widetilde{\mathcal{K}}(W_{j}, Y, \bullet)=0,~for~i\neq j;\nonumber\\
&(8)\widetilde{\mathcal{K}}(X, V_{i}, \bullet)\widetilde{\mathcal{K}}(W_{j}, U_{k}, \bullet)=b_{j}X(b_{j})\widetilde{\mathcal{K}}_{F_{j}}(W_{j}, U_{j}, V_{j}),~for~i=j=k;\nonumber\\
&(9)\widetilde{\mathcal{K}}(X, V_{i}, \bullet)\widetilde{\mathcal{K}}(W_{j}, U_{k}, \bullet)=\widetilde{\mathcal{K}}(X, V_{i}, \bullet)\widetilde{\mathcal{K}}(U_{k}, W_{j}, \bullet)=\widetilde{\mathcal{K}}(X, U_{k}, \bullet)\widetilde{\mathcal{K}}(V_{i}, W_{j}, \bullet)=0,\nonumber\\
&~for~i=j\neq k;\nonumber\\
&(10)\widetilde{\mathcal{K}}(V_{i}, X, \bullet)\widetilde{\mathcal{K}}(W_{j}, U_{k}, \bullet)=\frac{1}{2}b_{j}X(b_{j})\widetilde{\mathcal{K}}_{F_{j}}(W_{j}, U_{j}, V_{j})+\frac{1}{2}b_{j}J_{B}X(b_{j})\widetilde{\mathcal{K}}_{F_{j}}(W_{j}, U_{j}, J_{F_{j}}V_{j}),\nonumber\\
&~for~i=j=k;\nonumber\\
&(11)\widetilde{\mathcal{K}}(V_{i}, X, \bullet)\widetilde{\mathcal{K}}(W_{j}, U_{k}, \bullet)=\widetilde{\mathcal{K}}(V_{i}, X, \bullet)\widetilde{\mathcal{K}}(U_{k}, W_{j}, \bullet)=\widetilde{\mathcal{K}}(U_{k}, X, \bullet)\widetilde{\mathcal{K}}(V_{i}, W_{j}, \bullet)=0,\nonumber\\
&~for~i=j\neq k;\nonumber\\
&(12)\widetilde{\mathcal{K}}(X, V_{i}, \bullet)\widetilde{\mathcal{K}}(W_{j}, U_{k}, \bullet)=\widetilde{\mathcal{K}}(V_{i}, X, \bullet)\widetilde{\mathcal{K}}(W_{j}, U_{k}, \bullet)=0,~where~i, j, k~are~different;\nonumber
\end{align}
\begin{align}
&(13)\widetilde{\mathcal{K}}(V_{i}, W_{j}, \bullet)\widetilde{\mathcal{K}}(U_{k}, Q_{s}, \bullet)=\frac{1}{4}b^{2}_{j}(g^{*}_{B}(db_{j}, db_{j})g_{F_{j}}(W_{j}, V_{j})g_{F_{j}}(U_{j}, Q_{j})+g^{*}_{B}(db_{j}\circ J_{B}, db_{j})\nonumber\\
&\times g_{F_{j}}(W_{j}, J_{F_{j}}V_{j})g_{F_{j}}(U_{j}, Q_{j})+g^{*}_{B}(db_{j}, db_{j}\circ J_{B})g_{F_{j}}(W_{j}, V_{j})g_{F_{j}}(U_{j}, J_{F_{j}}Q_{j})\nonumber\\
&+g^{*}_{B}(db_{j}\circ J_{B}, db_{j}\circ J_{B})g_{F_{j}}(W_{j}, J_{F_{j}}V_{j})g_{F_{j}}(U_{j}, J_{F_{j}}Q_{j}))+b^{2}_{j}\widetilde{\mathcal{K}}_{F_{j}}(V_{j}, W_{j}, \bullet_{F_{j}})\widetilde{\mathcal{K}}_{F_{j}}(U_{j}, Q_{j}, \bullet_{F_{j}}),\nonumber\\
&~for~i=j=k=s;\nonumber\\
&(14)\widetilde{\mathcal{K}}(V_{i}, W_{j}, \bullet)\widetilde{\mathcal{K}}(U_{k}, Q_{s}, \bullet)=\widetilde{\mathcal{K}}(V_{i}, W_{j}, \bullet)\widetilde{\mathcal{K}}( Q_{s}, U_{k}, \bullet)=0,~for~i=j=s\neq k;\nonumber\\
&(15)\widetilde{\mathcal{K}}(V_{i}, W_{j}, \bullet)\widetilde{\mathcal{K}}(U_{k}, Q_{s}, \bullet)=\frac{1}{4}b_{j}b_{k}(g^{*}_{B}(db_{j}, db_{k})g_{F_{j}}(W_{j}, V_{j})g_{F_{k}}(U_{k}, Q_{k})\nonumber\\
&+g^{*}_{B}(db_{j}\circ J_{B}, db_{k})g_{F_{j}}(W_{j}, J_{F_{j}}V_{j})g_{F_{k}}(U_{k}, Q_{k})+g^{*}_{B}(db_{j}, db_{k}\circ J_{B})g_{F_{j}}(W_{j}, V_{j})\nonumber\\
&\times g_{F_{k}}(U_{k}, J_{F_{k}}Q_{k})+g^{*}_{B}(db_{j}\circ J_{B}, db_{k}\circ J_{B})g_{F_{j}}(W_{j}, J_{F_{j}}V_{j})g_{F_{k}}(U_{k}, J_{F_{k}}Q_{k})),\nonumber\\
&~for~i=j\neq k=s;\nonumber\\
&(16)\widetilde{\mathcal{K}}(V_{i}, W_{j}, \bullet)\widetilde{\mathcal{K}}(U_{k}, Q_{s}, \bullet)=0,~for~i=k\neq j=s~or~i=s\neq j=k;\nonumber\\
&(17)\widetilde{\mathcal{K}}(V_{i}, W_{j}, \bullet)\widetilde{\mathcal{K}}(U_{k}, Q_{s}, \bullet)=0,~other~cases.\nonumber
\end{align}
Clearly $(1)-(17)$ are smooth and smoothly extend to the region $b_{j}=0,$ and, in consequence, $(B\times_{b_{1}}F_{1}\times\cdots\times_{b_{m}}F_{m}, J)$ is a almost product semi-regular semi-Riemannian manifold.
\end{proof}

By using Proposition 4.4, Proposition 4.5 and making tedious calculations, we state the following propositions.
\begin{thm}
Let $(B, g_{B}, J_{B})$ be a  almost product non-degenerate semi-Riemannian manifold, $(F_{j}, g_{F_{j}}, J_{F_{j}})$ be almost product semi-regular semi-Riemannian manifolds and $b_{j}\in C^{\infty}(B),$ $1\leq j\leq m.$
Let the vector fields $X, Y, Z, T\in\Gamma(TB)$ and $U_{j}, V_{j}, W_{j}, Q_{j}\in\Gamma(TF_{j}),$ then
\begin{align}
&(1)\widetilde{R}(X, Y, Z, T)=\widetilde{R}_{B}(X, Y, Z, T);\nonumber\\
&(2)\widetilde{R}(X, Y, Z, W_{j})=\widetilde{R}(Z, W_{j}, X, Y)=0;\nonumber\\
&(3)\widetilde{R}(X, Y,  V_{i}, W_{j})=\widetilde{R}(V_{i}, W_{j}, X, Y)=0;\nonumber\\
&(4)\widetilde{R}(X, V_{i}, Y, W_{j})=\frac{1}{2}b_{j}XY(b_{j})g_{F_{j}}(V_{j}, W_{j})-\frac{1}{2}b_{j}g^{*}_{B}(\widetilde{\nabla}^{\flat}_{X}Y, db_{j})g_{F_{j}}(V_{j}, W_{j})\nonumber\\
&+\frac{1}{2}b_{j}XJ_{B}Y(b_{j})g_{F_{j}}(V_{j}, J_{F_{j}}W_{j})-\frac{1}{2}b_{j}g^{*}_{B}(\widetilde{\nabla}^{\flat}_{X}Y, db_{j}\circ J_{B})g_{F_{j}}(V_{j}, J_{F_{j}}W_{j}),~if~i=j;\nonumber\\
&(5)\widetilde{R}(X, V_{i}, Y, W_{j})=0,~if~i\neq j;\nonumber\\
&(6)\widetilde{R}(X, V_{i}, W_{j}, U_{k})=0;\nonumber
\end{align}
\begin{align}
&(7)\widetilde{R}(U_{k}, V_{i}, X, W_{j})=\frac{1}{4}b_{j}X(b_{j})(-\mathcal{K}_{F_{j}}(V_{j}, W_{j}, U_{j})+\mathcal{K}_{F_{j}}(U_{j}, W_{j}, V_{j})+\mathcal{K}_{F_{j}}(V_{j}, J_{F_{j}}W_{j}, J_{F_{j}}U_{j})\nonumber\\
&-\mathcal{K}_{F_{j}}(U_{j}, J_{F_{j}}W_{j}, J_{F_{j}}V_{j}))-\frac{1}{4}b_{j}J_{B}X(b_{j})(-\mathcal{K}_{F_{j}}(V_{j}, W_{j}, J_{F_{j}}U_{j})+\mathcal{K}_{F_{j}}(U_{j}, W_{j}, J_{F_{j}}V_{j})\nonumber\\
&+\mathcal{K}_{F_{j}}(V_{j}, J_{F_{j}}W_{j}, U_{j})-\mathcal{K}_{F_{j}}(U_{j}, J_{F_{j}}W_{j}, V_{j})),~if~i=j=k;\nonumber\\
&(8)\widetilde{R}(U_{k}, V_{i}, X, W_{j})=0,~other~cases;\nonumber\\
&(9)\widetilde{R}(U_{k}, V_{i}, W_{j}, Q_{s})=b^{2}_{j}\widetilde{R}_{F_{j}}(U_{j}, V_{j}, W_{j}, Q_{j})+\frac{1}{4}b^{2}_{j}[g^{*}_{B}(db_{j}, db_{j})(g_{F_{j}}(U_{j}, W_{j})g_{F_{j}}(V_{j}, Q_{j})\nonumber\\
&-g_{F_{j}}(V_{j}, W_{j})g_{F_{j}}(U_{j}, Q_{j}))+g^{*}_{B}(db_{j}\circ J_{B}, db_{j})(g_{F_{j}}(U_{j}, J_{F_{j}}W_{j})g_{F_{j}}(V_{j}, Q_{j})\nonumber\\
&-g_{F_{j}}(V_{j}, J_{F_{j}}W_{j})g_{F_{j}}(U_{j}, Q_{j})+g_{F_{j}}(U_{j}, W_{j})g_{F_{j}}(V_{j}, J_{F_{j}}Q_{j})-g_{F_{j}}(V_{j}, W_{j})g_{F_{j}}(U_{j}, J_{F_{j}}Q_{j}))\nonumber\\
&+g^{*}_{B}(db_{j}\circ J_{B}, db_{j}\circ J_{B})(g_{F_{j}}(U_{j}, J_{F_{j}}W_{j})g_{F_{j}}(V_{j}, J_{F_{j}}Q_{j})-g_{F_{j}}(V_{j}, J_{F_{j}}W_{j})g_{F_{j}}(U_{j}, J_{F_{j}}Q_{j}))],\nonumber\\
&~if~i=j=k=s;\nonumber\\
&(10)\widetilde{R}(U_{k}, V_{i}, W_{j}, Q_{s})=\frac{1}{4}b_{i}b_{j}(g^{*}_{B}(db_{i}, db_{j})g_{F_{i}}(V_{i}, Q_{i})g_{F_{j}}(U_{j}, W_{j})+g^{*}_{B}(db_{i}\circ J_{B}, db_{j})\nonumber\\
&\times g_{F_{i}}(V_{i}, J_{F_{i}}Q_{i})g_{F_{j}}(U_{j}, W_{j})+g^{*}_{B}(db_{i}, db_{j}\circ J_{B})g_{F_{i}}(V_{i}, Q_{i})g_{F_{j}}(U_{j}, J_{F_{j}}W_{j})\nonumber\\
&+g^{*}_{B}(db_{i}\circ J_{B}, db_{j}\circ J_{B})g_{F_{i}}(V_{i}, J_{F_{j}}Q_{i})g_{F_{j}}(U_{j}, J_{F_{j}}W_{j})),~if~j=k\neq i=s;\nonumber\\
&(11)\widetilde{R}(U_{k}, V_{i}, W_{j}, Q_{s})=0,~other~cases.\nonumber
\end{align}
\end{thm}

\begin{proof}
This follows by the same method as in \cite{Wang1}.\\
\begin{align}
(1)\widetilde{R}(X, Y, Z, T)&=X(\widetilde{\mathcal{K}}(Y, Z, T))-Y(\widetilde{\mathcal{K}}(X, Z, T))-\widetilde{\mathcal{K}}([X, Y], Z, T)\nonumber\\
&+\widetilde{\mathcal{K}}(X, Z, \bullet)\widetilde{\mathcal{K}}(Y, T, \bullet)-\widetilde{\mathcal{K}}(Y, Z, \bullet)\widetilde{\mathcal{K}}(X, T, \bullet)\nonumber\\
&=X(\widetilde{\mathcal{K}}_{B}(Y, Z, T))-Y(\widetilde{\mathcal{K}}_{B}(X, Z, T))-\widetilde{\mathcal{K}}_{B}([X, Y], Z, T)\nonumber\\
&+\widetilde{\mathcal{K}}_{B}(X, Z, \bullet_{B})\widetilde{\mathcal{K}}_{B}(Y, T, \bullet_{B})-\widetilde{\mathcal{K}}_{B}(Y, Z, \bullet_{B})\widetilde{\mathcal{K}}_{B}(X, T, \bullet_{B})\nonumber\\
&=\widetilde{R}_{B}(X, Y, Z, T),\nonumber
\end{align}
where we applied (1) from Proposition 4.4 and Theorem 4.7.
\begin{align}
(4)\widetilde{R}(X, V_{j}, Y, W_{j})&=X(\widetilde{\mathcal{K}}(V_{j}, Y, W_{j}))-V_{j}(\widetilde{\mathcal{K}}(X, Y, W_{j}))-\widetilde{\mathcal{K}}([X, V_{j}], Y, W_{j})\nonumber\\
&+\widetilde{\mathcal{K}}(X, Y, \bullet)\widetilde{\mathcal{K}}(V_{j}, W_{j}, \bullet)-\widetilde{\mathcal{K}}(V_{j}, Y, \bullet)\widetilde{\mathcal{K}}(X, W_{j}, \bullet)\nonumber\\
&=X[\frac{1}{2}(b_{j}Y(b_{j})g_{F_{j}}(V_{j}, W_{j})+b_{j}J_{B}Y(b_{j})g_{F_{j}}(V_{j}, J_{F_{j}}W_{j}))]\nonumber\\
&+(-\frac{1}{2}b_{j}g^{*}_{B}(\widetilde{\nabla}^{\flat}_{X}Y, db_{j})g_{F_{j}}(W_{j}, V_{j})-\frac{1}{2}b_{j}g^{*}_{B}(\widetilde{\nabla}^{\flat}_{X}Y, db_{j}\circ J_{B})\nonumber\\
&\times g_{F_{j}}(W_{j}, J_{F_{j}}V_{j}))-(\frac{1}{2}X(b_{j})Y(b_{j})g_{F_{j}}(W_{j}, V_{j})+\frac{1}{2}X(b_{j})\nonumber\\
&\times J_{B}Y(b_{j})g_{F_{j}}(W_{j}, J_{F_{j}}V_{j}))\nonumber
\end{align}
\begin{align}
&=\frac{1}{2}b_{j}XY(b_{j})g_{F_{j}}(V_{j}, W_{j})-\frac{1}{2}b_{j}g^{*}_{B}(\widetilde{\nabla}^{\flat}_{X}Y, db_{j})g_{F_{j}}(V_{j}, W_{j})\nonumber\\
&+\frac{1}{2}b_{j}XJ_{B}Y(b_{j})g_{F_{j}}(V_{j}, J_{F_{j}}W_{j})-\frac{1}{2}b_{j}g^{*}_{B}(\widetilde{\nabla}^{\flat}_{X}Y, db_{j}\circ J_{B})\nonumber\\
&\times g_{F_{j}}(V_{j}, J_{F_{j}}W_{j}).\nonumber
\end{align}
\begin{align}
(10)\widetilde{R}(U_{j}, V_{i}, W_{j}, Q_{i})&=U_{j}(\widetilde{\mathcal{K}}(V_{i}, W_{j}, Q_{i}))-V_{i}(\widetilde{\mathcal{K}}(U_{j}, W_{j}, Q_{i}))-\widetilde{\mathcal{K}}([U_{j}, V_{i}], W_{j}, Q_{i})\nonumber\\
&+\widetilde{\mathcal{K}}(U_{j}, W_{j}, \bullet)\widetilde{\mathcal{K}}(V_{i}, Q_{i}, \bullet)-\widetilde{\mathcal{K}}(V_{i}, W_{j}, \bullet)\widetilde{\mathcal{K}}(U_{j}, Q_{i}, \bullet)\nonumber\\
&=\frac{1}{4}b_{i}b_{j}(g^{*}_{B}(db_{i}, db_{j})g_{F_{i}}(V_{i}, Q_{i})g_{F_{j}}(U_{j}, W_{j})+g^{*}_{B}(db_{i}\circ J_{B}, db_{j})\nonumber\\
&\times g_{F_{i}}(V_{i}, J_{F_{i}}Q_{i})g_{F_{j}}(U_{j}, W_{j})+g^{*}_{B}(db_{i}, db_{j}\circ J_{B})g_{F_{i}}(V_{i}, Q_{i})g_{F_{j}}(U_{j}, J_{F_{j}}W_{j})\nonumber\\
&+g^{*}_{B}(db_{i}\circ J_{B}, db_{j}\circ J_{B})g_{F_{i}}(V_{i}, J_{F_{j}}Q_{i})g_{F_{j}}(U_{j}, J_{F_{j}}W_{j})).\nonumber
\end{align}

\end{proof}
\vskip 1 true cm
\section{ Semi-Symmetric Metric Koszul Forms and Their Curvature of Singular Twisted Products}
We introduce the notion of twisted product that is useful for the study in this section.

\begin{defn}{\rm\cite{MF} }
Let $(B, g_{B})$ and $(F, g_{F})$  be semi-Riemannian manifolds, and let, $\pi_{B}: B\times F\rightarrow B$ and $\pi_{F}: B\times F\rightarrow F$ be the canonical
projections. Also let $f: B\times F\rightarrow(0, \infty)$ and $b: B\times F\rightarrow(0, \infty)$ be smooth functions.
Then the doubly twisted product of $(B, g_{B})$ and $(F, g_{F})$ with twisting functions $f$ and $b$ is defined to be the product manifold $M=B\times F$ with metric
tensor $g=f^{2}g_{B}\oplus b^{2}g_{F}$ given by $g=f^{2}\pi^{*}_{B}(g_{B})+b^{2}\pi^{*}_{F}(g_{F}).$ For brevity in notation, we denote this semi-Riemannian manifold $(M, g)$ by ${}_{f~}\!B\times_{b}F.$ In particular, if $f = 1$ then ${}_{1~}\!B\times_{b}F=B\times_{b}F$ is called the twisted product of $(B, g_{B})$ and $(F, g_{F})$ with twisting function $b.$
\end{defn}

As in Definition 3.1 of \cite{OC2}, we can generalize the twisted product to singular semi-Riemannian manifolds.

\begin{defn}
Let $(B, g_{B})$ and $(F, g_{F})$  be singular semi-Riemannian manifolds, and $b: B\times F\rightarrow \mathbb{R}$ a smooth function. The twisted product of $B$ and $F$ with twisting function $b$ is the semi-Riemannian manifold
\begin{equation}
B\times_{b}F:=(B\times F, \pi^{*}_{B}(g_{B})+b^{2}\pi^{*}_{F}(g_{F})),
\end{equation}
where $\pi_{B}: B\times F\rightarrow B$ and $\pi_{F}: B\times F\rightarrow F$ are the canonical projections.
\end{defn}

By computations, we have the following proposition.
\begin{prop}
Let $B\times_{b}F$ be a degenerate twisted product and let the vector fields $X, Y, Z\in\Gamma(TB)$ and $U, V, W\in\Gamma(TF).$ Let $\mathcal{K}$ be the Koszul form on $B\times_{b}F$ and $\mathcal{K}_{B}$, $\mathcal{K}_{F}$ be the lifts of the Koszul form on $B,$ $F,$ respectively. Then
\begin{align}
&(1)\mathcal{K}(X, Y, Z)=\mathcal{K}_{B}(X, Y, Z);\nonumber\\
&(2)\mathcal{K}(X, Y, W)=\mathcal{K}(X, W, Y)=\mathcal{K}(W, X, Y)=0;\nonumber\\
&(3)\mathcal{K}(X, V, W)=\mathcal{K}(V, X, W)=-\mathcal{K}(V, W, X)=bX(b)g_{F}(V, W);\nonumber\\
&(4)\mathcal{K}(U, V, W)=bU(b)g_{F}(V, W)+bV(b)g_{F}(W, U)-bW(b)g_{F}(U, V)+b^{2}\mathcal{K}_{F}(U, V, W).\nonumber
\end{align}
\end{prop}

Through the analyses above and propositions in \cite{OC1}, we get the following results.

\begin{thm}
Let $(B, g_{B})$ and $(F, g_{F})$ be two radical-stationary semi-Riemannian manifolds and $b\in C^{\infty}(B\times F),$ so that $db\in A^{\bullet}(B\times_{b} F).$ Then the twisted product manifold $B\times_{b} F$ is a radical-stationary semi-Riemannian manifold.
\end{thm}

\begin{proof}
The proof is completed by showing that $\mathcal{K}(X, Y, W_{p})=0$ for any $X, Y\in \Gamma(B\times_{b}F)$ and $W_{p}\in \Gamma_{0}(B\times_{b}F).$ According to $X\in \Gamma(B\times_{b}F),$ we have $X=h_{B}X_{B}+h_{F}X_{F},$ where $h_{B}, h_{F}\in C^{\infty}(B\times F),$ $X_{B}\in \Gamma(B)$ and $X_{F}\in \Gamma(F).$ Similarly, we obtain $W_{p}=\overline{h}_{B}W_{p,B}+\overline{h}_{F}W_{p,F},$ where $\overline{h}_{B}, \overline{h}_{F}\in C^{\infty}(B\times F),$ $W_{p,B}\in \Gamma(B)$ and $W_{p,F}\in \Gamma(F).$ Since $W_{p}\in \Gamma_{0}(B\times_{b}F),$ we can get $g(W_{p}, Z)=0,$ for all $Z\in \Gamma(B\times_{b}F),$ that is, $g_{B}(\overline{h}_{B}W_{p,B}, Z_{B})=0$ and  $b^{2}g_{F}(\overline{h}_{F}W_{p,F}, Z_{F})=0.$ In this way  $b=0$ or $g_{F}(\overline{h}_{F}W_{p,F}, Z_{F})=0.$ By the Proposition 5.3 above and $db\in A^{\bullet}(B\times_{b} F),$ the theorem is proved.
\end{proof}

\begin{thm}
Let $(B, g_{B})$ and $(F, g_{F})$ be two non-degenerate semi-Riemannian manifolds, $b\in C^{\infty}(B\times F).$ Suppose that $db|_{b=0}=0,$ then the twisted product manifold $B\times_{b}F$ is a semi-regular semi-Riemannian manifold.
\end{thm}

\begin{proof}
We shall have established the theorem if we prove the following: $\mathcal{K}(X, Y, \bullet)\mathcal{K}(Z, T, \bullet)\\ \in C^{\infty}(B\times_{b}F).$
Let $X, Y, Z, T\in\Gamma(TB)$ and $U, V, W, Q\in\Gamma(TF).$ Assume that $b>0,$ we have
\begin{align}
&(1)\mathcal{K}(X, Y, \bullet)\mathcal{K}(Z, T, \bullet)=\mathcal{K}_{B}(X, Y, \bullet_{B})\mathcal{K}_{B}(Z, T, \bullet_{B});\nonumber\\
&(2)\mathcal{K}(X, Y, \bullet)\mathcal{K}(Z, W, \bullet)=\mathcal{K}(X, Y, \bullet)\mathcal{K}(W, Z, \bullet)=0;\nonumber\\
&(3)\mathcal{K}(X, Y, \bullet)\mathcal{K}(V, W, \bullet)=-bg^{*}_{B}(\nabla^{\flat}_{X}Y, db)g_{F}(V, W);\nonumber\\
&(4)\mathcal{K}(X, V, \bullet)\mathcal{K}(Y, W, \bullet)=\mathcal{K}(X, V, \bullet)\mathcal{K}(W, Y, \bullet)=\mathcal{K}(V, X, \bullet)\mathcal{K}(Y, W, \bullet)\nonumber\\
&=\mathcal{K}(V, X, \bullet)\mathcal{K}(W, Y, \bullet)=X(b)Y(b)g_{F}(V, W);\nonumber\\
&(5)\mathcal{K}(X, V, \bullet)\mathcal{K}(W, U, \bullet)=\mathcal{K}(V, X, \bullet)\mathcal{K}(W, U, \bullet)=X(b)W(b)g_{F}(U, V)+X(b)U(b)g_{F}(V, W)\nonumber\\
&-X(b)V(b)g_{F}(W, U)+bX(b)\mathcal{K}_{F}(W, U, V);\nonumber\\
&(6)\mathcal{K}(V, W, \bullet)\mathcal{K}(U, Q, \bullet)=b^{2}g^{*}_{B}(db, db)g_{F}(V, W)g_{F}(U, Q)+g^{*}_{F}(db, db)g_{F}(V, W)g_{F}(U, Q)\nonumber\\
&-bg^{*}_{F}(\nabla^{\flat}_{V}W, db)g_{F}(U, Q)-bg^{*}_{F}(\nabla^{\flat}_{U}Q, db)g_{F}(V, W)+
b^{2}\mathcal{K}_{F}(V, W, \bullet_{F})\mathcal{K}_{F}(U, Q, \bullet_{F})\nonumber\\
&+bV(b)\mathcal{K}_{F}(U, Q, W)+bW(b)\mathcal{K}_{F}(U, Q, V)+bU(b)\mathcal{K}_{F}(V, W, Q)+bQ(b)\mathcal{K}_{F}(V, W, U)\nonumber\\
&+V(b)U(b)g_{F}(Q, W)+V(b)Q(b)g_{F}(U, W)+W(b)U(b)g_{F}(Q, V)+W(b)Q(b)g_{F}(U, V)\nonumber\\
&-2U(b)Q(b)g_{F}(V, W)-2V(b)W(b)g_{F}(U, Q).\nonumber
\end{align}
Because $db|_{b=0}=0,$ we know that $(1)-(5)$ are smooth and smoothly extend to the region $b=0.$ From the above it follows that $B\times_{b}F$ is a semi-regular semi-Riemannian manifold.
\end{proof}

Using the formulae (2.7), we calculate
\begin{thm}
Let $(B, g_{B})$ and $(F, g_{F})$ be two non-degenerate semi-Riemannian manifolds, $b\in C^{\infty}(B\times F)$ and $db|_{b=0}=0.$
Let the vector fields $X, Y, Z, T\in\Gamma(TB)$ and $U, V, W, Q\in\Gamma(TF).$ Then
\begin{align}
&(1)R(X, Y, Z, T)=R_{B}(X, Y, Z, T);\nonumber\\
&(2)R(X, Y, Z, W)=R(Z, W, X, Y)=0;\nonumber\\
&(3)R(X, Y,  V, W)=0;\nonumber\\
&(4)R(X, V, Y, W)=R(V, X, W, Y)=bXY(b)g_{F}(V, W)-bg^{*}_{B}(\nabla^{\flat}_{X}Y, db)g_{F}(V, W);\nonumber\\
&(5)R(X, U, V, W)=bXV(b)g_{F}(U, W)-X(b)V(b)g_{F}(U, W)-bXW(b)g_{F}(U, V)\nonumber\\
&+X(b)W(b)g_{F}(U, V);\nonumber\\
&(6)R(V, W, U, Q)=b^{2}R_{F}(V, W, U, Q)+(b^{2}g^{*}_{B}(db, db)+g^{*}_{F}(db, db))(g_{F}(V, U)g_{F}(W, Q)\nonumber\\
&-g_{F}(W, U)g_{F}(V, Q))+(bVU(b)-2V(b)U(b)-bg^{*}_{F}(\nabla^{\flat}_{V}U, db))g_{F}(W, Q)\nonumber\\
&+(bWQ(b)-2W(b)Q(b)-bg^{*}_{F}(\nabla^{\flat}_{W}Q, db))g_{F}(V, U)-(bWU(b)-2W(b)U(b)\nonumber\\
&-bg^{*}_{F}(\nabla^{\flat}_{W}U, db))g_{F}(V, Q)-(bVQ(b)-2V(b)Q(b)-bg^{*}_{F}(\nabla^{\flat}_{V}Q, db))g_{F}(W, U).\nonumber
\end{align}
\end{thm}

The following propositions can be stated:
\begin{prop}
Let $B\times_{b}F$ be a degenerate twisted product and let the vector fields $X, Y, Z\in\Gamma(TB)$ and $U, V, W\in\Gamma(TF).$ Let $\overline{\mathcal{K}}$ be the semi-symmetric metric Koszul form on $B\times_{b}F$ and $\overline{\mathcal{K}}_{B}$, $\overline{\mathcal{K}}_{F}$ be the lifts of the semi-symmetric metric
Koszul form on $B,$ $F,$ respectively. Let $P\in\Gamma(TB),$ then
\begin{align}
&(1)\overline{\mathcal{K}}(X, Y, Z)=\overline{\mathcal{K}}_{B}(X, Y, Z);\nonumber\\
&(2)\overline{\mathcal{K}}(X, Y, W)=\overline{\mathcal{K}}(X, W, Y)=\overline{\mathcal{K}}(W, X, Y)=0;\nonumber\\
&(3)\overline{\mathcal{K}}(X, V, W)=bX(b)g_{F}(V, W);\nonumber\\
&(4)\overline{\mathcal{K}}(V, X, W)=-\overline{\mathcal{K}}(V, W, X)=bX(b)g_{F}(V, W)+b^{2}g_{B}(X, P)g_{F}(V, W);\nonumber\\
&(5)\overline{\mathcal{K}}(U, V, W)=bU(b)g_{F}(V, W)+bV(b)g_{F}(W, U)-bW(b)g_{F}(U, V)+b^{2}\mathcal{K}_{F}(U, V, W).\nonumber
\end{align}
\end{prop}

\begin{prop}
Let $B\times_{b}F$ be a degenerate twisted product and let the vector fields $X, Y, Z\in\Gamma(TB)$ and $U, V, W\in\Gamma(TF).$ Let $\overline{\mathcal{K}}$ be the semi-symmetric metric Koszul form on $B\times_{b}F$ and $\overline{\mathcal{K}}_{B}$, $\overline{\mathcal{K}}_{F}$ be the lifts of the semi-symmetric metric
Koszul form on $B,$ $F,$ respectively. Let $P\in\Gamma(TF),$ then
\begin{align}
&(1)\overline{\mathcal{K}}(X, Y, Z)=\mathcal{K}_{B}(X, Y, Z);\nonumber\\
&(2)\overline{\mathcal{K}}(X, Y, W)=-\overline{\mathcal{K}}(X, W, Y)=-b^{2}g_{B}(X, Y)g_{F}(P, W);\nonumber\\
&(3)\overline{\mathcal{K}}(W, X, Y)=0;\nonumber\\
&(4)\overline{\mathcal{K}}(X, V, W)=\overline{\mathcal{K}}(V, X, W)=-\overline{\mathcal{K}}(V, W, X)=bX(b)g_{F}(V, W);\nonumber\\
&(5)\overline{\mathcal{K}}(U, V, W)=bU(b)g_{F}(V, W)+bV(b)g_{F}(W, U)-bW(b)g_{F}(U, V)+b^{2}\mathcal{K}_{F}(U, V, W)\nonumber\\
&+b^{4}g_{F}(V, P)g_{F}(U, W)-b^{4}g_{F}(W, P)g_{F}(U, V).\nonumber
\end{align}
\end{prop}

Analysis similar to that in the proof of Theorem 2.11 shows that
\begin{thm}
Let $(B, g_{B})$ and $(F, g_{F})$ be two non-degenerate semi-Riemannian manifolds, $b\in C^{\infty}(B\times F)$ and $db|_{b=0}=0.$
Let the vector fields $X, Y, Z, T\in\Gamma(TB)$ and $U, V, W, Q\in\Gamma(TF).$  Let $P\in\Gamma(TB),$ then
\begin{align}
&(1)\overline{R}(X, Y, Z, T)=\overline{R}_{B}(X, Y, Z, T);\nonumber\\
&(2)\overline{R}(X, Y, Z, W)=\overline{R}(Z, W, X, Y)=0;\nonumber\\
&(3)\overline{R}(X, Y,  V, W)=\overline{R}(V, W, X, Y)=0;\nonumber\\
&(4)\overline{R}(X, V, Y, W)=bXY(b)g_{F}(V, W)-bg^{*}_{B}(\nabla^{\flat}_{X}Y, db)g_{F}(V, W)+b^{2}\mathcal{K}_{B}(X, P, Y)g_{F}(V, W)\nonumber\\
&+b^{2}g_{B}(P, P)g_{B}(X, Y)g_{F}(V, W)-b^{2}g_{B}(X, P)g_{B}(P, Y)g_{F}(V, W)+bP(b)g_{B}(X, Y)g_{F}(V, W);\nonumber\\
&(5)\overline{R}(X, V, U, W)=\overline{R}(U, W, X, V)=bXU(b)g_{F}(W, V)-X(b)U(b)g_{F}(W, V)\nonumber\\
&-bXW(b)g_{F}(U, V)+X(b)W(b)g_{F}(U, V);\nonumber\\
&(6)\overline{R}(U, V, W, Q)=b^{2}R_{F}(U, V, W, Q)+(b^{2}g^{*}_{B}(db, db)-g^{*}_{F}(db, db)+2b^{3}P(b)+b^{4}g_{B}(P, P))\nonumber\\
&\times (g_{F}(U, W)g_{F}(V, Q)-g_{F}(V, W)g_{F}(U, Q))
+(bUW(b)-2U(b)W(b)-bg^{*}_{F}(\nabla^{\flat}_{U}W, db))\nonumber\\
&\times g_{F}(V, Q)+(bVQ(b)-2V(b)Q(b)-bg^{*}_{F}(\nabla^{\flat}_{V}Q, db))g_{F}(U, W)
-(bVW(b)-2V(b)W(b)\nonumber\\&-bg^{*}_{F}(\nabla^{\flat}_{V}W, db))g_{F}(U, Q)
-(bUQ(b)-2U(b)Q(b)-bg^{*}_{F}(\nabla^{\flat}_{U}Q, db))g_{F}(V, W).\nonumber
\end{align}
\end{thm}

\begin{thm}
Let $(B, g_{B})$ and $(F, g_{F})$ be two non-degenerate semi-Riemannian manifolds, $b\in C^{\infty}(B\times F)$ and $db|_{b=0}=0.$
Let the vector fields $X, Y, Z, T\in\Gamma(TB)$ and $U, V, W, Q\in\Gamma(TF).$  Let $P\in\Gamma(TF),$ then
\begin{align}
&(1)\overline{R}(X, Y, Z, T)=R_{B}(X, Y, Z, T)+b^{2}g_{B}(X, Z)g_{B}(Y, T)g_{F}(P, P)\nonumber\\
&-b^{2}g_{B}(X, T)g_{B}(Y, Z)g_{F}(P, P);\nonumber\\
&(2)\overline{R}(X, Y, Z, W)=-\overline{R}(Z, W, X, Y)=-bX(b)g_{B}(Y, Z)g_{F}(P, W)+bY(b)g_{B}(X, Z)g_{F}(P, W);\nonumber\\
&(3)\overline{R}(X, Y,  V, W)=\overline{R}(V, W, X, Y)=0;\nonumber
\end{align}
\begin{align}
&(4)\overline{R}(X, V, Y, W)=bXY(b)g_{F}(V, W)-bg^{*}_{B}(\nabla^{\flat}_{X}Y, db)g_{F}(V, W)-b^{4}g_{B}(X, Y)g_{F}(V, P)\nonumber\\
&\times g_{F}(P, W)+b^{4}g_{B}(X, Y)g_{F}(V, W)g_{F}(P, P)+bV(b)g_{B}(X, Y)g_{F}(P, W)+bP(b)g_{B}(X, Y)\nonumber\\
&\times g_{F}(W, V)-bW(b)g_{B}(X, Y)g_{F}(V, P)+b^{2}g_{B}(X, Y)\mathcal{K}_{F}(V, P, W);\nonumber\\
&(5)\overline{R}(X, V, U, W)=bXU(b)g_{F}(W, V)-X(b)U(b)g_{F}(W, V)-bXW(b)g_{F}(U, V)\nonumber\\
&+X(b)W(b)g_{F}(U, V)+b^{3}X(b)g_{F}(U, P)g_{F}(W, V)-b^{3}X(b)g_{F}(W, P)g_{F}(U, V);\nonumber\\
&(6)\overline{R}(U, W, X, V)=bXU(b)g_{F}(W, V)-X(b)U(b)g_{F}(W, V)-bXW(b)g_{F}(U, V)\nonumber\\
&+X(b)W(b)g_{F}(U, V)-b^{3}X(b)g_{F}(U, P)g_{F}(W, V)+b^{3}X(b)g_{F}(W, P)g_{F}(U, V);\nonumber\\
&(7)\overline{R}(U, V, W, Q)=b^{2}R_{F}(U, V, W, Q)+(b^{2}g^{*}_{B}(db, db)+g^{*}_{F}(db, db)+2b^{3}P(b)+b^{6}g_{F}(P, P))\nonumber\\
&\times (g_{F}(U, W)g_{F}(V, Q)-g_{F}(V, W)g_{F}(U, Q))
+(bUW(b)-2U(b)W(b)+b^{4}\mathcal{K}_{F}(U, P, W)\nonumber\\
&-bg^{*}_{F}(\nabla^{\flat}_{U}W, db))g_{F}(V, Q)+(bVQ(b)-2V(b)Q(b)+b^{4}\mathcal{K}_{F}(V, P, Q)\nonumber\\
&-bg^{*}_{F}(\nabla^{\flat}_{V}Q, db))g_{F}(U, W)-(bVW(b)-2V(b)W(b)+b^{4}\mathcal{K}_{F}(V, P, W)\nonumber\\
&-bg^{*}_{F}(\nabla^{\flat}_{V}W, db))g_{F}(U, Q)-(bUQ(b)-2U(b)Q(b)+b^{4}\mathcal{K}_{F}(U, P, Q)\nonumber\\
&-bg^{*}_{F}(\nabla^{\flat}_{U}Q, db))g_{F}(V, W)+b^{3}Q(b)(g_{F}(U, P)g_{F}(V, W)-g_{F}(V, P)g_{F}(U, W))\nonumber\\
&+(b^{6}g_{F}(U, P)-b^{3}U(b))(g_{F}(V, W)g_{F}(P, Q)-g_{F}(P, W)g_{F}(V, Q))\nonumber\\
&-(b^{6}g_{F}(V, P)-b^{3}V(b))(g_{F}(U, W)g_{F}(P, Q)-g_{F}(P, W)g_{F}(U, Q))\nonumber\\
&-b^{3}W(b)(g_{F}(U, P)g_{F}(V, Q)-g_{F}(V, P)g_{F}(U, Q)).\nonumber
\end{align}
\end{thm}

\vskip 1 true cm
\section{ Semi-Symmetric Non-Metric Koszul Forms and Their Curvature of Singular Twisted Products}
By the analysis above, we can immediately state the following propositions.
\begin{prop}
Let $B\times_{b}F$ be a degenerate twisted product and let the vector fields $X, Y, Z\in\Gamma(TB)$ and $U, V, W\in\Gamma(TF).$ Let $\widehat{\mathcal{K}}$ be the semi-symmetric non-metric Koszul form on $B\times_{b}F$ and $\widehat{\mathcal{K}}_{B}$, $\widehat{\mathcal{K}}_{F}$ be the lifts of the semi-symmetric non-metric
Koszul form on $B,$ $F,$ respectively. Let $P\in\Gamma(TB),$ then
\begin{align}
&(1)\widehat{\mathcal{K}}(X, Y, Z)=\widehat{\mathcal{K}}_{B}(X, Y, Z);\nonumber\\
&(2)\widehat{\mathcal{K}}(X, Y, W)=\widehat{\mathcal{K}}(X, W, Y)=\widehat{\mathcal{K}}(W, X, Y)=0;\nonumber\\
&(3)\widehat{\mathcal{K}}(X, V, W)=-\widehat{\mathcal{K}}(V, W, X)=bX(b)g_{F}(V, W);\nonumber\\
&(4)\widehat{\mathcal{K}}(V, X, W)=bX(b)g_{F}(V, W)+b^{2}g_{B}(X, P)g_{F}(V, W);\nonumber\\
&(5)\widehat{\mathcal{K}}(U, V, W)=bU(b)g_{F}(V, W)+bV(b)g_{F}(W, U)-bW(b)g_{F}(U, V)+b^{2}\mathcal{K}_{F}(U, V, W).\nonumber
\end{align}
\end{prop}

\begin{prop}
Let $B\times_{b}F$ be a degenerate twisted product and let the vector fields $X, Y, Z\in\Gamma(TB)$ and $U, V, W\in\Gamma(TF).$ Let $\widehat{\mathcal{K}}$ be the semi-symmetric non-metric Koszul form on $B\times_{b}F$ and $\widehat{\mathcal{K}}_{B}$, $\widehat{\mathcal{K}}_{F}$ be the lifts of the semi-symmetric non-metric
Koszul form on $B,$ $F,$ respectively. Let $P\in\Gamma(TF),$ then
\begin{align}
&(1)\widehat{\mathcal{K}}(X, Y, Z)=\mathcal{K}_{B}(X, Y, Z);\nonumber\\
&(2)\widehat{\mathcal{K}}(X, Y, W)=\widehat{\mathcal{K}}(W, X, Y)=0;\nonumber\\
&(3)\widehat{\mathcal{K}}(X, W, Y)=b^{2}g_{B}(X, Y)g_{F}(W, P);\nonumber\\
&(4)\widehat{\mathcal{K}}(X, V, W)=\widehat{\mathcal{K}}(V, X, W)=-\widehat{\mathcal{K}}(V, W, X)=bX(b)g_{F}(V, W);\nonumber\\
&(5)\widehat{\mathcal{K}}(U, V, W)=bU(b)g_{F}(V, W)+bV(b)g_{F}(W, U)-bW(b)g_{F}(U, V)+b^{2}\mathcal{K}_{F}(U, V, W)\nonumber\\
&+b^{4}g_{F}(V, P)g_{F}(U, W).\nonumber
\end{align}
\end{prop}

Likewise, we have the following.
\begin{thm}
Let $(B, g_{B})$ and $(F, g_{F})$ be two non-degenerate semi-Riemannian manifolds, $b\in C^{\infty}(B\times F)$ and $db|_{b=0}=0.$
Let the vector fields $X, Y, Z, T\in\Gamma(TB)$ and $U, V, W, Q\in\Gamma(TF).$  Let $P\in\Gamma(TB),$ then
\begin{align}
&(1)\widehat{R}(X, Y, Z, T)=\widehat{R}_{B}(X, Y, Z, T);\nonumber\\
&(2)\widehat{R}(X, Y, Z, W)=\widehat{R}(X, Y, W, Z)=\widehat{R}(Z, W, X, Y)=0;\nonumber\\
&(3)\widehat{R}(X, Y, V, W)=\widehat{R}(V, W, X, Y)=0;\nonumber\\
&(4)\widehat{R}(V, X, W, Y)=-\widehat{R}(X, V, W, Y)=bXY(b)g_{F}(V, W)-bg^{*}_{B}(\nabla^{\flat}_{X}Y, db)g_{F}(V, W)\nonumber\\
&-2bX(b)g_{B}(Y, P)g_{F}(V, W);\nonumber\\
&(5)\widehat{R}(V, X, Y, W)=-\widehat{R}(X, V, Y, W)=-bXY(b)g_{F}(V, W)+bg^{*}_{B}(\nabla^{\flat}_{X}Y, db)g_{F}(V, W)\nonumber\\
&-b^{2}X(g_{B}(Y, P))g_{F}(V, W);\nonumber\\
&(6)\widehat{R}(X, U, V, W)=\widehat{R}(V, W, X, U)=bXV(b)g_{F}(W, U)-X(b)V(b)g_{F}(W, U)\nonumber\\
&-bXW(b)g_{F}(U, V)+X(b)W(b)g_{F}(U, V);\nonumber\\
&(7)\widehat{R}(V, W, U, X)=-bXV(b)g_{F}(W, U)+X(b)V(b)g_{F}(W, U)+bXW(b)g_{F}(U, V)\nonumber\\
&-X(b)W(b)g_{F}(U, V)-2bW(b)g_{B}(X, P)g_{F}(U, V)+2bV(b)g_{B}(X, P)g_{F}(U, W)\nonumber\\
&+b^{2}g_{B}(X, P)(\mathcal{K}_{F}(W, U, V)-\mathcal{K}_{F}(V, U, W));\nonumber\\
&(8)\widehat{R}(U, V, W, Q)=b^{2}R_{F}(U, V, W, Q)+(b^{2}g^{*}_{B}(db, db)+g^{*}_{F}(db, db))(g_{F}(U, W)g_{F}(V, Q)\nonumber\\
&-g_{F}(V, W)g_{F}(U, Q))+(bUW(b)-2U(b)W(b)-bg^{*}_{F}(\nabla^{\flat}_{U}W, db))g_{F}(V, Q)\nonumber\\
&+(bVQ(b)-2V(b)Q(b)-bg^{*}_{F}(\nabla^{\flat}_{V}Q, db))g_{F}(U, W)-(bVW(b)-2V(b)W(b)\nonumber\\
&-bg^{*}_{F}(\nabla^{\flat}_{V}W, db))g_{F}(U, Q)-(bUQ(b)-2U(b)Q(b)-bg^{*}_{F}(\nabla^{\flat}_{U}Q, db))g_{F}(V, W).\nonumber
\end{align}
\end{thm}

\begin{thm}
Let $(B, g_{B})$ and $(F, g_{F})$ be two non-degenerate semi-Riemannian manifolds, $b\in C^{\infty}(B\times F)$ and $db|_{b=0}=0.$
Let the vector fields $X, Y, Z, T\in\Gamma(TB)$ and $U, V, W, Q\in\Gamma(TF).$  Let $P\in\Gamma(TF),$ then
\begin{align}
&(1)\widehat{R}(X, Y, Z, T)=R_{B}(X, Y, Z, T);\nonumber\\
&(2)\widehat{R}(X, Y, Z, W)=b^{2}g_{F}(W, P)(\mathcal{K}_{B}(X, Z, Y)-\mathcal{K}_{B}(Y, Z, X));\nonumber\\
&(3)\widehat{R}(X, Y, W, Z)=2bg_{F}(W, P)(X(b)g_{B}(Y, Z)-Y(b)g_{B}(X, Z));\nonumber\\
&(4)\widehat{R}(Z, W, X, Y)=0;\nonumber\\
&(5)\widehat{R}(X, Y, V, W)=\widehat{R}(V, W, X, Y)=0;\nonumber\\
&(6)\widehat{R}(X, V, W, Y)=-\widehat{R}(V, X, W, Y)=-bXY(b)g_{F}(V, W)+bg^{*}_{B}(\nabla^{\flat}_{X}Y, db)g_{F}(V, W)\nonumber\\
&-2bV(b)g_{B}(X, Y)g_{F}(W, P)-b^{2}V(g_{F}(W, P))g_{B}(X, Y);\nonumber\\
&(7)\widehat{R}(X, V, Y, W)=-\widehat{R}(V, X, Y, W)=bXY(b)g_{F}(V, W)-bg^{*}_{B}(\nabla^{\flat}_{X}Y, db)g_{F}(V, W);\nonumber\\
&(8)\widehat{R}(X, V, U, W)=-\widehat{R}(V, X, U, W)=bXU(b)g_{F}(W, V)-X(b)U(b)g_{F}(W, V)\nonumber\\
&-bXW(b)g_{F}(V, U)+X(b)W(b)g_{F}(V, U)+2b^{3}X(b)(g_{F}(U, P)g_{F}(V, W)\nonumber\\
&+g_{F}(W, P)g_{F}(U, V));\nonumber\\
&(9)\widehat{R}(U, W, X, V)=-\widehat{R}(U, W, V, X)=bXU(b)g_{F}(W, V)-X(b)U(b)g_{F}(W, V)\nonumber\\
&-bXW(b)g_{F}(V, U)+X(b)W(b)g_{F}(V, U);\nonumber\\
&(10)\widehat{R}(U, V, W, Q)=b^{2}R_{F}(U, V, W, Q)+(b^{2}g^{*}_{B}(db, db)+g^{*}_{F}(db, db))(g_{F}(U, W)g_{F}(V, Q)\nonumber\\
&-g_{F}(V, W)g_{F}(U, Q))+(bUW(b)-2U(b)W(b)-bg^{*}_{F}(\nabla^{\flat}_{U}W, db))g_{F}(V, Q)\nonumber\\
&+(bVQ(b)-2V(b)Q(b)-bg^{*}_{F}(\nabla^{\flat}_{V}Q, db))g_{F}(U, W)-(bVW(b)-2V(b)W(b)\nonumber\\
&-bg^{*}_{F}(\nabla^{\flat}_{V}W, db))g_{F}(U, Q)-(bUQ(b)-2U(b)Q(b)-bg^{*}_{F}(\nabla^{\flat}_{U}Q, db))g_{F}(V, W)\nonumber\\
&+2b^{3}U(b)(g_{F}(W, P)g_{F}(V, Q)+g_{F}(Q, P)g_{F}(V, W))-2b^{3}V(b)(g_{F}(W, P)g_{F}(U, Q)\nonumber\\
&+g_{F}(Q, P)g_{F}(U, W))+b^{4}U(g_{F}(W, P))g_{F}(V, Q)-b^{4}V(g_{F}(W, P))g_{F}(U, Q)\nonumber\\
&+b^{4}g_{F}(P, Q)(\mathcal{K}_{F}(U, W, V)-\mathcal{K}_{F}(V, W, U)).\nonumber
\end{align}
\end{thm}

\vskip 1 true cm
\section{ Koszul Forms Associated with the Almost Product Structure and Their Curvature of Singular Twisted Products}

It is evident that
\begin{prop}
Let $B\times_{b}F$ be a degenerate twisted product and let the vector fields $X, Y, Z\in\Gamma(TB)$ and $U, V, W\in\Gamma(TF).$ Let $\widetilde{\mathcal{K}}$ be the almost product Koszul form on $B\times_{b}F$ and $\widetilde{\mathcal{K}}_{B}$, $\widetilde{\mathcal{K}}_{F}$ be the lifts of the almost product Koszul form on $B,$ $F,$ respectively. Then
\begin{align}
&(1)\widetilde{\mathcal{K}}(X, Y, Z)=\widetilde{\mathcal{K}}_{B}(X, Y, Z);\nonumber\\
&(2)\widetilde{\mathcal{K}}(X, Y, W)=\widetilde{\mathcal{K}}(X, W, Y)=\widetilde{\mathcal{K}}(W, X, Y)=0;\nonumber\\
&(3)\widetilde{\mathcal{K}}(X, V, W)=bX(b)g_{F}(V, W);\nonumber\\
&(4)\widetilde{\mathcal{K}}(V, X, W)=-\widetilde{\mathcal{K}}(V, W, X)=\frac{1}{2}(bX(b)g_{F}(V, W)+bJ_{B}X(b)g_{F}(V, J_{F}W));\nonumber\\
&(5)\widetilde{\mathcal{K}}(U, V, W)=bU(b)g_{F}(V, W)+\frac{1}{2}bV(b)g_{F}(U, W)-\frac{1}{2}bW(b)g_{F}(U, V)\nonumber\\
&+\frac{1}{2}bJ_{F}V(b)g_{F}(U, J_{F}W)-\frac{1}{2}bJ_{F}W(b)g_{F}(U, J_{F}V)+b^{2}\widetilde{\mathcal{K}}_{F}(U, V, W).\nonumber
\end{align}
\end{prop}

\begin{thm}
Let $(B, g_{B}, J_{B})$ and $(F, g_{F}, J_{F})$ be two almost product radical-stationary semi-Riemannian manifolds and $b\in C^{\infty}(B\times F),$ so that $db\in A^{\bullet}(B\times_{b} F).$ Then the twisted product manifold $(B\times_{b} F, J)$ is a almost product radical-stationary semi-Riemannian manifold.
\end{thm}

A trivial verification shows that
\begin{thm}
Let $(B, g_{B}, J_{B})$ and $(F, g_{F}, J_{F})$ be two almost product non-degenerate semi-Riemannian manifolds, $b\in C^{\infty}(B\times F).$ Suppose that $db|_{b=0}=0,$ then the twisted product manifold $(B\times_{b}F, J)$ is a almost product semi-regular semi-Riemannian manifold.
\end{thm}

\begin{proof}
Let $X, Y, Z, T\in\Gamma(TB)$ and $U, V, W, Q\in\Gamma(TF).$ If $b>0,$ we have
\begin{align}
&(1)\widetilde{\mathcal{K}}(X, Y, \bullet)\widetilde{\mathcal{K}}(Z, T, \bullet)=\widetilde{\mathcal{K}}_{B}(X, Y, \bullet_{B})\widetilde{\mathcal{K}}_{B}(Z, T, \bullet_{B});\nonumber\\
&(2)\widetilde{\mathcal{K}}(X, Y, \bullet)\widetilde{\mathcal{K}}(Z, W, \bullet)=\widetilde{\mathcal{K}}(X, Y, \bullet)\widetilde{\mathcal{K}}(W, Z, \bullet)=0;\nonumber\\
&(3)\widetilde{\mathcal{K}}(X, Y, \bullet)\widetilde{\mathcal{K}}(V, W, \bullet)=-\frac{1}{2}bg^{*}_{B}(\widetilde{\nabla}^{\flat}_{X}Y, db)g_{F}(W, V)-\frac{1}{2}bg^{*}_{B}(\widetilde{\nabla}^{\flat}_{X}Y, db\circ J_{B})g_{F}(W, J_{F}V);\nonumber\\
&(4)\widetilde{\mathcal{K}}(X, V, \bullet)\widetilde{\mathcal{K}}(Y, W, \bullet)=X(b)Y(b)g_{F}(W, V);\nonumber\\
&(5)\widetilde{\mathcal{K}}(X, V, \bullet)\widetilde{\mathcal{K}}(W, Y, \bullet)=\frac{1}{2}X(b)Y(b)g_{F}(W, V)+\frac{1}{2}X(b)J_{B}Y(b)g_{F}(W, J_{F}V);\nonumber\\
&(6)\widetilde{\mathcal{K}}(V, X, \bullet)\widetilde{\mathcal{K}}(W, Y, \bullet)=\frac{1}{4}X(b)Y(b)g_{F}(W, V)+\frac{1}{4}X(b)J_{B}Y(b)g_{F}(W, J_{F}V)\nonumber\\
&+\frac{1}{4}J_{B}X(b)Y(b)g_{F}(W, J_{F}V)+\frac{1}{4}J_{B}X(b)J_{B}Y(b)g_{F}(W, V);\nonumber\\
&(7)\widetilde{\mathcal{K}}(X, V, \bullet)\widetilde{\mathcal{K}}(W, U, \bullet)=bX(b)\widetilde{\mathcal{K}}_{F}(W, U, V)+X(b)W(b)g_{F}(U, V)+\frac{1}{2}X(b)U(b)g_{F}(W, V)\nonumber\\
&-\frac{1}{2}X(b)V(b)g_{F}(W, U)+\frac{1}{2}X(b)J_{F}U(b)g_{F}(W, J_{F}V)-\frac{1}{2}X(b)J_{F}V(b)g_{F}(W, J_{F}U);\nonumber
\end{align}
\begin{align}
&(8)\widetilde{\mathcal{K}}(V, X, \bullet)\widetilde{\mathcal{K}}(W, U, \bullet)=\frac{1}{2}bX(b)\widetilde{\mathcal{K}}_{F}(W, U, V)+\frac{1}{2}bJ_{B}X(b)\widetilde{\mathcal{K}}_{F}(W, U, J_{F}V)\nonumber\\
&+\frac{1}{2}X(b)W(b)g_{F}(U, V)+\frac{1}{2}J_{B}X(b)W(b)g_{F}(U, J_{F}V)+\frac{1}{4}(X(b)U(b)+J_{B}X(b)J_{F}U(b))\nonumber\\
&\times g_{F}(W, V)-\frac{1}{4}(X(b)V(b)+J_{B}X(b)J_{F}V(b))g_{F}(W, U)+\frac{1}{4}(X(b)J_{F}U(b)+J_{B}X(b)U(b))\nonumber\\
&\times g_{F}(W, J_{F}V)-\frac{1}{4}(X(b)J_{F}V(b)+J_{B}X(b)V(b))g_{F}(W, J_{F}U);\nonumber\\
&(9)\widetilde{\mathcal{K}}(V, W, \bullet)\widetilde{\mathcal{K}}(U, Q, \bullet)=\frac{1}{4}(b^{2}g^{*}_{B}(db, db)+g^{*}_{F}(db, db))g_{F}(W, V)g_{F}(U, Q)\nonumber\\
&+\frac{1}{4}(b^{2}g^{*}_{B}(db\circ J_{B}, db\circ J_{B})+g^{*}_{F}(db\circ J_{F}, db\circ J_{F}))g_{F}(W, J_{F}V)g_{F}(U, J_{F}Q)\nonumber\\
&+\frac{1}{4}(b^{2}g^{*}_{B}(db\circ J_{B}, db)+g^{*}_{F}(db\circ J_{F}, db))(g_{F}(W, V)g_{F}(U, J_{F}Q)+g_{F}(W, J_{F}V)g_{F}(U, Q))\nonumber\\
&+bV(b)\widetilde{\mathcal{K}}_{F}(U, Q, W)+\frac{1}{2}bW(b)\widetilde{\mathcal{K}}_{F}(U, Q, V)+bU(b)\widetilde{\mathcal{K}}_{F}(V, W, Q)\nonumber\\
&+\frac{1}{2}bQ(b)\widetilde{\mathcal{K}}_{F}(V, W, U)+\frac{1}{2}bJ_{F}W(b)\widetilde{\mathcal{K}}_{F}(U, Q, J_{F}V)+\frac{1}{2}bJ_{F}Q(b)\widetilde{\mathcal{K}}_{F}(V, W, J_{F}U)\nonumber\\
&+V(b)U(b)g_{F}(Q, W)+\frac{1}{2}V(b)Q(b)g_{F}(U, W)+\frac{1}{2}V(b)J_{F}Q(b)g_{F}(U, J_{F}W)\nonumber\\
&+\frac{1}{2}W(b)U(b)g_{F}(Q, V)+\frac{1}{2}J_{F}W(b)U(b)g_{F}(Q, J_{F}V)+b^{2}\widetilde{\mathcal{K}}_{F}(V, W, \bullet_{F})\widetilde{\mathcal{K}}_{F}(U, Q, \bullet_{F})\nonumber\\
&+\frac{1}{4}(W(b)Q(b)+J_{F}W(b)J_{F}Q(b))g_{F}(U, V)+\frac{1}{4}(W(b)J_{F}Q(b)+J_{F}W(b)Q(b))g_{F}(U, J_{F}V)\nonumber\\
&-\frac{1}{4}(3V(b)W(b)+J_{F}V(b)J_{F}W(b)+2bg^{*}_{F}(\widetilde{\nabla}^{\flat}_{V}W, db))g_{F}(U, Q)\nonumber\\
&-\frac{1}{4}(3V(b)J_{F}W(b)+J_{F}V(b)W(b)+2bg^{*}_{F}(\widetilde{\nabla}^{\flat}_{V}W, db\circ J_{F}))g_{F}(U, J_{F}Q)\nonumber\\
&-\frac{1}{4}(3U(b)Q(b)+J_{F}U(b)J_{F}Q(b)+2bg^{*}_{F}(\widetilde{\nabla}^{\flat}_{U}Q, db))g_{F}(V, W)\nonumber\\
&-\frac{1}{4}(3U(b)J_{F}Q(b)+J_{F}U(b)Q(b)+2bg^{*}_{F}(\widetilde{\nabla}^{\flat}_{U}Q, db\circ J_{F}))g_{F}(V, J_{F}W).\nonumber
\end{align}
Obviously $(B\times_{b}F, J)$ is a almost product semi-regular semi-Riemannian manifold.
\end{proof}

Computations show that
\begin{thm}
Let $(B, g_{B}, J_{B})$ and $(F, g_{F}, J_{F})$ be two almost product non-degenerate semi-Riemannian manifolds, $b\in C^{\infty}(B\times F)$ and $db|_{b=0}=0.$
Let the vector fields $X, Y, Z, T\in\Gamma(TB)$ and $U, V, W, Q\in\Gamma(TF),$ then
\begin{align}
&(1)\widetilde{R}(X, Y, Z, T)=\widetilde{R}_{B}(X, Y, Z, T);\nonumber\\
&(2)\widetilde{R}(X, Y, Z, W)=\widetilde{R}(Z, W, X, Y)=0;\nonumber\\
&(3)\widetilde{R}(X, Y,  V, W)=\widetilde{R}(V, W, X, Y)=0;\nonumber\\
&(4)\widetilde{R}(X, V, Y, W)=\frac{1}{2}bXY(b)g_{F}(V, W)-\frac{1}{2}bg^{*}_{B}(\widetilde{\nabla}^{\flat}_{X}Y, db)g_{F}(V, W)+\frac{1}{2}bXJ_{B}Y(b)\nonumber\\
&\times g_{F}(V, J_{F}W)-\frac{1}{2}bg^{*}_{B}(\widetilde{\nabla}^{\flat}_{X}Y, db\circ J_{B})g_{F}(V, J_{F}W);\nonumber\\
&(5)\widetilde{R}(X, V, W, U)=-\frac{1}{2}(X(b)W(b)-bXW(b))g_{F}(V, U)-\frac{1}{2}(X(b)J_{F}W(b)-bXJ_{F}W(b))\nonumber\\
&\times g_{F}(V, J_{F}U)+\frac{1}{2}(X(b)U(b)-bXU(b))g_{F}(V, W)+\frac{1}{2}(X(b)J_{F}U(b)-bXJ_{F}U(b))g_{F}(V, J_{F}W);\nonumber\\
&(6)\widetilde{R}(W, U, X, V)=\frac{1}{4}(2bWX(b)-X(b)W(b)-J_{B}X(b)J_{F}W(b))g_{F}(U, V)+\frac{1}{4}(2bWJ_{B}X(b)\nonumber\\
&-X(b)J_{F}W(b)-J_{B}X(b)W(b))g_{F}(U, J_{F}V)-\frac{1}{4}(2bUX(b)-X(b)U(b)-J_{B}X(b)J_{F}U(b))\nonumber\\
&\times g_{F}(W, V)-\frac{1}{4}(2bUJ_{B}X(b)-X(b)J_{F}U(b)-J_{B}X(b)U(b))g_{F}(W, J_{F}V)\nonumber\\
&-\frac{1}{4}bX(b)(\mathcal{K}_{F}(U, V, W)-\mathcal{K}_{F}(W, V, U)-\mathcal{K}_{F}(U, J_{F}V, J_{F}W)+\mathcal{K}_{F}(W, J_{F}V, J_{F}U))\nonumber\\
&+\frac{1}{4}bJ_{B}X(b)(\mathcal{K}_{F}(U, V, J_{F}W)-\mathcal{K}_{F}(W, V, J_{F}U)-\mathcal{K}_{F}(U, J_{F}V, W)+\mathcal{K}_{F}(W, J_{F}V, U));\nonumber\\
&(9)\widetilde{R}(U, V, W, Q)=b^{2}\widetilde{R}_{F}(U, V, W, Q)+\frac{1}{4}(b^{2}g^{*}_{B}(db, db)+g^{*}_{F}(db, db))(g_{F}(U, W)g_{F}(V, Q)\nonumber\\
&-g_{F}(V, W)g_{F}(U, Q))+\frac{1}{4}(b^{2}g^{*}_{B}(db\circ J_{B}, db\circ J_{B})+g^{*}_{F}(db\circ J_{F}, db\circ J_{F}))(g_{F}(U, J_{F}W)\nonumber\\
&\times g_{F}(V, J_{F}Q)-g_{F}(V, J_{F}W)g_{F}(U, J_{F}Q))+\frac{1}{4}(b^{2}g^{*}_{B}(db\circ J_{B}, db)+g^{*}_{F}(db\circ J_{F}, db))\nonumber\\
&\times (g_{F}(U, J_{F}W)g_{F}(V, Q)-g_{F}(V, J_{F}W)g_{F}(U, Q)+g_{F}(U, W)g_{F}(V, J_{F}Q)-g_{F}(V, W)\nonumber\\
&\times g_{F}(U, J_{F}Q))-\frac{1}{4}(3U(b)W(b)-2bUW(b)+J_{F}U(b)J_{F}W(b)+2bg^{*}_{F}(\widetilde{\nabla}^{\flat}_{U}W, db))g_{F}(V, Q)\nonumber\\
&-\frac{1}{4}(3U(b)J_{F}W(b)-2bUJ_{F}W(b)+J_{F}U(b)W(b)+2bg^{*}_{F}(\widetilde{\nabla}^{\flat}_{U}W, db\circ J_{F}))g_{F}(V, J_{F}Q)\nonumber\\
&-\frac{1}{4}(3V(b)Q(b)-2bVQ(b)+J_{F}V(b)J_{F}Q(b)+2bg^{*}_{F}(\widetilde{\nabla}^{\flat}_{V}Q, db))g_{F}(U, W)\nonumber\\
&-\frac{1}{4}(3V(b)J_{F}Q(b)-2bVJ_{F}Q(b)+J_{F}V(b)Q(b)+2bg^{*}_{F}(\widetilde{\nabla}^{\flat}_{V}Q, db\circ J_{F}))g_{F}(U, J_{F}W)\nonumber\\
&+\frac{1}{4}(3V(b)W(b)-2bVW(b)+J_{F}V(b)J_{F}W(b)+2bg^{*}_{F}(\widetilde{\nabla}^{\flat}_{V}W, db))g_{F}(U, Q)\nonumber
\end{align}
\begin{align}
&+\frac{1}{4}(3V(b)J_{F}W(b)-2bVJ_{F}W(b)+J_{F}V(b)W(b)+2bg^{*}_{F}(\widetilde{\nabla}^{\flat}_{V}W, db\circ J_{F}))g_{F}(U, J_{F}Q)\nonumber\\
&+\frac{1}{4}(3U(b)Q(b)-2bUQ(b)+J_{F}U(b)J_{F}Q(b)+2bg^{*}_{F}(\widetilde{\nabla}^{\flat}_{U}Q, db))g_{F}(V, W)\nonumber\\
&+\frac{1}{4}(3U(b)J_{F}Q(b)-2bUJ_{F}Q(b)+J_{F}U(b)Q(b)+2bg^{*}_{F}(\widetilde{\nabla}^{\flat}_{U}Q, db\circ J_{F}))g_{F}(V, J_{F}W)\nonumber\\
&+\frac{1}{4}bQ(b)(\mathcal{K}_{F}(V, W, U)-\mathcal{K}_{F}(U, W, V)-\mathcal{K}_{F}(V, J_{F}W, J_{F}U)+\mathcal{K}_{F}(U, J_{F}W, J_{F}V))\nonumber\\
&-\frac{1}{4}bJ_{F}Q(b)(\mathcal{K}_{F}(V, W, J_{F}U)-\mathcal{K}_{F}(U, W, J_{F}V)-\mathcal{K}_{F}(V, J_{F}W, U)+\mathcal{K}_{F}(U, J_{F}W, V))\nonumber\\
&-\frac{1}{4}bW(b)(\mathcal{K}_{F}(V, Q, U)-\mathcal{K}_{F}(U, Q, V)-\mathcal{K}_{F}(V, J_{F}Q, J_{F}U)+\mathcal{K}_{F}(U, J_{F}Q, J_{F}V))\nonumber\\
&+\frac{1}{4}bJ_{F}W(b)(\mathcal{K}_{F}(V, Q, J_{F}U)-\mathcal{K}_{F}(U, Q, J_{F}V)-\mathcal{K}_{F}(V, J_{F}Q, U)+\mathcal{K}_{F}(U, J_{F}Q, V)).\nonumber
\end{align}
\end{thm}

\vskip 1 true cm
\section{ Special Singular Multiply Warped Products and Special Singular Twisted Products}

A generalized Kasner space-time $(M, g)$ is a Lorentzian multiply warped product of the form $M=I\times_{\phi^{p_{1}}}F_{1}\times_{\phi^{p_{2}}}F_{2}\times\cdots\times_{\phi^{p_{m}}}F_{m}$ with the metric tensor $g=-dt^{2}\oplus \phi^{2p_{1}}g_{F_{1}}\oplus \phi^{2p_{2}}g_{F_{2}}\cdots\oplus \phi^{2p_{m}}g_{F_{m}},$ where $I$ is an open interval in $\mathbb{R},$ $\phi: I\rightarrow \mathbb{R}$ is smooth and $p_{j}\in\mathbb{R}$ for any $j\in \{1, 2, \cdots, m\}.$
For simplicity of notations, we denote $M_{1}=I\times_{b}F$ to be a degenerate warped product with the metric tensor $g=-dt^{2}\oplus b^{2}g_{F}$ where $\dim F =3,$ denote $I\times_{\phi^{p_{1}}}F_{1}\times_{\phi^{p_{2}}}F_{2}$ to be $M_{2}$ where $\dim F_{1}=1,$ $\dim F_{2}=2,$ and denote $I\times_{\phi^{p_{1}}}F_{1}\times_{\phi^{p_{2}}}F_{2}\times_{\phi^{p_{3}}}F_{3}$ to be $M_{3}$ where $\dim F_{1}=1,$ $\dim F_{2}=1,$ $\dim F_{3}=1.$

\begin{prop}
Let $M_{1}=I\times_{b}F$ be a degenerate warped product, let the vector fields $\partial/ \partial t\in\Gamma(TI)$ and $U, V, W\in\Gamma(TF).$ Let $\overline{\mathcal{K}}$ be the semi-symmetric metric Koszul form on $I\times_{b}F$ and $\overline{\mathcal{K}}_{I}$, $\overline{\mathcal{K}}_{F}$ be the lifts of the semi-symmetric metric Koszul form on $I,$ $F,$ respectively. Let $P=\partial/ \partial t,$ then
\begin{align}
&(1)\overline{\mathcal{K}}(\frac{\partial}{\partial t}, \frac{\partial}{\partial t}, \frac{\partial}{\partial t})=\overline{\mathcal{K}}_{I}(\frac{\partial}{\partial t}, \frac{\partial}{\partial t}, \frac{\partial}{\partial t})=0;\nonumber\\
&(2)\overline{\mathcal{K}}(\frac{\partial}{\partial t}, \frac{\partial}{\partial t}, W)=\overline{\mathcal{K}}(\frac{\partial}{\partial t}, W, \frac{\partial}{\partial t})=\overline{\mathcal{K}}(W, \frac{\partial}{\partial t}, \frac{\partial}{\partial t})=0;\nonumber\\
&(3)\overline{\mathcal{K}}(\frac{\partial}{\partial t}, V, W)=b\frac{\partial b}{\partial t}g_{F}(V, W);\nonumber\\
&(4)\overline{\mathcal{K}}(V, \frac{\partial}{\partial t}, W)=-\overline{\mathcal{K}}(V, W, \frac{\partial}{\partial t})=b\frac{\partial b}{\partial t}g_{F}(V, W)-b^{2}g_{F}(V, W);\nonumber\\
&(5)\overline{\mathcal{K}}(U, V, W)=b^{2}\mathcal{K}_{F}(U, V, W).\nonumber
\end{align}
\end{prop}

\begin{prop}
Let $M_{1}=I\times_{b}F$ be a degenerate warped product, let the vector fields $\partial/ \partial t\in\Gamma(TI)$ and $U, V, W\in\Gamma(TF).$ Let $\overline{\mathcal{K}}$ be the semi-symmetric metric Koszul form on $I\times_{b}F$ and $\overline{\mathcal{K}}_{I}$, $\overline{\mathcal{K}}_{F}$ be the lifts of the semi-symmetric metric Koszul form on $I,$ $F,$ respectively. Let $P\in\Gamma(TF),$ then
\begin{align}
&(1)\overline{\mathcal{K}}(\frac{\partial}{\partial t}, \frac{\partial}{\partial t}, \frac{\partial}{\partial t})=\mathcal{K}_{I}(\frac{\partial}{\partial t}, \frac{\partial}{\partial t}, \frac{\partial}{\partial t})=0;\nonumber\\
&(2)\overline{\mathcal{K}}(\frac{\partial}{\partial t}, \frac{\partial}{\partial t}, W)=-\overline{\mathcal{K}}(\frac{\partial}{\partial t}, W, \frac{\partial}{\partial t})=b^{2}g_{F}(W, P);\nonumber\\
&(3)\overline{\mathcal{K}}(W, \frac{\partial}{\partial t}, \frac{\partial}{\partial t})=0;\nonumber\\
&(4)\overline{\mathcal{K}}(\frac{\partial}{\partial t}, V, W)=\overline{\mathcal{K}}(V, \frac{\partial}{\partial t}, W)=-\overline{\mathcal{K}}(V, W, \frac{\partial}{\partial t})=b\frac{\partial b}{\partial t}g_{F}(V, W);\nonumber\\
&(5)\overline{\mathcal{K}}(U, V, W)=b^{2}\mathcal{K}_{F}(U, V, W)+b^{4}g_{F}(U, W)g_{F}(V, P)-b^{4}g_{F}(U, V)g_{F}(W, P).\nonumber
\end{align}
\end{prop}

\begin{thm}
Let $M_{1}=I\times_{b}F$ be a degenerate warped product, let the vector fields $\partial/ \partial t\in\Gamma(TI)$ and $U, V, W, Q\in\Gamma(TF).$  Let $P=\partial/ \partial t,$ then
\begin{align}
&(1)\overline{R}(\frac{\partial}{\partial t}, \frac{\partial}{\partial t}, \frac{\partial}{\partial t}, \frac{\partial}{\partial t})=\overline{R}_{I}(\frac{\partial}{\partial t}, \frac{\partial}{\partial t}, \frac{\partial}{\partial t}, \frac{\partial}{\partial t})=0;\nonumber\\
&(2)\overline{R}(\frac{\partial}{\partial t}, \frac{\partial}{\partial t}, \frac{\partial}{\partial t}, W)=\overline{R}(\frac{\partial}{\partial t}, W, \frac{\partial}{\partial t}, \frac{\partial}{\partial t})=0;\nonumber\\
&(3)\overline{R}(\frac{\partial}{\partial t}, \frac{\partial}{\partial t},  V, W)=\overline{R}(V, W, \frac{\partial}{\partial t}, \frac{\partial}{\partial t})=0;\nonumber\\
&(4)\overline{R}(\frac{\partial}{\partial t}, V, W, \frac{\partial}{\partial t})=-b\frac{\partial^{2} b}{\partial t^{2}}g_{F}(V, W)+b\frac{\partial b}{\partial t}g_{F}(V, W);\nonumber\\
&(5)\overline{R}(\frac{\partial}{\partial t}, V, W, U)=\overline{R}(W, U, \frac{\partial}{\partial t}, V)=0;\nonumber\\
&(6)\overline{R}(U, V, W, Q)=b^{2}R_{F}(U, V, W, Q)+(b^{2}g^{*}_{I}(db, db)+2b^{3}\frac{\partial b}{\partial t}-b^{4})\nonumber\\
&\times(g_{F}(U, W)g_{F}(V, Q)-g_{F}(V, W)g_{F}(U, Q)).\nonumber
\end{align}
\end{thm}

\begin{thm}
Let $M_{1}=I\times_{b}F$ be a degenerate warped product, let the vector fields $\partial/ \partial t\in\Gamma(TI)$ and $U, V, W, Q\in\Gamma(TF).$  Let $P\in\Gamma(TF),$ then
\begin{align}
&(1)\overline{R}(\frac{\partial}{\partial t}, \frac{\partial}{\partial t}, \frac{\partial}{\partial t}, \frac{\partial}{\partial t})=0;\nonumber\\
&(2)\overline{R}(\frac{\partial}{\partial t}, \frac{\partial}{\partial t}, \frac{\partial}{\partial t}, W)=-\overline{R}(\frac{\partial}{\partial t}, W, \frac{\partial}{\partial t}, \frac{\partial}{\partial t})=0;\nonumber\\
&(3)\overline{R}(\frac{\partial}{\partial t}, \frac{\partial}{\partial t},  V, W)=\overline{R}(V, W, \frac{\partial}{\partial t}, \frac{\partial}{\partial t})=0;\nonumber\\
&(4)\overline{R}(\frac{\partial}{\partial t}, V, W, \frac{\partial}{\partial t})=-b\frac{\partial^{2} b}{\partial t^{2}}g_{F}(V, W)+b^{2}(\mathcal{K}_{F}(V, P, W)+b^{2}g_{F}(V, W)g_{F}(P, P)\nonumber\\
&-b^{2}g_{F}(V, P)g_{F}(W, P));\nonumber
\end{align}
\begin{align}
&(5)\overline{R}(\frac{\partial}{\partial t}, V, W, U)=-\overline{R}(W, U, \frac{\partial}{\partial t}, V)=b^{3}\frac{\partial b}{\partial t}(g_{F}(V, U)g_{F}(W, P)-g_{F}(V, W)g_{F}(U, P));\nonumber\\
&(6)\overline{R}(U, V, W, Q)=b^{2}R_{F}(U, V, W, Q)+b^{2}g^{*}_{I}(db, db)(g_{F}(U, W)g_{F}(V, Q)-g_{F}(V, W)g_{F}(U, Q))\nonumber\\
&+b^{4}g_{F}(V, Q)\mathcal{K}_{F}(U, P, W)-b^{4}g_{F}(U, Q)\mathcal{K}_{F}(V, P, W)-b^{4}g_{F}(V, W)\mathcal{K}_{F}(U, P, Q)\nonumber\\
&+b^{4}g_{F}(U, W)\mathcal{K}_{F}(V, P, Q)+b^{6}g_{F}(U, P)(g_{F}(V, W)g_{F}(Q, P)-g_{F}(V, Q)g_{F}(W, P))\nonumber\\
&-b^{6}g_{F}(V, P)(g_{F}(U, W)g_{F}(Q, P)-g_{F}(U, Q)g_{F}(W, P)).\nonumber
\end{align}
\end{thm}

\begin{prop}
Let $M_{1}=I\times_{b}F$ be a degenerate warped product, let the vector fields $\partial/ \partial t\in\Gamma(TI)$ and $U, V, W\in\Gamma(TF).$ Let $\widehat{\mathcal{K}}$ be the semi-symmetric non-metric Koszul form on $I\times_{b}F$ and $\widehat{\mathcal{K}}_{I}$, $\widehat{\mathcal{K}}_{F}$ be the lifts of the semi-symmetric non-metric Koszul form on $B,$ $F$ respectively. Let $P=\partial/ \partial t,$ then
\begin{align}
&(1)\widehat{\mathcal{K}}(\frac{\partial}{\partial t}, \frac{\partial}{\partial t}, \frac{\partial}{\partial t})=\widehat{\mathcal{K}}_{I}(\frac{\partial}{\partial t}, \frac{\partial}{\partial t}, \frac{\partial}{\partial t})=1;\nonumber\\
&(2)\widehat{\mathcal{K}}(\frac{\partial}{\partial t}, \frac{\partial}{\partial t}, W)=\widehat{\mathcal{K}}(\frac{\partial}{\partial t}, W, \frac{\partial}{\partial t})=\widehat{\mathcal{K}}(W, \frac{\partial}{\partial t}, \frac{\partial}{\partial t})=0;\nonumber\\
&(3)\widehat{\mathcal{K}}(\frac{\partial}{\partial t}, V, W)=-\widehat{\mathcal{K}}(V, W, \frac{\partial}{\partial t})=b\frac{\partial b}{\partial t}g_{F}(V, W);\nonumber\\
&(4)\widehat{\mathcal{K}}(V, \frac{\partial}{\partial t}, W)=b\frac{\partial b}{\partial t}g_{F}(V, W)-b^{2}g_{F}(V, W);\nonumber\\
&(5)\widehat{\mathcal{K}}(U, V, W)=b^{2}\mathcal{K}_{F}(U, V, W).\nonumber
\end{align}
\end{prop}

\begin{prop}
Let $M_{1}=I\times_{b}F$ be a degenerate warped product, let the vector fields $\partial/ \partial t\in\Gamma(TI)$ and $U, V, W\in\Gamma(TF).$ Let $\widehat{\mathcal{K}}$ be the semi-symmetric non-metric Koszul form on $I\times_{b}F$ and $\widehat{\mathcal{K}}_{I}$, $\widehat{\mathcal{K}}_{F}$ be the lifts of the semi-symmetric non-metric Koszul form on $B,$ $F$ respectively. Let $P\in\Gamma(TF),$ then
\begin{align}
&(1)\widehat{\mathcal{K}}(\frac{\partial}{\partial t}, \frac{\partial}{\partial t}, \frac{\partial}{\partial t})=\mathcal{K}_{I}(\frac{\partial}{\partial t}, \frac{\partial}{\partial t}, \frac{\partial}{\partial t})=0;\nonumber\\
&(2)\widehat{\mathcal{K}}(\frac{\partial}{\partial t}, \frac{\partial}{\partial t}, W)=\widehat{\mathcal{K}}(W, \frac{\partial}{\partial t}, \frac{\partial}{\partial t})=0;\nonumber\\
&(3)\widehat{\mathcal{K}}(\frac{\partial}{\partial t}, W, \frac{\partial}{\partial t})=-b^{2}g_{F}(W, P);\nonumber\\
&(4)\widehat{\mathcal{K}}(\frac{\partial}{\partial t}, V, W)=\widehat{\mathcal{K}}(V, \frac{\partial}{\partial t}, W)=-\widehat{\mathcal{K}}(V, W, \frac{\partial}{\partial t})=b\frac{\partial b}{\partial t}g_{F}(V, W);\nonumber\\
&(5)\widehat{\mathcal{K}}(U, V, W)=b^{2}\mathcal{K}_{F}(U, V, W)+b^{4}g_{F}(U, W)g_{F}(V, P).\nonumber
\end{align}
\end{prop}

\begin{thm}
Let $M_{1}=I\times_{b}F$ be a degenerate warped product, let the vector fields $\partial/ \partial t\in\Gamma(TI)$ and $U, V, W, Q\in\Gamma(TF).$  Let $P=\partial/ \partial t,$ then
\begin{align}
&(1)\widehat{R}(\frac{\partial}{\partial t}, \frac{\partial}{\partial t}, \frac{\partial}{\partial t}, \frac{\partial}{\partial t})=\widehat{R}_{I}(\frac{\partial}{\partial t}, \frac{\partial}{\partial t}, \frac{\partial}{\partial t}, \frac{\partial}{\partial t})=0;\nonumber\\
&(2)\widehat{R}(\frac{\partial}{\partial t}, \frac{\partial}{\partial t}, \frac{\partial}{\partial t}, W)=\widehat{R}(\frac{\partial}{\partial t}, \frac{\partial}{\partial t}, W, \frac{\partial}{\partial t})=\widehat{R}(\frac{\partial}{\partial t}, W, \frac{\partial}{\partial t}, \frac{\partial}{\partial t})=0;\nonumber\\
&(3)\widehat{R}(\frac{\partial}{\partial t}, \frac{\partial}{\partial t},  V, W)=\widehat{R}(V, W, \frac{\partial}{\partial t}, \frac{\partial}{\partial t})=0;\nonumber\\
&(4)\widehat{R}(V, \frac{\partial}{\partial t}, W, \frac{\partial}{\partial t})=-\widehat{R}(\frac{\partial}{\partial t}, V, W, \frac{\partial}{\partial t})=b\frac{\partial^{2} b}{\partial t^{2}}g_{F}(V, W)+2b\frac{\partial b}{\partial t}g_{F}(V, W);\nonumber\\
&(5)\widehat{R}(V, \frac{\partial}{\partial t}, \frac{\partial}{\partial t}, W)=-\widehat{R}(\frac{\partial}{\partial t}, V, \frac{\partial}{\partial t}, W)=-b\frac{\partial^{2} b}{\partial t^{2}}g_{F}(V, W);\nonumber\\
&(6)\widehat{R}(\frac{\partial}{\partial t}, V, W, U)=\widehat{R}(W, U, \frac{\partial}{\partial t}, V)=0;\nonumber\\
&(7)\widehat{R}(W, U, V, \frac{\partial}{\partial t})=-b^{2}(\mathcal{K}_{F}(W, V, U)-\mathcal{K}_{F}(U, V, W));\nonumber\\
&(8)\widehat{R}(U, V, W, Q)=b^{2}R_{F}(U, V, W, Q)+b^{2}g^{*}_{I}(db, db)(g_{F}(U, W)g_{F}(V, Q)-g_{F}(V, W)g_{F}(U, Q)).\nonumber
\end{align}
\end{thm}

\begin{thm}
Let $M_{1}=I\times_{b}F$ be a degenerate warped product, let the vector fields $\partial/ \partial t\in\Gamma(TI)$ and $U, V, W, Q\in\Gamma(TF).$  Let $P\in\Gamma(TF),$ then
\begin{align}
&(1)\widehat{R}(\frac{\partial}{\partial t}, \frac{\partial}{\partial t}, \frac{\partial}{\partial t}, \frac{\partial}{\partial t})=R_{I}(\frac{\partial}{\partial t}, \frac{\partial}{\partial t}, \frac{\partial}{\partial t}, \frac{\partial}{\partial t})=0;\nonumber\\
&(2)\widehat{R}(\frac{\partial}{\partial t}, \frac{\partial}{\partial t}, \frac{\partial}{\partial t}, W)=\widehat{R}(\frac{\partial}{\partial t}, \frac{\partial}{\partial t}, W, \frac{\partial}{\partial t})=\widehat{R}(\frac{\partial}{\partial t}, W, \frac{\partial}{\partial t}, \frac{\partial}{\partial t})=0;\nonumber\\
&(3)\widehat{R}(\frac{\partial}{\partial t}, \frac{\partial}{\partial t},  V, W)=\widehat{R}(V, W, \frac{\partial}{\partial t}, \frac{\partial}{\partial t})=0;\nonumber\\
&(4)\widehat{R}(V, \frac{\partial}{\partial t}, W, \frac{\partial}{\partial t})=-\widehat{R}(\frac{\partial}{\partial t}, V, W, \frac{\partial}{\partial t})=b\frac{\partial^{2} b}{\partial t^{2}}g_{F}(V, W)-b^{2}V(g_{F}(W, P));\nonumber\\
&(5)\widehat{R}(V, \frac{\partial}{\partial t}, \frac{\partial}{\partial t}, W)=-\widehat{R}(\frac{\partial}{\partial t}, V, \frac{\partial}{\partial t}, W)=-b\frac{\partial^{2} b}{\partial t^{2}}g_{F}(V, W);\nonumber\\\
&(6)\widehat{R}(\frac{\partial}{\partial t}, V, W, U)=2b^{3}\frac{\partial b}{\partial t}(g_{F}(V, U)g_{F}(W, P)+g_{F}(V, W)g_{F}(U, P));\nonumber\\
&(7)\widehat{R}(W, U, \frac{\partial}{\partial t}, V)=\widehat{R}(W, U, V, \frac{\partial}{\partial t})=0;\nonumber\\
&(8)\widehat{R}(U, V, W, Q)=b^{2}R_{F}(U, V, W, Q)+b^{2}g^{*}_{I}(db, db)(g_{F}(U, W)g_{F}(V, Q)\nonumber\\
&-g_{F}(V, W)g_{F}(U, Q))+b^{4}g_{F}(Q, P)(\mathcal{K}_{F}(U, W, V)-\mathcal{K}_{F}(V, W, U))\nonumber\\
&+b^{4}U(g_{F}(W, P))g_{F}(V, Q)-b^{4}V(g_{F}(W, P))g_{F}(U, Q).\nonumber
\end{align}
\end{thm}

\begin{prop}
Let $M_{1}=I\times_{b}F$ be a degenerate warped product, let the vector fields $\partial/ \partial t\in\Gamma(TI)$ and $U, V, W\in\Gamma(TF).$ Let  $\widetilde{\mathcal{K}}$ be the almost product Koszul form on $I\times_{b}F$ and $\widetilde{\mathcal{K}}_{I}$, $\widetilde{\mathcal{K}}_{F}$ be the lifts of the almost product Koszul form on $I,$ $F,$ respectively. Then
\begin{align}
&(1)\widetilde{\mathcal{K}}(\frac{\partial}{\partial t}, \frac{\partial}{\partial t}, \frac{\partial}{\partial t})=\widetilde{\mathcal{K}}_{I}(\frac{\partial}{\partial t}, \frac{\partial}{\partial t}, \frac{\partial}{\partial t})=0;\nonumber\\
&(2)\widetilde{\mathcal{K}}(\frac{\partial}{\partial t}, \frac{\partial}{\partial t}, W)=\widetilde{\mathcal{K}}(\frac{\partial}{\partial t}, W, \frac{\partial}{\partial t})=\widetilde{\mathcal{K}}(W, \frac{\partial}{\partial t}, \frac{\partial}{\partial t})=0;\nonumber\\
&(3)\widetilde{\mathcal{K}}(\frac{\partial}{\partial t}, V, W)=b\frac{\partial b}{\partial t}g_{F}(V, W);\nonumber\\
&(4)\widetilde{\mathcal{K}}(V, \frac{\partial}{\partial t}, W)=-\widetilde{\mathcal{K}}(V, W, \frac{\partial}{\partial t})=\frac{1}{2}(b\frac{\partial b}{\partial t}g_{F}(V, W)+bJ_{I}(\frac{\partial b}{\partial t})g_{F}(V, J_{F}W));\nonumber\\
&(5)\widetilde{\mathcal{K}}(U, V, W)=b^{2}\widetilde{\mathcal{K}}_{F}(U, V, W).\nonumber
\end{align}
\end{prop}

\begin{thm}
Let $M_{1}=I\times_{b}F$ be a degenerate warped product, let the vector fields $\partial/ \partial t\in\Gamma(TI)$ and $U, V, W, Q\in\Gamma(TF).$  Then
\begin{align}
&(1)\widetilde{R}(\frac{\partial}{\partial t}, \frac{\partial}{\partial t}, \frac{\partial}{\partial t}, \frac{\partial}{\partial t})=\widetilde{R}_{I}(\frac{\partial}{\partial t}, \frac{\partial}{\partial t}, \frac{\partial}{\partial t}, \frac{\partial}{\partial t})=0;\nonumber\\
&(2)\widetilde{R}(\frac{\partial}{\partial t}, \frac{\partial}{\partial t}, \frac{\partial}{\partial t}, W)=\widetilde{R}(\frac{\partial}{\partial t}, W, \frac{\partial}{\partial t}, \frac{\partial}{\partial t})=0;\nonumber\\
&(3)\widetilde{R}(\frac{\partial}{\partial t}, \frac{\partial}{\partial t},  V, W)=\widetilde{R}(V, W, \frac{\partial}{\partial t}, \frac{\partial}{\partial t})=0;\nonumber\\
&(4)\widetilde{R}(\frac{\partial}{\partial t}, V, \frac{\partial}{\partial t}, W)=\frac{b}{2}\frac{\partial^{2} b}{\partial t^{2}}g_{F}(V, W)+\frac{b}{2}\frac{\partial}{\partial t}(J_{I}(\frac{\partial b}{\partial t}))g_{F}(V, J_{F}W);\nonumber\\
&(5)\widetilde{R}(\frac{\partial}{\partial t}, V, W, U)=0;\nonumber\\
&(6)\widetilde{R}(U, V, \frac{\partial}{\partial t}, W)=\frac{b}{4}\frac{\partial b}{\partial t}(-\mathcal{K}_{F}(V, W, U)+\mathcal{K}_{F}(U, W, V)+\mathcal{K}_{F}(V, J_{F}W, J_{F}U)\nonumber\\
&-\mathcal{K}_{F}(U, J_{F}W, J_{F}V))-\frac{b}{4}J_{I}(\frac{\partial b}{\partial t})(-\mathcal{K}_{F}(V, W, J_{F}U)+\mathcal{K}_{F}(U, W, J_{F}V)+\mathcal{K}_{F}(V, J_{F}W, U)\nonumber\\
&-\mathcal{K}_{F}(U, J_{F}W, V));\nonumber\\
&(7)\widetilde{R}(U, V, W, Q)=b^{2}\widetilde{R}_{F}(U, V, W, Q)+\frac{1}{4}b^{2}[g^{*}_{I}(db, db)(g_{F}(U, W)g_{F}(V, Q)\nonumber\\
&-g_{F}(V, W)g_{F}(U, Q))+g^{*}_{I}(db\circ J_{I}, db)(g_{F}(U, J_{F}W)g_{F}(V, Q)-g_{F}(V, J_{F}W)g_{F}(U, Q)\nonumber\\
&+g_{F}(U, W)g_{F}(V, J_{F}Q)-g_{F}(V, W)g_{F}(U, J_{F}Q))+g^{*}_{I}(db\circ J_{I}, db\circ J_{I})(g_{F}(U, J_{F}W)\nonumber\\
&\times g_{F}(V, J_{F}Q)-g_{F}(V, J_{F}W)g_{F}(U, J_{F}Q))].\nonumber
\end{align}
\end{thm}

\begin{prop}
Let $M_{2}=I\times_{\phi^{p_{1}}}F_{1}\times_{\phi^{p_{2}}}F_{2}$ be a degenerate multiply warped product, let the vector fields $\partial/ \partial t\in\Gamma(TI)$ and $U_{j}, V_{j}, W_{j}\in\Gamma(TF_{j}),$ for $j\in \{1, 2\}.$ Let $\overline{\mathcal{K}}$ be the semi-symmetric metric Koszul form on $I\times_{\phi^{p_{1}}}F_{1}\times_{\phi^{p_{2}}}F_{2}$ and $\overline{\mathcal{K}}_{I}$, $\overline{\mathcal{K}}_{F_{j}}$ be the lifts of the semi-symmetric metric Koszul form on $I,$ $F_{j},$ respectively. Let $P=\partial/ \partial t,$ then
\begin{align}
&(1)\overline{\mathcal{K}}(\frac{\partial}{\partial t}, \frac{\partial}{\partial t}, \frac{\partial}{\partial t})=\overline{\mathcal{K}}_{I}(\frac{\partial}{\partial t}, \frac{\partial}{\partial t}, \frac{\partial}{\partial t})=0;\nonumber\\
&(2)\overline{\mathcal{K}}(\frac{\partial}{\partial t}, \frac{\partial}{\partial t}, W_{j})=\overline{\mathcal{K}}(\frac{\partial}{\partial t}, W_{j}, \frac{\partial}{\partial t})=\overline{\mathcal{K}}(W_{j}, \frac{\partial}{\partial t}, \frac{\partial}{\partial t})=0;\nonumber\\
&(3)\overline{\mathcal{K}}(\frac{\partial}{\partial t}, V_{i}, W_{j})=\phi^{p_{j}}\frac{\partial \phi^{p_{j}}}{\partial t}g_{F_{j}}(V_{j}, W_{j}),~if~i=j;\nonumber\\
&(4)\overline{\mathcal{K}}(V_{i}, \frac{\partial}{\partial t}, W_{j})=-\overline{\mathcal{K}}(V_{i}, W_{j}, \frac{\partial}{\partial t})=\phi^{p_{j}}\frac{\partial \phi^{p_{j}}}{\partial t}g_{F_{j}}(V_{j}, W_{j})-\phi^{2p_{j}}g_{F_{j}}(V_{j}, W_{j}),~if~i=j;\nonumber\\
&(5)\overline{\mathcal{K}}(\frac{\partial}{\partial t}, V_{i}, W_{j})=\overline{\mathcal{K}}(V_{i}, \frac{\partial}{\partial t}, W_{j})=\overline{\mathcal{K}}(V_{i}, W_{j}, \frac{\partial}{\partial t})=0,~if~i\neq j;\nonumber\\
&(6)\overline{\mathcal{K}}(U_{i}, V_{j}, W_{k})=\phi^{2p_{j}}\mathcal{K}_{F_{j}}(U_{j}, V_{j}, W_{j}),~if~i=j=k;\nonumber\\
&(7)\overline{\mathcal{K}}(U_{i}, V_{j}, W_{k})=0,~other~cases.\nonumber
\end{align}
\end{prop}

\begin{prop}
Let $M_{2}=I\times_{\phi^{p_{1}}}F_{1}\times_{\phi^{p_{2}}}F_{2}$ be a degenerate multiply warped product, let the vector fields $\partial/ \partial t\in\Gamma(TI)$ and $U_{j}, V_{j}, W_{j}\in\Gamma(TF_{j}),$ for $j\in \{1, 2\}.$ Let $\overline{\mathcal{K}}$ be the semi-symmetric metric Koszul form on $I\times_{\phi^{p_{1}}}F_{1}\times_{\phi^{p_{2}}}F_{2}$ and $\overline{\mathcal{K}}_{I}$, $\overline{\mathcal{K}}_{F_{j}}$ be the lifts of the semi-symmetric metric Koszul form on $I,$ $F_{j},$ respectively. Let $P\in\Gamma(TF_{l}),$ then
\begin{align}
&(1)\overline{\mathcal{K}}(\frac{\partial}{\partial t}, \frac{\partial}{\partial t}, \frac{\partial}{\partial t})=\mathcal{K}_{I}(\frac{\partial}{\partial t}, \frac{\partial}{\partial t}, \frac{\partial}{\partial t})=0;\nonumber\\
&(2)\overline{\mathcal{K}}(\frac{\partial}{\partial t}, \frac{\partial}{\partial t}, W_{j})=-\overline{\mathcal{K}}(\frac{\partial}{\partial t}, W_{j}, \frac{\partial}{\partial t})=\phi^{2p_{j}}g_{F_{j}}(W_{j}, P),~if~j=l;\nonumber\\
&(3)\overline{\mathcal{K}}(\frac{\partial}{\partial t}, \frac{\partial}{\partial t}, W_{j})=\overline{\mathcal{K}}(\frac{\partial}{\partial t}, W_{j}, \frac{\partial}{\partial t})=0,~if~j\neq l;\nonumber\\
&(4)\overline{\mathcal{K}}(W_{j}, \frac{\partial}{\partial t}, \frac{\partial}{\partial t})=0;\nonumber\\
&(5)\overline{\mathcal{K}}(\frac{\partial}{\partial t}, V_{i}, W_{j})=\overline{\mathcal{K}}(V_{i}, \frac{\partial}{\partial t}, W_{j})=-\overline{\mathcal{K}}(V_{i}, W_{j}, \frac{\partial}{\partial t})=\phi^{p_{j}}\frac{\partial \phi^{p_{j}}}{\partial t}g_{F_{j}}(V_{j}, W_{j}),~if~i=j;\nonumber\\
&(6)\overline{\mathcal{K}}(\frac{\partial}{\partial t}, V_{i}, W_{j})=\overline{\mathcal{K}}(V_{i}, \frac{\partial}{\partial t}, W_{j})=\overline{\mathcal{K}}(V_{i}, W_{j}, \frac{\partial}{\partial t})=0,~if~i\neq j;\nonumber\\
&(7)\overline{\mathcal{K}}(U_{i}, V_{j}, W_{k})=\phi^{2p_{j}}\mathcal{K}_{F_{j}}(U_{j}, V_{j}, W_{j})+\phi^{4p_{j}}g_{F_{j}}(U_{j}, W_{j})g_{F_{j}}(V_{j}, P)-\phi^{4p_{j}}g_{F_{j}}(U_{j}, V_{j})\nonumber\\
&\times g_{F_{j}}(W_{j}, P),~if~i=j=k=l;\nonumber\\
&(8)\overline{\mathcal{K}}(U_{i}, V_{j}, W_{k})=-\overline{\mathcal{K}}(U_{i}, W_{k}, V_{j})=-\phi^{2p_{j}}\phi^{2p_{k}}g_{F_{j}}(U_{j}, V_{j})g_{F_{k}}(W_{k}, P),~if~i=j\neq k=l;\nonumber\\
&(9)\overline{\mathcal{K}}(U_{i}, V_{j}, W_{k})=0,~other~cases.\nonumber
\end{align}
\end{prop}

\begin{thm}
Let $M_{2}=I\times_{\phi^{p_{1}}}F_{1}\times_{\phi^{p_{2}}}F_{2}$ be a degenerate multiply warped product, let the vector fields $\partial/ \partial t\in\Gamma(TI)$ and $U_{j}, V_{j}, W_{j}, Q_{j}\in\Gamma(TF_{j}).$  Let $P=\partial/ \partial t,$ then
\begin{align}
&(1)\overline{R}(\frac{\partial}{\partial t}, \frac{\partial}{\partial t}, \frac{\partial}{\partial t}, \frac{\partial}{\partial t})=\overline{R}_{I}(\frac{\partial}{\partial t}, \frac{\partial}{\partial t}, \frac{\partial}{\partial t}, \frac{\partial}{\partial t})=0;\nonumber\\
&(2)\overline{R}(\frac{\partial}{\partial t}, \frac{\partial}{\partial t}, \frac{\partial}{\partial t}, W_{j})=\overline{R}(\frac{\partial}{\partial t}, W_{j}, \frac{\partial}{\partial t}, \frac{\partial}{\partial t})=0;\nonumber\\
&(3)\overline{R}(\frac{\partial}{\partial t}, \frac{\partial}{\partial t},  V_{i}, W_{j})=\overline{R}(V_{i}, W_{j}, \frac{\partial}{\partial t}, \frac{\partial}{\partial t})=0,~if~i=j;\nonumber\\
&(4)\overline{R}(\frac{\partial}{\partial t}, V_{i}, W_{j}, \frac{\partial}{\partial t})=-\phi^{p_{j}}\frac{\partial^{2} \phi^{p_{j}}}{\partial t^{2}}g_{F_{j}}(V_{j}, W_{j})+\phi^{p_{j}}\frac{\partial \phi^{p_{j}}}{\partial t}g_{F_{j}}(V_{j}, W_{j}),~if~i=j;\nonumber\\
&(5)\overline{R}(\frac{\partial}{\partial t}, \frac{\partial}{\partial t},  V_{i}, W_{j})=\overline{R}(V_{i}, W_{j}, \frac{\partial}{\partial t}, \frac{\partial}{\partial t})=\overline{R}(\frac{\partial}{\partial t}, V_{i}, W_{j}, \frac{\partial}{\partial t})=0,~if~i\neq j;\nonumber\\
&(6)\overline{R}(\frac{\partial}{\partial t}, V_{i}, W_{j}, U_{k})=\overline{R}(W_{j}, U_{k}, \frac{\partial}{\partial t}, V_{i})=0,~if~i=j=k;\nonumber\\
&(7)\overline{R}(\frac{\partial}{\partial t}, V_{i}, W_{j}, U_{k})=\overline{R}(\frac{\partial}{\partial t},  U_{k}, V_{i}, W_{j})=\overline{R}(W_{j}, U_{k}, \frac{\partial}{\partial t}, V_{i})=\overline{R}(V_{i}, W_{j}, \frac{\partial}{\partial t},  U_{k})=0,\nonumber\\
&~if~i=j\neq k;\nonumber\\
&(8)\overline{R}(U_{k}, V_{i}, W_{j}, Q_{s})=\phi^{2p_{j}}R_{F_{j}}(U_{j}, V_{j}, W_{j}, Q_{j})+(\phi^{2p_{j}}g^{*}_{I}(d\phi^{p_{j}}, d\phi^{p_{j}})+2\phi^{3p_{j}}\frac{\partial \phi^{p_{j}}}{\partial t}-\phi^{4p_{j}})\nonumber\\
&\times(g_{F_{j}}(U_{j}, W_{j})g_{F_{j}}(V_{j}, Q_{j})-g_{F_{j}}(V_{j}, W_{j})g_{F_{j}}(U_{j}, Q_{j})),~if~i=j=k=s;\nonumber\\
&(9)\overline{R}(U_{k}, V_{i}, W_{j}, Q_{s})=\overline{R}(W_{j}, Q_{s}, U_{k}, V_{i},)=0,~if~i=j=s\neq k;\nonumber\\
&(10)\overline{R}(U_{k}, V_{i}, W_{j}, Q_{s})=(\phi^{p_{i}}\phi^{p_{j}}g^{*}_{I}(d\phi^{p_{i}}, d\phi^{p_{j}})+\phi^{p_{i}}\phi^{2p_{j}}\frac{\partial \phi^{p_{i}}}{\partial t}+\phi^{p_{j}}\phi^{2p_{i}}\frac{\partial \phi^{p_{j}}}{\partial t}-\phi^{2p_{i}}\phi^{2p_{j}})\nonumber\\
&\times g_{F_{i}}(V_{i}, Q_{i})g_{F_{j}}(U_{j}, W_{j}),~if~j=k\neq i=s;\nonumber\\
&(11)\overline{R}(U_{k}, V_{i}, W_{j}, Q_{s})=0,~other~cases.\nonumber
\end{align}
\end{thm}

\begin{thm}
Let $M_{2}=I\times_{\phi^{p_{1}}}F_{1}\times_{\phi^{p_{2}}}F_{2}$ be a degenerate multiply warped product, let the vector fields $\partial/ \partial t\in\Gamma(TI)$ and $U_{j}, V_{j}, W_{j}, Q_{j}\in\Gamma(TF_{j}).$  Let $P\in\Gamma(TF_{l}),$ then
\begin{align}
&(1)\overline{R}(\frac{\partial}{\partial t}, \frac{\partial}{\partial t}, \frac{\partial}{\partial t}, \frac{\partial}{\partial t})=0;\nonumber\\
&(2)\overline{R}(\frac{\partial}{\partial t}, \frac{\partial}{\partial t}, \frac{\partial}{\partial t}, W_{j})=\overline{R}(\frac{\partial}{\partial t}, W_{j}, \frac{\partial}{\partial t}, \frac{\partial}{\partial t})=0;\nonumber\\
&(3)\overline{R}(\frac{\partial}{\partial t}, \frac{\partial}{\partial t},  V_{i}, W_{j})=\overline{R}(V_{i}, W_{j}, \frac{\partial}{\partial t}, \frac{\partial}{\partial t})=0;\nonumber\\
&(4)\overline{R}(\frac{\partial}{\partial t}, V_{i}, W_{j}, \frac{\partial}{\partial t})=-\phi^{p_{j}}\frac{\partial^{2} \phi^{p_{j}}}{\partial t^{2}}g_{F_{j}}(V_{j}, W_{j})+\phi^{2p_{j}}(\mathcal{K}_{F_{j}}(V_{j}, P, W_{j})+\phi^{2p_{j}}g_{F_{j}}(V_{j}, W_{j})\nonumber\\
&\times g_{F_{j}}(P, P)-\phi^{2p_{j}}g_{F_{j}}(V_{j}, P)g_{F_{j}}(W_{j}, P)),~if~i=j=l;\nonumber
\end{align}
\begin{align}
&(5)\overline{R}(\frac{\partial}{\partial t}, V_{i}, W_{j}, \frac{\partial}{\partial t})=-\phi^{p_{j}}\frac{\partial^{2} \phi^{p_{j}}}{\partial t^{2}}g_{F_{j}}(V_{j}, W_{j})+\phi^{2p_{j}}\phi^{2p_{l}}g_{F_{j}}(V_{j}, W_{j})g_{F_{j}}(P, P),~if~i=j\neq l;\nonumber\\
&(6)\overline{R}(\frac{\partial}{\partial t}, V_{i}, W_{j}, \frac{\partial}{\partial t})=0,~other~cases;\nonumber\\
&(7)\overline{R}(\frac{\partial}{\partial t}, V_{i}, W_{j}, U_{k})=-\overline{R}(W_{j}, U_{k}, \frac{\partial}{\partial t}, V_{i})=\phi^{3p_{j}}\frac{\partial \phi^{p_{j}}}{\partial t}(g_{F_{j}}(V_{j}, U_{j})g_{F_{j}}(W_{j}, P)-g_{F_{j}}(V_{j}, W_{j})\nonumber\\
&\times g_{F_{j}}(U_{j}, P)),~if~i=j=k=l;\nonumber\\
&(8)\overline{R}(\frac{\partial}{\partial t}, V_{i}, W_{j}, U_{k})=-\overline{R}(W_{j}, U_{k}, \frac{\partial}{\partial t}, V_{i})=-\phi^{p_{l}}\phi^{2p_{j}}\frac{\partial \phi^{p_{l}}}{\partial t}g_{F_{l}}(U_{l}, P)g_{F_{j}}(V_{j}, W_{j}),\nonumber\\
&~if~i=j\neq k=l;\nonumber\\
&(9)\overline{R}(\frac{\partial}{\partial t}, V_{i}, W_{j}, U_{k})=0,~other~cases;\nonumber\\
&(10)\overline{R}(U_{k}, V_{i}, W_{j}, Q_{s})=\phi^{2p_{j}}R_{F_{j}}(U_{j}, V_{j}, W_{j}, Q_{j})+(\phi^{2p_{j}}g^{*}_{I}(d\phi^{p_{j}}, d\phi^{p_{j}})+\phi^{6p_{j}}g_{F_{j}}(P, P))\nonumber\\
&\times(g_{F_{j}}(U_{j}, W_{j})g_{F_{j}}(V_{j}, Q_{j})-g_{F_{j}}(V_{j}, W_{j})g_{F_{j}}(U_{j}, Q_{j}))+\phi^{4p_{j}}g_{F_{j}}(V_{j}, Q_{j})\mathcal{K}_{F_{j}}(U_{j}, P, W_{j})\nonumber\\
&-\phi^{4p_{j}}g_{F_{j}}(U_{j}, Q_{j})\mathcal{K}_{F_{j}}(V_{j}, P, W_{j})-\phi^{4p_{j}}g_{F_{j}}(V_{j}, W_{j})\mathcal{K}_{F_{j}}(U_{j}, P, Q_{j})+\phi^{4p_{j}}g_{F_{j}}(U_{j}, W_{j})\nonumber\\
&\times \mathcal{K}_{F_{j}}(V_{j}, P, Q_{j})+\phi^{6p_{j}}g_{F_{j}}(U_{j}, P)(g_{F_{j}}(V_{j}, W_{j})g_{F_{j}}(Q_{j}, P)-g_{F_{j}}(V_{j}, Q_{j})g_{F_{j}}(W_{j}, P))\nonumber\\
&-\phi^{6p_{j}}g_{F_{j}}(V_{j}, P)(g_{F_{j}}(U_{j}, W_{j})g_{F_{j}}(Q_{j}, P)-g_{F_{j}}(U_{j}, Q_{j})g_{F_{j}}(W_{j}, P)),~if~i=j=k=s=l;\nonumber\\
&(11)\overline{R}(U_{k}, V_{i}, W_{j}, Q_{s})=\phi^{2p_{j}}R_{F_{j}}(U_{j}, V_{j}, W_{j}, Q_{j})+(\phi^{2p_{j}}g^{*}_{I}(d\phi^{p_{j}}, d\phi^{p_{j}})+\phi^{2p_{l}}\phi^{4p_{j}}g_{F_{l}}(P, P))\nonumber\\
&\times(g_{F_{j}}(U_{j}, W_{j})g_{F_{j}}(V_{j}, Q_{j})-g_{F_{j}}(V_{j}, W_{j})g_{F_{j}}(U_{j}, Q_{j})),~if~i=j=k=s\neq l;\nonumber\\
&(12)\overline{R}(U_{k}, V_{i}, W_{j}, Q_{s})=\phi^{p_{i}}\phi^{p_{j}}g^{*}_{I}(d\phi^{p_{i}}, d\phi^{p_{j}})g_{F_{i}}(V_{i}, Q_{i})g_{F_{j}}(U_{j}, W_{j})+\phi^{2p_{i}}\phi^{2p_{j}}g_{F_{j}}(U_{j}, W_{j})\nonumber\\
&\times(\mathcal{K}_{F_{i}}(V_{i}, P, Q_{i})+\phi^{2p_{i}}g_{F_{i}}(P, P)g_{F_{i}}(V_{i}, Q_{i})-\phi^{2p_{i}}g_{F_{i}}(V_{i}, P)g_{F_{i}}(Q_{i}, P)),\nonumber\\
&~if~j=k\neq i=s=l;\nonumber\\
&(13)\overline{R}(U_{k}, V_{i}, W_{j}, Q_{s})=\phi^{p_{i}}\phi^{p_{j}}g^{*}_{I}(d\phi^{p_{i}}, d\phi^{p_{j}})g_{F_{i}}(V_{i}, Q_{i})g_{F_{j}}(U_{j}, W_{j})+\phi^{2p_{i}}\phi^{2p_{j}}g_{F_{i}}(V_{i}, Q_{i})\nonumber\\
&\times(\mathcal{K}_{F_{j}}(U_{j}, P, W_{j})+\phi^{2p_{j}}g_{F_{j}}(P, P)g_{F_{j}}(U_{j}, W_{j})-\phi^{2p_{j}}g_{F_{j}}(U_{j}, P)g_{F_{j}}(W_{j}, P)),\nonumber\\
&~if~l=j=k\neq i=s;\nonumber\\
&(14)\overline{R}(U_{k}, V_{i}, W_{j}, Q_{s})=0,~other~cases.\nonumber
\end{align}
\end{thm}

\begin{prop}
Let $M_{2}=I\times_{\phi^{p_{1}}}F_{1}\times_{\phi^{p_{2}}}F_{2}$ be a degenerate multiply warped product, let the vector fields $\partial/ \partial t\in\Gamma(TI)$ and $U_{j}, V_{j}, W_{j}\in\Gamma(TF_{j}),$ for $j\in \{1, 2\}.$ Let $\widehat{\mathcal{K}}$ be the semi-symmetric non-metric Koszul form on $I\times_{\phi^{p_{1}}}F_{1}\times_{\phi^{p_{2}}}F_{2}$ and $\widehat{\mathcal{K}}_{I}$, $\widehat{\mathcal{K}}_{F_{j}}$ be the lifts of the semi-symmetric non-metric Koszul form on $I,$ $F_{j},$ respectively. Let $P=\partial/ \partial t,$ then
\begin{align}
&(1)\widehat{\mathcal{K}}(\frac{\partial}{\partial t}, \frac{\partial}{\partial t}, \frac{\partial}{\partial t})=\widehat{\mathcal{K}}_{I}(\frac{\partial}{\partial t}, \frac{\partial}{\partial t}, \frac{\partial}{\partial t})=1;\nonumber\\
&(2)\widehat{\mathcal{K}}(\frac{\partial}{\partial t}, \frac{\partial}{\partial t}, W_{j})=\widehat{\mathcal{K}}(\frac{\partial}{\partial t}, W_{j}, \frac{\partial}{\partial t})=\widehat{\mathcal{K}}(W_{j}, \frac{\partial}{\partial t}, \frac{\partial}{\partial t})=0;\nonumber
\end{align}
\begin{align}
&(3)\widehat{\mathcal{K}}(\frac{\partial}{\partial t}, V_{i}, W_{j})=-\widehat{\mathcal{K}}(V_{i}, W_{j}, \frac{\partial}{\partial t})=\phi^{p_{j}}\frac{\partial \phi^{p_{j}}}{\partial t}g_{F_{j}}(V_{j}, W_{j}),~if~i=j;\nonumber\\
&(4)\widehat{\mathcal{K}}(V_{i}, \frac{\partial}{\partial t}, W_{j})=\phi^{p_{j}}\frac{\partial \phi^{p_{j}}}{\partial t}g_{F_{j}}(V_{j}, W_{j})-\phi^{2p_{j}}g_{F_{j}}(V_{j}, W_{j}),~if~i=j;\nonumber\\
&(5)\widehat{\mathcal{K}}(\frac{\partial}{\partial t}, V_{i}, W_{j})=\widehat{\mathcal{K}}(V_{i}, \frac{\partial}{\partial t}, W_{j})=\widehat{\mathcal{K}}(V_{i}, W_{j}, \frac{\partial}{\partial t})=0,~if~i\neq j;\nonumber\\
&(6)\widehat{\mathcal{K}}(U_{i}, V_{j}, W_{k})=\phi^{2p_{j}}\mathcal{K}_{F_{j}}(U_{j}, V_{j}, W_{j}),~if~i=j=k;\nonumber\\
&(7)\widehat{\mathcal{K}}(U_{i}, V_{j}, W_{k})=0,~if~i=j\neq k.\nonumber
\end{align}
\end{prop}

\begin{prop}
Let $M_{2}=I\times_{\phi^{p_{1}}}F_{1}\times_{\phi^{p_{2}}}F_{2}$ be a degenerate multiply warped product, let the vector fields $\partial/ \partial t\in\Gamma(TI)$ and $U_{j}, V_{j}, W_{j}\in\Gamma(TF_{j}),$ for $j\in \{1, 2\}.$ Let $\widehat{\mathcal{K}}$ be the semi-symmetric non-metric Koszul form on $I\times_{\phi^{p_{1}}}F_{1}\times_{\phi^{p_{2}}}F_{2}$ and $\widehat{\mathcal{K}}_{I}$, $\widehat{\mathcal{K}}_{F_{j}}$ be the lifts of the semi-symmetric non-metric Koszul form on $I,$ $F_{j},$ respectively. Let $P\in\Gamma(TF_{l}),$ then
\begin{align}
&(1)\widehat{\mathcal{K}}(\frac{\partial}{\partial t}, \frac{\partial}{\partial t}, \frac{\partial}{\partial t})=\mathcal{K}_{I}(\frac{\partial}{\partial t}, \frac{\partial}{\partial t}, \frac{\partial}{\partial t})=0;\nonumber\\
&(2)\widehat{\mathcal{K}}(\frac{\partial}{\partial t}, \frac{\partial}{\partial t}, W_{j})=\widehat{\mathcal{K}}(W_{j}, \frac{\partial}{\partial t}, \frac{\partial}{\partial t})=0,~if~j=l;\nonumber\\
&(3)\widehat{\mathcal{K}}(\frac{\partial}{\partial t}, W_{j}, \frac{\partial}{\partial t})=-\phi^{2p_{j}}g_{F_{j}}(W_{j}, P),~if~j=l;\nonumber\\
&(4)\widehat{\mathcal{K}}(\frac{\partial}{\partial t}, \frac{\partial}{\partial t}, W_{j})=\widehat{\mathcal{K}}(\frac{\partial}{\partial t}, W_{j}, \frac{\partial}{\partial t})=\widehat{\mathcal{K}}(W_{j}, \frac{\partial}{\partial t}, \frac{\partial}{\partial t})=0,~if~j\neq l;\nonumber\\
&(5)\widehat{\mathcal{K}}(\frac{\partial}{\partial t}, V_{i}, W_{j})=\widehat{\mathcal{K}}(V_{i}, \frac{\partial}{\partial t}, W_{j})=-\widehat{\mathcal{K}}(V_{i}, W_{j}, \frac{\partial}{\partial t})=\phi^{p_{j}}\frac{\partial \phi^{p_{j}}}{\partial t}g_{F_{j}}(V_{j}, W_{j}),~if~i=j;\nonumber\\
&(6)\widehat{\mathcal{K}}(\frac{\partial}{\partial t}, V_{i}, W_{j})=\widehat{\mathcal{K}}(V_{i}, \frac{\partial}{\partial t}, W_{j})=\widehat{\mathcal{K}}(V_{i}, W_{j}, \frac{\partial}{\partial t})=0,~if~i\neq j;\nonumber\\
&(7)\widehat{\mathcal{K}}(U_{i}, V_{j}, W_{k})=\phi^{2p_{j}}\mathcal{K}_{F_{j}}(U_{j}, V_{j}, W_{j})+\phi^{4p_{j}}g_{F_{j}}(U_{j}, W_{j})g_{F_{j}}(V_{j}, P),~if~i=j=k=l;\nonumber\\
&(8)\widehat{\mathcal{K}}(U_{i}, V_{j}, W_{k})=\phi^{2p_{j}}\mathcal{K}_{F_{j}}(U_{j}, V_{j}, W_{j}),~if~i=j= k\neq l;\nonumber\\
&(9)\widehat{\mathcal{K}}(U_{i}, W_{k}, V_{j})=\phi^{2p_{j}}\phi^{2p_{k}}g_{F_{j}}(U_{j}, V_{j})g_{F_{k}}(W_{k}, P),~if~i=j\neq k=l;\nonumber\\
&(10)\widehat{\mathcal{K}}(U_{i}, V_{j}, W_{k})=0,~other~cases.\nonumber
\end{align}
\end{prop}

\begin{thm}
Let $M_{2}=I\times_{\phi^{p_{1}}}F_{1}\times_{\phi^{p_{2}}}F_{2}$ be a degenerate multiply warped product, let the vector fields $\partial/ \partial t\in\Gamma(TI)$ and $U_{j}, V_{j}, W_{j}, Q_{j}\in\Gamma(TF_{j}).$ Let $P=\partial/ \partial t,$ then
\begin{align}
&(1)\widehat{R}(\frac{\partial}{\partial t}, \frac{\partial}{\partial t}, \frac{\partial}{\partial t}, \frac{\partial}{\partial t})=\widehat{R}_{I}(\frac{\partial}{\partial t}, \frac{\partial}{\partial t}, \frac{\partial}{\partial t}, \frac{\partial}{\partial t})=0;\nonumber\\
&(2)\widehat{R}(\frac{\partial}{\partial t}, \frac{\partial}{\partial t}, \frac{\partial}{\partial t}, W_{j})=\widehat{R}(\frac{\partial}{\partial t}, \frac{\partial}{\partial t}, W_{j}, \frac{\partial}{\partial t})=\widehat{R}(\frac{\partial}{\partial t}, W_{j}, \frac{\partial}{\partial t}, \frac{\partial}{\partial t})=0;\nonumber\\
&(3)\widehat{R}(\frac{\partial}{\partial t}, \frac{\partial}{\partial t},  V_{i}, W_{j})=\widehat{R}(V_{i}, W_{j}, \frac{\partial}{\partial t}, \frac{\partial}{\partial t})=0,~if~i=j;\nonumber
\end{align}
\begin{align}
&(4)\widehat{R}(V_{i}, \frac{\partial}{\partial t}, W_{j}, \frac{\partial}{\partial t})=-\widehat{R}(\frac{\partial}{\partial t}, V_{i}, W_{j}, \frac{\partial}{\partial t})=\phi^{p_{j}}\frac{\partial^{2} \phi^{p_{j}}}{\partial t^{2}}g_{F_{j}}(V_{j}, W_{j})+2\phi^{p_{j}}\frac{\partial \phi^{p_{j}}}{\partial t}g_{F_{j}}(V_{j}, W_{j}),\nonumber\\
&~if~i=j;\nonumber\\
&(5)\widehat{R}(V_{i}, \frac{\partial}{\partial t}, \frac{\partial}{\partial t}, W_{j})=-\widehat{R}(\frac{\partial}{\partial t}, V_{i}, \frac{\partial}{\partial t}, W_{j})=-\phi^{p_{j}}\frac{\partial^{2} \phi^{p_{j}}}{\partial t^{2}}g_{F_{j}}(V_{j}, W_{j}),~if~i=j;\nonumber\\
&(6)\widehat{R}(\frac{\partial}{\partial t}, \frac{\partial}{\partial t},  V_{i}, W_{j})=\widehat{R}(V_{i}, W_{j}, \frac{\partial}{\partial t}, \frac{\partial}{\partial t})=\widehat{R}(\frac{\partial}{\partial t}, V_{i}, W_{j}, \frac{\partial}{\partial t})=\widehat{R}(\frac{\partial}{\partial t}, V_{i}, \frac{\partial}{\partial t}, W_{j})=0,\nonumber\\
&~if~i\neq j;\nonumber\\
&(7)\widehat{R}(\frac{\partial}{\partial t}, V_{i}, W_{j}, U_{k})=\widehat{R}(W_{j}, U_{k}, \frac{\partial}{\partial t}, V_{i})=0,~if~i=j=k;\nonumber\\
&(8)\widehat{R}(W_{j}, U_{k}, V_{i}, \frac{\partial}{\partial t})=-\phi^{2p_{j}}(\mathcal{K}_{F_{j}}(W_{j}, V_{j}, U_{j})-\mathcal{K}_{F_{j}}(U_{j}, V_{j}, W_{j})),~if~i=j=k;\nonumber\\
&(9)\widehat{R}(\frac{\partial}{\partial t}, V_{i}, W_{j}, U_{k})=\widehat{R}(W_{j}, U_{k}, \frac{\partial}{\partial t}, V_{i})=\widehat{R}(W_{j}, U_{k}, V_{i}, \frac{\partial}{\partial t})=0,~other~cases;\nonumber\\
&(10)\widehat{R}(U_{k}, V_{i}, W_{j}, Q_{s})=\phi^{2p_{j}}R_{F_{j}}(U_{j}, V_{j}, W_{j}, Q_{j})+\phi^{2p_{j}}g^{*}_{I}(d\phi^{p_{j}}, d\phi^{p_{j}})(g_{F_{j}}(U_{j}, W_{j})g_{F_{j}}(V_{j}, Q_{j})\nonumber\\
&-g_{F_{j}}(V_{j}, W_{j})g_{F_{j}}(U_{j}, Q_{j})),~if~i=j=k=s;\nonumber\\
&(11)\widehat{R}(U_{k}, V_{i}, W_{j}, Q_{s})=\phi^{p_{i}}\phi^{p_{j}}g^{*}_{I}(d\phi^{p_{i}}, d\phi^{p_{j}})g_{F_{i}}(V_{i}, Q_{i})g_{F_{j}}(U_{j}, W_{j}),~if~k=j\neq i=s;\nonumber\\
&(12)\widehat{R}(U_{k}, V_{i}, W_{j}, Q_{s})=-\phi^{p_{j}}\phi^{p_{k}}g^{*}_{I}(d\phi^{p_{j}}, d\phi^{p_{k}})g_{F_{j}}(V_{j}, W_{j})g_{F_{k}}(U_{k}, Q_{k}),~if~k=s\neq j=i;\nonumber\\
&(13)\widehat{R}(U_{k}, V_{i}, W_{j}, Q_{s})=0,~other~cases.\nonumber
\end{align}
\end{thm}

\begin{thm}
Let $M_{2}=I\times_{\phi^{p_{1}}}F_{1}\times_{\phi^{p_{2}}}F_{2}$ be a degenerate multiply warped product, let the vector fields $\partial/ \partial t\in\Gamma(TI)$ and $U_{j}, V_{j}, W_{j}, Q_{j}\in\Gamma(TF_{j}).$ Let $P\in\Gamma(TF_{l}),$ then
\begin{align}
&(1)\widehat{R}(\frac{\partial}{\partial t}, \frac{\partial}{\partial t}, \frac{\partial}{\partial t}, \frac{\partial}{\partial t})=R_{I}(\frac{\partial}{\partial t}, \frac{\partial}{\partial t}, \frac{\partial}{\partial t}, \frac{\partial}{\partial t})=0;\nonumber\\
&(2)\widehat{R}(\frac{\partial}{\partial t}, \frac{\partial}{\partial t}, \frac{\partial}{\partial t}, W_{j})=\widehat{R}(\frac{\partial}{\partial t}, \frac{\partial}{\partial t}, W_{j}, \frac{\partial}{\partial t})=\widehat{R}(\frac{\partial}{\partial t}, W_{j}, \frac{\partial}{\partial t}, \frac{\partial}{\partial t})=\widehat{R}(W_{j}, \frac{\partial}{\partial t}, \frac{\partial}{\partial t}, \frac{\partial}{\partial t})=0;\nonumber\\
&(3)\widehat{R}(\frac{\partial}{\partial t}, \frac{\partial}{\partial t},  V_{i}, W_{j})=\widehat{R}(V_{i}, W_{j}, \frac{\partial}{\partial t}, \frac{\partial}{\partial t})=0;\nonumber\\
&(4)\widehat{R}(V_{i}, \frac{\partial}{\partial t}, W_{j}, \frac{\partial}{\partial t})=-\widehat{R}(\frac{\partial}{\partial t}, V_{i}, W_{j}, \frac{\partial}{\partial t})=\phi^{p_{j}}\frac{\partial^{2} \phi^{p_{j}}}{\partial t^{2}}g_{F_{j}}(V_{j}, W_{j})-\phi^{2p_{j}}V_{j}(g_{F_{j}}(W_{j}, P)),\nonumber\\
&~if~i=j=l;\nonumber\\
&(5)\widehat{R}(V_{i}, \frac{\partial}{\partial t}, W_{j}, \frac{\partial}{\partial t})=-\widehat{R}(\frac{\partial}{\partial t}, V_{i}, W_{j}, \frac{\partial}{\partial t})=\phi^{p_{j}}\frac{\partial^{2} \phi^{p_{j}}}{\partial t^{2}}g_{F_{j}}(V_{j}, W_{j}),~if~i=j\neq l;\nonumber\\
&(6)\widehat{R}(V_{i}, \frac{\partial}{\partial t}, \frac{\partial}{\partial t}, W_{j})=-\widehat{R}(\frac{\partial}{\partial t}, V_{i}, \frac{\partial}{\partial t}, W_{j})=-\phi^{p_{j}}\frac{\partial^{2} \phi^{p_{j}}}{\partial t^{2}}g_{F_{j}}(V_{j}, W_{j}),~if~i=j=l\nonumber\\
&~or~i=j\neq l;\nonumber
\end{align}
\begin{align}
&(7)\widehat{R}(V_{i}, \frac{\partial}{\partial t}, W_{j}, \frac{\partial}{\partial t})=\widehat{R}(V_{i}, \frac{\partial}{\partial t}, \frac{\partial}{\partial t}, W_{j})=0,~other~cases;\nonumber\\
&(8)\widehat{R}(\frac{\partial}{\partial t}, V_{i}, W_{j}, U_{k})=2\phi^{3p_{j}}\frac{\partial \phi^{p_{j}}}{\partial t}(g_{F_{j}}(V_{j}, U_{j})g_{F_{j}}(W_{j}, P)+g_{F_{j}}(V_{j}, W_{j})g_{F_{j}}(U_{j}, P)),\nonumber\\
&~if~i=j=k=l;\nonumber\\
&(9)\widehat{R}(\frac{\partial}{\partial t}, V_{i}, W_{j}, U_{k})=\widehat{R}(\frac{\partial}{\partial t}, V_{i}, U_{k}, W_{j})=2\phi^{p_{j}}\phi^{2p_{k}}\frac{\partial \phi^{p_{j}}}{\partial t}g_{F_{j}}(V_{j}, W_{j})g_{F_{k}}(U_{k}, P),\nonumber\\
&~if~i=j\neq k=l;\nonumber\\
&(10)\widehat{R}(\frac{\partial}{\partial t}, V_{i}, W_{j}, U_{k})=\widehat{R}(W_{j}, U_{k}, \frac{\partial}{\partial t}, V_{i})=\widehat{R}(W_{j}, U_{k}, V_{i}, \frac{\partial}{\partial t})=0,~other~cases;\nonumber\\
&(11)\widehat{R}(U_{k}, V_{i}, W_{j}, Q_{s})=\phi^{2p_{j}}R_{F_{j}}(U_{j}, V_{j}, W_{j}, Q_{j})+\phi^{2p_{j}}g^{*}_{I}(d\phi^{p_{j}}, d\phi^{p_{j}})(g_{F_{j}}(U_{j}, W_{j})g_{F_{j}}(V_{j}, Q_{j})\nonumber\\
&-g_{F_{j}}(V_{j}, W_{j})g_{F_{j}}(U_{j}, Q_{j}))+\phi^{4p_{j}}g_{F_{j}}(Q_{j}, P)(\mathcal{K}_{F_{j}}(U_{j}, W_{j}, V_{j})-\mathcal{K}_{F_{j}}(V_{j}, W_{j}, U_{j}))\nonumber\\
&+\phi^{4p_{j}}U_{j}(g_{F_{j}}(W_{j}, P))g_{F_{j}}(V_{j}, Q_{j})-\phi^{4p_{j}}V_{j}(g_{F_{j}}(W_{j}, P))g_{F_{j}}(U_{j}, Q_{j}),~if~i=j=k=s=l;\nonumber\\
&(12)\widehat{R}(U_{k}, V_{i}, W_{j}, Q_{s})=\phi^{2p_{j}}R_{F_{j}}(U_{j}, V_{j}, W_{j}, Q_{j})+\phi^{2p_{j}}g^{*}_{I}(db_{j}, db_{j})(g_{F_{j}}(U_{j}, W_{j})g_{F_{j}}(V_{j}, Q_{j})\nonumber\\
&-g_{F_{j}}(V_{j}, W_{j})g_{F_{j}}(U_{j}, Q_{j}))~if~i=j=k=s\neq l;\nonumber\\
&(13)\widehat{R}(U_{k}, V_{i}, W_{j}, Q_{s})=\phi^{p_{i}}\phi^{p_{j}}g^{*}_{I}(d\phi^{p_{i}}, d\phi^{p_{j}})g_{F_{i}}(V_{i}, Q_{i})g_{F_{j}}(U_{j}, W_{j}),~if~j=k\neq i=s=l;\nonumber\\
&(14)\widehat{R}(U_{k}, V_{i}, W_{j}, Q_{s})=\phi^{p_{i}}\phi^{p_{j}}g^{*}_{I}(d\phi^{p_{i}}, d\phi^{p_{j}})g_{F_{i}}(V_{i}, Q_{i})g_{F_{j}}(U_{j}, W_{j})+\phi^{2p_{i}}\phi^{2p_{j}}U_{j}(g_{F_{j}}(W_{j}, P))\nonumber\\
&\times g_{F_{i}}(V_{i}, Q_{i}),~if~l=j=k\neq i=s;\nonumber\\
&(15)\widehat{R}(U_{k}, V_{i}, W_{j}, Q_{s})=-\phi^{p_{j}}\phi^{p_{k}}g^{*}_{I}(d\phi^{p_{j}}, d\phi^{p_{k}})g_{F_{j}}(V_{j}, W_{j})g_{F_{k}}(U_{k}, Q_{k})-\phi^{2p_{j}}\phi^{2p_{k}}V_{j}(g_{F_{j}}(W_{j}, P))\nonumber\\
&\times g_{F_{k}}(U_{k}, Q_{k}),~if~k=s\neq i=j=l;\nonumber\\
&(16)\widehat{R}(U_{k}, V_{i}, W_{j}, Q_{s})=-\phi^{p_{j}}\phi^{p_{k}}g^{*}_{I}(d\phi^{p_{j}}, d\phi^{p_{k}})g_{F_{j}}(V_{j}, W_{j})g_{F_{k}}(U_{k}, Q_{k}),~if~l=k=s\neq i=j;\nonumber\\
&(17)\widehat{R}(U_{k}, V_{i}, W_{j}, Q_{s})=\phi^{2p_{j}}\phi^{2p_{k}}g_{F_{k}}(U_{k}, P)(\mathcal{K}_{F_{j}}(V_{j}, Q_{j}, W_{j})-\mathcal{K}_{F_{j}}(W_{j}, Q_{j}, V_{j})),\nonumber\\
&~if~l=k\neq i=j=s;\nonumber\\
&(18)\widehat{R}(U_{k}, V_{i}, W_{j}, Q_{s})=0,~other~cases.\nonumber
\end{align}
\end{thm}

\begin{prop}
Let $M_{2}=I\times_{\phi^{p_{1}}}F_{1}\times_{\phi^{p_{2}}}F_{2}$ be a degenerate multiply warped product, let the vector fields $\partial/ \partial t\in\Gamma(TI)$ and $U_{j}, V_{j}, W_{j}\in\Gamma(TF_{j}),$ for $j\in \{1, 2\}.$ Let $\widetilde{\mathcal{K}}$ be the almost product Koszul form on $I\times_{\phi^{p_{1}}}F_{1}\times_{\phi^{p_{2}}}F_{2}$ and $\widetilde{\mathcal{K}}_{I}$, $\widetilde{\mathcal{K}}_{F_{j}}$ be the lifts of the almost product Koszul form on $I,$ $F_{j},$ respectively. Then
\begin{align}
&(1)\widetilde{\mathcal{K}}(\frac{\partial}{\partial t}, \frac{\partial}{\partial t}, \frac{\partial}{\partial t})=\widetilde{\mathcal{K}}_{I}(\frac{\partial}{\partial t}, \frac{\partial}{\partial t}, \frac{\partial}{\partial t})=0;\nonumber\\
&(2)\widetilde{\mathcal{K}}(\frac{\partial}{\partial t}, \frac{\partial}{\partial t}, W_{j})=\widetilde{\mathcal{K}}(\frac{\partial}{\partial t}, W_{j}, \frac{\partial}{\partial t})=\widetilde{\mathcal{K}}(W_{j}, \frac{\partial}{\partial t}, \frac{\partial}{\partial t})=0;\nonumber\\
&(3)\widetilde{\mathcal{K}}(\frac{\partial}{\partial t}, V_{i}, W_{j})=\phi^{p_{j}}\frac{\partial \phi^{p_{j}}}{\partial t}g_{F_{j}}(V_{j}, W_{j}),~if~i=j;\nonumber
\end{align}
\begin{align}
&(4)\widetilde{\mathcal{K}}(V_{i}, \frac{\partial}{\partial t}, W_{j})=-\widetilde{\mathcal{K}}(V_{i}, W_{j}, \frac{\partial}{\partial t})=\frac{1}{2}(\phi^{p_{j}}\frac{\partial \phi^{p_{j}}}{\partial t}g_{F_{j}}(V_{j}, W_{j})+\phi^{p_{j}}J_{I}(\frac{\partial \phi^{p_{j}}}{\partial t})g_{F_{j}}(V_{j}, J_{F_{j}}W_{j})),\nonumber\\
&~if~i=j;\nonumber\\
&(5)\widetilde{\mathcal{K}}(\frac{\partial}{\partial t}, V_{i}, W_{j})=\widetilde{\mathcal{K}}(V_{i}, \frac{\partial}{\partial t}, W_{j})=\widetilde{\mathcal{K}}(V_{i}, W_{j}, \frac{\partial}{\partial t})=0,~if~i\neq j;\nonumber\\
&(6)\widetilde{\mathcal{K}}(U_{i}, V_{j}, W_{k})=\phi^{2p_{j}}\widetilde{\mathcal{K}}_{F_{j}}(U_{j}, V_{j}, W_{j}),~if~i=j=k;\nonumber\\
&(7)\widetilde{\mathcal{K}}(U_{i}, V_{j}, W_{k})=\widetilde{\mathcal{K}}(U_{i}, W_{k}, V_{j})=\widetilde{\mathcal{K}}(W_{k}, U_{i}, V_{j})=0,~if~i=j\neq k.\nonumber
\end{align}
\end{prop}

\begin{thm}
Let $M_{2}=I\times_{\phi^{p_{1}}}F_{1}\times_{\phi^{p_{2}}}F_{2}$ be a degenerate multiply warped product, let the vector fields $\partial/ \partial t\in\Gamma(TI)$ and $U_{j}, V_{j}, W_{j}, Q_{j}\in\Gamma(TF_{j}),$ then
\begin{align}
&(1)\widetilde{R}(\frac{\partial}{\partial t}, \frac{\partial}{\partial t}, \frac{\partial}{\partial t}, \frac{\partial}{\partial t})=\widetilde{R}_{I}(\frac{\partial}{\partial t}, \frac{\partial}{\partial t}, \frac{\partial}{\partial t}, \frac{\partial}{\partial t})=0;\nonumber\\
&(2)\widetilde{R}(\frac{\partial}{\partial t}, \frac{\partial}{\partial t}, \frac{\partial}{\partial t}, W_{j})=\widetilde{R}(\frac{\partial}{\partial t}, W_{j}, \frac{\partial}{\partial t}, \frac{\partial}{\partial t})=0;\nonumber\\
&(3)\widetilde{R}(\frac{\partial}{\partial t}, \frac{\partial}{\partial t},  V_{i}, W_{j})=\widetilde{R}(V_{i}, W_{j}, \frac{\partial}{\partial t}, \frac{\partial}{\partial t})=0;\nonumber\\
&(4)\widetilde{R}(\frac{\partial}{\partial t}, V_{i}, \frac{\partial}{\partial t}, W_{j})=\frac{\phi^{p_{j}}}{2}\frac{\partial^{2} \phi^{p_{j}}}{\partial t^{2}}g_{F_{j}}(V_{j}, W_{j})+\frac{\phi^{p_{j}}}{2}\frac{\partial }{\partial t}(J_{I}(\frac{\partial \phi^{p_{j}}}{\partial t}))g_{F_{j}}(V_{j}, J_{F_{j}}W_{j}),~if~i=j;\nonumber\\
&(5)\widetilde{R}(\frac{\partial}{\partial t}, V_{i}, \frac{\partial}{\partial t}, W_{j})=0,~if~i\neq j;\nonumber\\
&(6)\widetilde{R}(\frac{\partial}{\partial t}, V_{i}, W_{j}, U_{k})=0;\nonumber\\
&(7)\widetilde{R}(U_{k}, V_{i}, \frac{\partial}{\partial t}, W_{j})=\frac{\phi^{p_{j}}}{4}\frac{\partial \phi^{p_{j}}}{\partial t}(-\mathcal{K}_{F_{j}}(V_{j}, W_{j}, U_{j})+\mathcal{K}_{F_{j}}(U_{j}, W_{j}, V_{j})+\mathcal{K}_{F_{j}}(V_{j}, J_{F_{j}}W_{j}, J_{F_{j}}U_{j})\nonumber\\
&-\mathcal{K}_{F_{j}}(U_{j}, J_{F_{j}}W_{j}, J_{F_{j}}V_{j}))-\frac{\phi^{p_{j}}}{4}J_{I}(\frac{\partial \phi^{p_{j}}}{\partial t})(-\mathcal{K}_{F_{j}}(V_{j}, W_{j}, J_{F_{j}}U_{j})+\mathcal{K}_{F_{j}}(U_{j}, W_{j}, J_{F_{j}}V_{j})\nonumber\\
&+\mathcal{K}_{F_{j}}(V_{j}, J_{F_{j}}W_{j}, U_{j})-\mathcal{K}_{F_{j}}(U_{j}, J_{F_{j}}W_{j}, V_{j})),~if~i=j=k;\nonumber\\
&(8)\widetilde{R}(U_{k}, V_{i}, \frac{\partial}{\partial t}, W_{j})=0,~other~cases;\nonumber\\
&(9)\widetilde{R}(U_{k}, V_{i}, W_{j}, Q_{s})=\phi^{2p_{j}}\widetilde{R}_{F_{j}}(U_{j}, V_{j}, W_{j}, Q_{j})+\frac{1}{4}\phi^{2p_{j}}[g^{*}_{I}(d\phi^{p_{j}}, d\phi^{p_{j}})(g_{F_{j}}(U_{j}, W_{j})g_{F_{j}}(V_{j}, Q_{j})\nonumber\\
&-g_{F_{j}}(V_{j}, W_{j})g_{F_{j}}(U_{j}, Q_{j}))+g^{*}_{I}(d\phi^{p_{j}}\circ J_{I}, d\phi^{p_{j}})(g_{F_{j}}(U_{j}, J_{F_{j}}W_{j})g_{F_{j}}(V_{j}, Q_{j})\nonumber\\
&-g_{F_{j}}(V_{j}, J_{F_{j}}W_{j})g_{F_{j}}(U_{j}, Q_{j})+g_{F_{j}}(U_{j}, W_{j})g_{F_{j}}(V_{j}, J_{F_{j}}Q_{j})-g_{F_{j}}(V_{j}, W_{j})g_{F_{j}}(U_{j}, J_{F_{j}}Q_{j}))\nonumber\\
&+g^{*}_{I}(d\phi^{p_{j}}\circ J_{I}, d\phi^{p_{j}}\circ J_{I})(g_{F_{j}}(U_{j}, J_{F_{j}}W_{j})g_{F_{j}}(V_{j}, J_{F_{j}}Q_{j})-g_{F_{j}}(V_{j}, J_{F_{j}}W_{j})g_{F_{j}}(U_{j}, J_{F_{j}}Q_{j}))],\nonumber\\
&~if~i=j=k=s;\nonumber
\end{align}
\begin{align}
&(10)\widetilde{R}(U_{k}, V_{i}, W_{j}, Q_{s})=\frac{1}{4}\phi^{p_{i}}\phi^{p_{j}}(g^{*}_{I}(d\phi^{p_{i}}, d\phi^{p_{j}})g_{F_{i}}(V_{i}, Q_{i})g_{F_{j}}(U_{j}, W_{j})+g^{*}_{I}(d\phi^{p_{i}}\circ J_{I}, d\phi^{p_{j}})\nonumber\\
&\times g_{F_{i}}(V_{i}, J_{F_{i}}Q_{i})g_{F_{j}}(U_{j}, W_{j})+g^{*}_{I}(d\phi^{p_{i}}, d\phi^{p_{j}}\circ J_{I})g_{F_{i}}(V_{i}, Q_{i})g_{F_{j}}(U_{j}, J_{F_{j}}W_{j})\nonumber\\
&+g^{*}_{I}(d\phi^{p_{i}}\circ J_{I}, d\phi^{p_{j}}\circ J_{I})g_{F_{i}}(V_{i}, J_{F_{j}}Q_{i})g_{F_{j}}(U_{j}, J_{F_{j}}W_{j})),~if~j=k\neq i=s;\nonumber\\
&(11)\widetilde{R}(U_{k}, V_{i}, W_{j}, Q_{s})=0,~if~i=k\neq j=s.\nonumber
\end{align}
\end{thm}

\begin{prop}
Let $M_{3}=I\times_{\phi^{p_{1}}}F_{1}\times_{\phi^{p_{2}}}F_{2}\times_{\phi^{p_{3}}}F_{3}$ be a degenerate multiply warped product, let the vector fields $\partial/ \partial t\in\Gamma(TI)$ and $U_{j}, V_{j}, W_{j}\in\Gamma(TF_{j}),$ for $j\in \{1, 2, 3\}.$ Let $\overline{\mathcal{K}}$ be the semi-symmetric metric Koszul form on $I\times_{\phi^{p_{1}}}F_{1}\times_{\phi^{p_{2}}}F_{2}\times_{\phi^{p_{3}}}F_{3}$ and $\overline{\mathcal{K}}_{I}$, $\overline{\mathcal{K}}_{F_{j}}$ be the lifts of the semi-symmetric metric Koszul form on $I,$ $F_{j},$ respectively. Let $P=\partial/ \partial t,$ then
\begin{align}
&(1)\overline{\mathcal{K}}(\frac{\partial}{\partial t}, \frac{\partial}{\partial t}, \frac{\partial}{\partial t})=\overline{\mathcal{K}}_{I}(\frac{\partial}{\partial t}, \frac{\partial}{\partial t}, \frac{\partial}{\partial t})=0;\nonumber\\
&(2)\overline{\mathcal{K}}(\frac{\partial}{\partial t}, \frac{\partial}{\partial t}, W_{j})=\overline{\mathcal{K}}(\frac{\partial}{\partial t}, W_{j}, \frac{\partial}{\partial t})=\overline{\mathcal{K}}(W_{j}, \frac{\partial}{\partial t}, \frac{\partial}{\partial t})=0;\nonumber\\
&(3)\overline{\mathcal{K}}(\frac{\partial}{\partial t}, V_{i}, W_{j})=\phi^{p_{j}}\frac{\partial \phi^{p_{j}}}{\partial t}g_{F_{j}}(V_{j}, W_{j}),~if~i=j;\nonumber\\
&(4)\overline{\mathcal{K}}(V_{i}, \frac{\partial}{\partial t}, W_{j})=-\overline{\mathcal{K}}(V_{i}, W_{j}, \frac{\partial}{\partial t})=\phi^{p_{j}}\frac{\partial \phi^{p_{j}}}{\partial t}g_{F_{j}}(V_{j}, W_{j})-\phi^{2p_{j}}g_{F_{j}}(V_{j}, W_{j}),~if~i=j;\nonumber\\
&(5)\overline{\mathcal{K}}(\frac{\partial}{\partial t}, V_{i}, W_{j})=\overline{\mathcal{K}}(V_{i}, \frac{\partial}{\partial t}, W_{j})=\overline{\mathcal{K}}(V_{i}, W_{j}, \frac{\partial}{\partial t})=0,~if~i\neq j;\nonumber\\
&(6)\overline{\mathcal{K}}(U_{i}, V_{j}, W_{k})=\phi^{2p_{j}}\mathcal{K}_{F_{j}}(U_{j}, V_{j}, W_{j}),~if~i=j=k;\nonumber\\
&(7)\overline{\mathcal{K}}(U_{i}, V_{j}, W_{k})=0,~other~cases.\nonumber
\end{align}
\end{prop}

\begin{prop}
Let $M_{3}=I\times_{\phi^{p_{1}}}F_{1}\times_{\phi^{p_{2}}}F_{2}\times_{\phi^{p_{3}}}F_{3}$ be a degenerate multiply warped product, let the vector fields $\partial/ \partial t\in\Gamma(TI)$ and $U_{j}, V_{j}, W_{j}\in\Gamma(TF_{j}),$ for $j\in \{1, 2, 3\}.$ Let $\overline{\mathcal{K}}$ be the semi-symmetric metric Koszul form on $I\times_{\phi^{p_{1}}}F_{1}\times_{\phi^{p_{2}}}F_{2}\times_{\phi^{p_{3}}}F_{3}$ and $\overline{\mathcal{K}}_{I}$, $\overline{\mathcal{K}}_{F_{j}}$ be the lifts of the semi-symmetric metric Koszul form on $I,$ $F_{j},$ respectively. Let $P\in\Gamma(TF_{l}),$ then
\begin{align}
&(1)\overline{\mathcal{K}}(\frac{\partial}{\partial t}, \frac{\partial}{\partial t}, \frac{\partial}{\partial t})=\mathcal{K}_{I}(\frac{\partial}{\partial t}, \frac{\partial}{\partial t}, \frac{\partial}{\partial t})=0;\nonumber\\
&(2)\overline{\mathcal{K}}(\frac{\partial}{\partial t}, \frac{\partial}{\partial t}, W_{j})=-\overline{\mathcal{K}}(\frac{\partial}{\partial t}, W_{j}, \frac{\partial}{\partial t})=\phi^{2p_{j}}g_{F_{j}}(W_{j}, P),~if~j=l;\nonumber\\
&(3)\overline{\mathcal{K}}(\frac{\partial}{\partial t}, \frac{\partial}{\partial t}, W_{j})=\overline{\mathcal{K}}(\frac{\partial}{\partial t}, W_{j}, \frac{\partial}{\partial t})=0,~if~j\neq l;\nonumber\\
&(4)\overline{\mathcal{K}}(W_{j}, \frac{\partial}{\partial t}, \frac{\partial}{\partial t})=0;\nonumber\\
&(5)\overline{\mathcal{K}}(\frac{\partial}{\partial t}, V_{i}, W_{j})=\overline{\mathcal{K}}(V_{i}, \frac{\partial}{\partial t}, W_{j})=-\overline{\mathcal{K}}(V_{i}, W_{j}, \frac{\partial}{\partial t})=\phi^{p_{j}}\frac{\partial \phi^{p_{j}}}{\partial t}g_{F_{j}}(V_{j}, W_{j}),~if~i=j;\nonumber
\end{align}
\begin{align}
&(6)\overline{\mathcal{K}}(\frac{\partial}{\partial t}, V_{i}, W_{j})=\overline{\mathcal{K}}(V_{i}, \frac{\partial}{\partial t}, W_{j})=\overline{\mathcal{K}}(V_{i}, W_{j}, \frac{\partial}{\partial t})=0,~if~i\neq j;\nonumber\\
&(7)\overline{\mathcal{K}}(U_{i}, V_{j}, W_{k})=\phi^{2p_{j}}\mathcal{K}_{F_{j}}(U_{j}, V_{j}, W_{j})+\phi^{4p_{j}}g_{F_{j}}(U_{j}, W_{j})g_{F_{j}}(V_{j}, P)-\phi^{4p_{j}}g_{F_{j}}(U_{j}, V_{j})g_{F_{j}}(W_{j}, P),\nonumber\\
&~if~i=j=k=l;\nonumber\\
&(8)\overline{\mathcal{K}}(U_{i}, V_{j}, W_{k})=-\overline{\mathcal{K}}(U_{i}, W_{k}, V_{j})=-\phi^{2p_{j}}\phi^{2p_{k}}g_{F_{j}}(U_{j}, V_{j})g_{F_{k}}(W_{k}, P),~if~i=j\neq k=l;\nonumber\\
&(9)\overline{\mathcal{K}}(U_{i}, V_{j}, W_{k})=0,~other~cases.\nonumber
\end{align}
\end{prop}

\begin{thm}
Let $M_{3}=I\times_{\phi^{p_{1}}}F_{1}\times_{\phi^{p_{2}}}F_{2}\times_{\phi^{p_{3}}}F_{3}$ be a degenerate multiply warped product, let the vector fields $\partial/ \partial t\in\Gamma(TI)$ and $U_{j}, V_{j}, W_{j}, Q_{j}\in\Gamma(TF_{j}).$  Let $P=\partial/ \partial t,$ then
\begin{align}
&(1)\overline{R}(\frac{\partial}{\partial t}, \frac{\partial}{\partial t}, \frac{\partial}{\partial t}, \frac{\partial}{\partial t})=\overline{R}_{I}(\frac{\partial}{\partial t}, \frac{\partial}{\partial t}, \frac{\partial}{\partial t}, \frac{\partial}{\partial t})=0;\nonumber\\
&(2)\overline{R}(\frac{\partial}{\partial t}, \frac{\partial}{\partial t}, \frac{\partial}{\partial t}, W_{j})=\overline{R}(\frac{\partial}{\partial t}, W_{j}, \frac{\partial}{\partial t}, \frac{\partial}{\partial t})=0;\nonumber\\
&(3)\overline{R}(\frac{\partial}{\partial t}, \frac{\partial}{\partial t},  V_{i}, W_{j})=\overline{R}(V_{i}, W_{j}, \frac{\partial}{\partial t}, \frac{\partial}{\partial t})=0,~if~i=j;\nonumber\\
&(4)\overline{R}(\frac{\partial}{\partial t}, V_{i}, W_{j}, \frac{\partial}{\partial t})=-\phi^{p_{j}}\frac{\partial^{2} \phi^{p_{j}}}{\partial t^{2}}g_{F_{j}}(V_{j}, W_{j})+\phi^{p_{j}}\frac{\partial \phi^{p_{j}}}{\partial t}g_{F_{j}}(V_{j}, W_{j}),~if~i=j;\nonumber\\
&(5)\overline{R}(\frac{\partial}{\partial t}, \frac{\partial}{\partial t},  V_{i}, W_{j})=\overline{R}(V_{i}, W_{j}, \frac{\partial}{\partial t}, \frac{\partial}{\partial t})=\overline{R}(\frac{\partial}{\partial t}, V_{i}, W_{j}, \frac{\partial}{\partial t})=0,~if~i\neq j;\nonumber\\
&(6)\overline{R}(\frac{\partial}{\partial t}, V_{i}, W_{j}, U_{k})=\overline{R}(W_{j}, U_{k}, \frac{\partial}{\partial t}, V_{i})=0,~if~i=j=k;\nonumber\\
&(7)\overline{R}(\frac{\partial}{\partial t}, V_{i}, W_{j}, U_{k})=\overline{R}(\frac{\partial}{\partial t},  U_{k}, V_{i}, W_{j})=\overline{R}(W_{j}, U_{k}, \frac{\partial}{\partial t}, V_{i})=\overline{R}(V_{i}, W_{j}, \frac{\partial}{\partial t},  U_{k})=0,\nonumber\\
&~if~i=j\neq k~or~i, j, k~are~different;\nonumber\\
&(8)\overline{R}(U_{k}, V_{i}, W_{j}, Q_{s})=\phi^{2p_{j}}R_{F_{j}}(U_{j}, V_{j}, W_{j}, Q_{j})+(\phi^{2p_{j}}g^{*}_{I}(d\phi^{p_{j}}, d\phi^{p_{j}})+2\phi^{3p_{j}}\frac{\partial \phi^{p_{j}}}{\partial t}-\phi^{4p_{j}})\nonumber\\
&\times(g_{F_{j}}(U_{j}, W_{j})g_{F_{j}}(V_{j}, Q_{j})-g_{F_{j}}(V_{j}, W_{j})g_{F_{j}}(U_{j}, Q_{j})),~if~i=j=k=s;\nonumber\\
&(9)\overline{R}(U_{k}, V_{i}, W_{j}, Q_{s})=\overline{R}(W_{j}, Q_{s}, U_{k}, V_{i},)=0,~if~i=j=s\neq k;\nonumber\\
&(10)\overline{R}(U_{k}, V_{i}, W_{j}, Q_{s})=(\phi^{p_{i}}\phi^{p_{j}}g^{*}_{I}(d\phi^{p_{i}}, d\phi^{p_{j}})+\phi^{p_{i}}\phi^{2p_{j}}\frac{\partial \phi^{p_{i}}}{\partial t}+\phi^{p_{j}}\phi^{2p_{i}}\frac{\partial \phi^{p_{j}}}{\partial t}-\phi^{2p_{i}}\phi^{2p_{j}})\nonumber\\
&\times g_{F_{i}}(V_{i}, Q_{i})g_{F_{j}}(U_{j}, W_{j}),~if~j=k\neq i=s;\nonumber\\
&(11)\overline{R}(U_{k}, V_{i}, W_{j}, Q_{s})=0,~other~cases.\nonumber
\end{align}
\end{thm}

\begin{thm}
Let $M_{3}=I\times_{\phi^{p_{1}}}F_{1}\times_{\phi^{p_{2}}}F_{2}\times_{\phi^{p_{3}}}F_{3}$ be a degenerate multiply warped product, let the vector fields $\partial/ \partial t\in\Gamma(TI)$ and $U_{j}, V_{j}, W_{j}, Q_{j}\in\Gamma(TF_{j}).$  Let $P\in\Gamma(TF_{l}),$ then
\begin{align}
&(1)\overline{R}(\frac{\partial}{\partial t}, \frac{\partial}{\partial t}, \frac{\partial}{\partial t}, \frac{\partial}{\partial t})=0;\nonumber\\
&(2)\overline{R}(\frac{\partial}{\partial t}, \frac{\partial}{\partial t}, \frac{\partial}{\partial t}, W_{j})=\overline{R}(\frac{\partial}{\partial t}, W_{j}, \frac{\partial}{\partial t}, \frac{\partial}{\partial t})=0;\nonumber\\
&(3)\overline{R}(\frac{\partial}{\partial t}, \frac{\partial}{\partial t},  V_{i}, W_{j})=\overline{R}(V_{i}, W_{j}, \frac{\partial}{\partial t}, \frac{\partial}{\partial t})=0;\nonumber\\
&(4)\overline{R}(\frac{\partial}{\partial t}, V_{i}, W_{j}, \frac{\partial}{\partial t})=-\phi^{p_{j}}\frac{\partial^{2} \phi^{p_{j}}}{\partial t^{2}}g_{F_{j}}(V_{j}, W_{j})+\phi^{2p_{j}}(\mathcal{K}_{F_{j}}(V_{j}, P, W_{j})+\phi^{2p_{j}}g_{F_{j}}(V_{j}, W_{j})\nonumber\\
&\times g_{F_{j}}(P, P)-\phi^{2p_{j}}g_{F_{j}}(V_{j}, P)g_{F_{j}}(W_{j}, P)),~if~i=j=l;\nonumber\\
&(5)\overline{R}(\frac{\partial}{\partial t}, V_{i}, W_{j}, \frac{\partial}{\partial t})=-\phi^{p_{j}}\frac{\partial^{2} \phi^{p_{j}}}{\partial t^{2}}g_{F_{j}}(V_{j}, W_{j})+\phi^{2p_{j}}\phi^{2p_{l}}g_{F_{j}}(V_{j}, W_{j})g_{F_{j}}(P, P),~if~i=j\neq l;\nonumber\\
&(6)\overline{R}(\frac{\partial}{\partial t}, V_{i}, W_{j}, \frac{\partial}{\partial t})=0,~other~cases;\nonumber\\
&(7)\overline{R}(\frac{\partial}{\partial t}, V_{i}, W_{j}, U_{k})=-\overline{R}(W_{j}, U_{k}, \frac{\partial}{\partial t}, V_{i})=\phi^{3p_{j}}\frac{\partial \phi^{p_{j}}}{\partial t}(g_{F_{j}}(V_{j}, U_{j})g_{F_{j}}(W_{j}, P)-g_{F_{j}}(V_{j}, W_{j})\nonumber\\
&\times g_{F_{j}}(U_{j}, P)),~if~i=j=k=l;\nonumber\\
&(8)\overline{R}(\frac{\partial}{\partial t}, V_{i}, W_{j}, U_{k})=-\overline{R}(W_{j}, U_{k}, \frac{\partial}{\partial t}, V_{i})=-\phi^{p_{l}}\phi^{2p_{j}}\frac{\partial \phi^{p_{l}}}{\partial t}g_{F_{l}}(U_{l}, P)g_{F_{j}}(V_{j}, W_{j}),\nonumber\\
&~if~i=j\neq k=l;\nonumber\\
&(9)\overline{R}(\frac{\partial}{\partial t}, V_{i}, W_{j}, U_{k})=0,~other~cases;\nonumber\\
&(10)\overline{R}(U_{k}, V_{i}, W_{j}, Q_{s})=\phi^{2p_{j}}R_{F_{j}}(U_{j}, V_{j}, W_{j}, Q_{j})+(\phi^{2p_{j}}g^{*}_{I}(d\phi^{p_{j}}, d\phi^{p_{j}})+\phi^{6p_{j}}g_{F_{j}}(P, P))\nonumber\\
&\times(g_{F_{j}}(U_{j}, W_{j})g_{F_{j}}(V_{j}, Q_{j})-g_{F_{j}}(V_{j}, W_{j})g_{F_{j}}(U_{j}, Q_{j}))+\phi^{4p_{j}}g_{F_{j}}(V_{j}, Q_{j})\mathcal{K}_{F_{j}}(U_{j}, P, W_{j})\nonumber\\
&-\phi^{4p_{j}}g_{F_{j}}(U_{j}, Q_{j})\mathcal{K}_{F_{j}}(V_{j}, P, W_{j})-\phi^{4p_{j}}g_{F_{j}}(V_{j}, W_{j})\mathcal{K}_{F_{j}}(U_{j}, P, Q_{j})+\phi^{4p_{j}}g_{F_{j}}(U_{j}, W_{j})\nonumber\\
&\times \mathcal{K}_{F_{j}}(V_{j}, P, Q_{j})+\phi^{6p_{j}}g_{F_{j}}(U_{j}, P)(g_{F_{j}}(V_{j}, W_{j})g_{F_{j}}(Q_{j}, P)-g_{F_{j}}(V_{j}, Q_{j})g_{F_{j}}(W_{j}, P))\nonumber\\
&-\phi^{6p_{j}}g_{F_{j}}(V_{j}, P)(g_{F_{j}}(U_{j}, W_{j})g_{F_{j}}(Q_{j}, P)-g_{F_{j}}(U_{j}, Q_{j})g_{F_{j}}(W_{j}, P)),~if~i=j=k=s=l;\nonumber\\
&(11)\overline{R}(U_{k}, V_{i}, W_{j}, Q_{s})=\phi^{2p_{j}}R_{F_{j}}(U_{j}, V_{j}, W_{j}, Q_{j})+(\phi^{2p_{j}}g^{*}_{I}(d\phi^{p_{j}}, d\phi^{p_{j}})+\phi^{2p_{l}}\phi^{4p_{j}}g_{F_{l}}(P, P))\nonumber\\
&\times(g_{F_{j}}(U_{j}, W_{j})g_{F_{j}}(V_{j}, Q_{j})-g_{F_{j}}(V_{j}, W_{j})g_{F_{j}}(U_{j}, Q_{j})),~if~i=j=k=s\neq l;\nonumber\\
&(12)\overline{R}(U_{k}, V_{i}, W_{j}, Q_{s})=\phi^{p_{i}}\phi^{p_{j}}g^{*}_{I}(d\phi^{p_{i}}, d\phi^{p_{j}})g_{F_{i}}(V_{i}, Q_{i})g_{F_{j}}(U_{j}, W_{j})+\phi^{2p_{i}}\phi^{2p_{j}}g_{F_{j}}(U_{j}, W_{j})\nonumber\\
&\times(\mathcal{K}_{F_{i}}(V_{i}, P, Q_{i})+\phi^{2p_{i}}g_{F_{i}}(P, P)g_{F_{i}}(V_{i}, Q_{i})-\phi^{2p_{i}}g_{F_{i}}(V_{i}, P)g_{F_{i}}(Q_{i}, P)),\nonumber\\
&~if~j=k\neq i=s=l;\nonumber\\
&(13)\overline{R}(U_{k}, V_{i}, W_{j}, Q_{s})=\phi^{p_{i}}\phi^{p_{j}}g^{*}_{I}(d\phi^{p_{i}}, d\phi^{p_{j}})g_{F_{i}}(V_{i}, Q_{i})g_{F_{j}}(U_{j}, W_{j})+\phi^{2p_{i}}\phi^{2p_{j}}g_{F_{i}}(V_{i}, Q_{i})\nonumber\\
&\times(\mathcal{K}_{F_{j}}(U_{j}, P, W_{j})+\phi^{2p_{j}}g_{F_{j}}(P, P)g_{F_{j}}(U_{j}, W_{j})-\phi^{2p_{j}}g_{F_{j}}(U_{j}, P)g_{F_{j}}(W_{j}, P)),\nonumber\\
&~if~l=j=k\neq i=s;\nonumber
\end{align}
\begin{align}
&(14)\overline{R}(U_{k}, V_{i}, W_{j}, Q_{s})=\phi^{p_{i}}\phi^{p_{j}}g^{*}_{I}(d\phi^{p_{i}}, d\phi^{p_{j}})g_{F_{i}}(V_{i}, Q_{i})g_{F_{j}}(U_{j}, W_{j})+\phi^{2p_{i}}\phi^{2p_{j}}\phi^{2p_{l}}g_{F_{i}}(V_{i}, Q_{i})\nonumber\\
&\times g_{F_{j}}(U_{j}, W_{j})g_{F_{l}}(P, P),~where~i=s, j=k, l~are~different;\nonumber\\
&(15)\overline{R}(U_{k}, V_{i}, W_{j}, Q_{s})=0,~other~cases.\nonumber
\end{align}
\end{thm}

\begin{prop}
Let $M_{3}=I\times_{\phi^{p_{1}}}F_{1}\times_{\phi^{p_{2}}}F_{2}\times_{\phi^{p_{3}}}F_{3}$ be a degenerate multiply warped product, let the vector fields $\partial/ \partial t\in\Gamma(TI)$ and $U_{j}, V_{j}, W_{j}\in\Gamma(TF_{j}),$ for $j\in \{1, 2, 3\}.$ Let $\widehat{\mathcal{K}}$ be the semi-symmetric non-metric Koszul form on $I\times_{\phi^{p_{1}}}F_{1}\times_{\phi^{p_{2}}}F_{2}\times_{\phi^{p_{3}}}F_{3}$ and $\widehat{\mathcal{K}}_{I}$, $\widehat{\mathcal{K}}_{F_{j}}$ be the lifts of the semi-symmetric non-metric Koszul form on $I,$ $F_{j},$ respectively. Let $P=\partial/ \partial t,$ then
\begin{align}
&(1)\widehat{\mathcal{K}}(\frac{\partial}{\partial t}, \frac{\partial}{\partial t}, \frac{\partial}{\partial t})=\widehat{\mathcal{K}}_{I}(\frac{\partial}{\partial t}, \frac{\partial}{\partial t}, \frac{\partial}{\partial t})=1;\nonumber\\
&(2)\widehat{\mathcal{K}}(\frac{\partial}{\partial t}, \frac{\partial}{\partial t}, W_{j})=\widehat{\mathcal{K}}(\frac{\partial}{\partial t}, W_{j}, \frac{\partial}{\partial t})=\widehat{\mathcal{K}}(W_{j}, \frac{\partial}{\partial t}, \frac{\partial}{\partial t})=0;\nonumber\\
&(3)\widehat{\mathcal{K}}(\frac{\partial}{\partial t}, V_{i}, W_{j})=-\widehat{\mathcal{K}}(V_{i}, W_{j}, \frac{\partial}{\partial t})=\phi^{p_{j}}\frac{\partial \phi^{p_{j}}}{\partial t}g_{F_{j}}(V_{j}, W_{j}),~if~i=j;\nonumber\\
&(4)\widehat{\mathcal{K}}(V_{i}, \frac{\partial}{\partial t}, W_{j})=\phi^{p_{j}}\frac{\partial \phi^{p_{j}}}{\partial t}g_{F_{j}}(V_{j}, W_{j})-\phi^{2p_{j}}g_{F_{j}}(V_{j}, W_{j}),~if~i=j;\nonumber\\
&(5)\widehat{\mathcal{K}}(\frac{\partial}{\partial t}, V_{i}, W_{j})=\widehat{\mathcal{K}}(V_{i}, \frac{\partial}{\partial t}, W_{j})=\widehat{\mathcal{K}}(V_{i}, W_{j}, \frac{\partial}{\partial t})=0,~if~i\neq j;\nonumber\\
&(6)\widehat{\mathcal{K}}(U_{i}, V_{j}, W_{k})=\phi^{2p_{j}}\mathcal{K}_{F_{j}}(U_{j}, V_{j}, W_{j}),~if~i=j=k;\nonumber\\
&(7)\widehat{\mathcal{K}}(U_{i}, V_{j}, W_{k})=0,~if~i=j\neq k~or~i, j, k~are~different.\nonumber
\end{align}
\end{prop}

\begin{prop}
Let $M_{3}=I\times_{\phi^{p_{1}}}F_{1}\times_{\phi^{p_{2}}}F_{2}\times_{\phi^{p_{3}}}F_{3}$ be a degenerate multiply warped product, let the vector fields $\partial/ \partial t\in\Gamma(TI)$ and $U_{j}, V_{j}, W_{j}\in\Gamma(TF_{j}),$ for $j\in \{1, 2, 3\}.$ Let $\widehat{\mathcal{K}}$ be the semi-symmetric non-metric Koszul form on $I\times_{\phi^{p_{1}}}F_{1}\times_{\phi^{p_{2}}}F_{2}\times_{\phi^{p_{3}}}F_{3}$ and $\widehat{\mathcal{K}}_{I}$, $\widehat{\mathcal{K}}_{F_{j}}$ be the lifts of the semi-symmetric non-metric Koszul form on $I,$ $F_{j},$ respectively. Let $P\in\Gamma(TF_{l}),$ then
\begin{align}
&(1)\widehat{\mathcal{K}}(\frac{\partial}{\partial t}, \frac{\partial}{\partial t}, \frac{\partial}{\partial t})=\mathcal{K}_{I}(\frac{\partial}{\partial t}, \frac{\partial}{\partial t}, \frac{\partial}{\partial t})=0;\nonumber\\
&(2)\widehat{\mathcal{K}}(\frac{\partial}{\partial t}, \frac{\partial}{\partial t}, W_{j})=\widehat{\mathcal{K}}(W_{j}, \frac{\partial}{\partial t}, \frac{\partial}{\partial t})=0,~if~j=l;\nonumber\\
&(3)\widehat{\mathcal{K}}(\frac{\partial}{\partial t}, W_{j}, \frac{\partial}{\partial t})=-\phi^{2p_{j}}g_{F_{j}}(W_{j}, P),~if~j=l;\nonumber\\
&(4)\widehat{\mathcal{K}}(\frac{\partial}{\partial t}, \frac{\partial}{\partial t}, W_{j})=\widehat{\mathcal{K}}(\frac{\partial}{\partial t}, W_{j}, \frac{\partial}{\partial t})=\widehat{\mathcal{K}}(W_{j}, \frac{\partial}{\partial t}, \frac{\partial}{\partial t})=0,~if~j\neq l;\nonumber\\
&(5)\widehat{\mathcal{K}}(\frac{\partial}{\partial t}, V_{i}, W_{j})=\widehat{\mathcal{K}}(V_{i}, \frac{\partial}{\partial t}, W_{j})=-\widehat{\mathcal{K}}(V_{i}, W_{j}, \frac{\partial}{\partial t})=\phi^{p_{j}}\frac{\partial \phi^{p_{j}}}{\partial t}g_{F_{j}}(V_{j}, W_{j}),~if~i=j;\nonumber
\end{align}
\begin{align}
&(6)\widehat{\mathcal{K}}(\frac{\partial}{\partial t}, V_{i}, W_{j})=\widehat{\mathcal{K}}(V_{i}, \frac{\partial}{\partial t}, W_{j})=\widehat{\mathcal{K}}(V_{i}, W_{j}, \frac{\partial}{\partial t})=0,~if~i\neq j;\nonumber\\
&(7)\widehat{\mathcal{K}}(U_{i}, V_{j}, W_{k})=\phi^{2p_{j}}\mathcal{K}_{F_{j}}(U_{j}, V_{j}, W_{j})+\phi^{4p_{j}}g_{F_{j}}(U_{j}, W_{j})g_{F_{j}}(V_{j}, P),~if~i=j=k=l;\nonumber\\
&(8)\widehat{\mathcal{K}}(U_{i}, V_{j}, W_{k})=\phi^{2p_{j}}\mathcal{K}_{F_{j}}(U_{j}, V_{j}, W_{j}),~if~i=j= k\neq l;\nonumber\\
&(9)\widehat{\mathcal{K}}(U_{i}, W_{k}, V_{j})=\phi^{2p_{j}}\phi^{2p_{k}}g_{F_{j}}(U_{j}, V_{j})g_{F_{k}}(W_{k}, P),~if~i=j\neq k=l;\nonumber\\
&(10)\widehat{\mathcal{K}}(U_{i}, V_{j}, W_{k})=0,~other~cases.\nonumber
\end{align}
\end{prop}

\begin{thm}
Let $M_{3}=I\times_{\phi^{p_{1}}}F_{1}\times_{\phi^{p_{2}}}F_{2}\times_{\phi^{p_{3}}}F_{3}$ be a degenerate multiply warped product, let the vector fields $\partial/ \partial t\in\Gamma(TI)$ and $U_{j}, V_{j}, W_{j}, Q_{j}\in\Gamma(TF_{j}).$ Let $P=\partial/ \partial t,$ then
\begin{align}
&(1)\widehat{R}(\frac{\partial}{\partial t}, \frac{\partial}{\partial t}, \frac{\partial}{\partial t}, \frac{\partial}{\partial t})=\widehat{R}_{I}(\frac{\partial}{\partial t}, \frac{\partial}{\partial t}, \frac{\partial}{\partial t}, \frac{\partial}{\partial t})=0;\nonumber\\
&(2)\widehat{R}(\frac{\partial}{\partial t}, \frac{\partial}{\partial t}, \frac{\partial}{\partial t}, W_{j})=\widehat{R}(\frac{\partial}{\partial t}, \frac{\partial}{\partial t}, W_{j}, \frac{\partial}{\partial t})=\widehat{R}(\frac{\partial}{\partial t}, W_{j}, \frac{\partial}{\partial t}, \frac{\partial}{\partial t})=0;\nonumber\\
&(3)\widehat{R}(\frac{\partial}{\partial t}, \frac{\partial}{\partial t},  V_{i}, W_{j})=\widehat{R}(V_{i}, W_{j}, \frac{\partial}{\partial t}, \frac{\partial}{\partial t})=0,~if~i=j;\nonumber\\
&(4)\widehat{R}(V_{i}, \frac{\partial}{\partial t}, W_{j}, \frac{\partial}{\partial t})=-\widehat{R}(\frac{\partial}{\partial t}, V_{i}, W_{j}, \frac{\partial}{\partial t})=\phi^{p_{j}}\frac{\partial^{2} \phi^{p_{j}}}{\partial t^{2}}g_{F_{j}}(V_{j}, W_{j})+2\phi^{p_{j}}\frac{\partial \phi^{p_{j}}}{\partial t}g_{F_{j}}(V_{j}, W_{j}),\nonumber\\
&~if~i=j;\nonumber\\
&(5)\widehat{R}(V_{i}, \frac{\partial}{\partial t}, \frac{\partial}{\partial t}, W_{j})=-\widehat{R}(\frac{\partial}{\partial t}, V_{i}, \frac{\partial}{\partial t}, W_{j})=-\phi^{p_{j}}\frac{\partial^{2} \phi^{p_{j}}}{\partial t^{2}}g_{F_{j}}(V_{j}, W_{j}),~if~i=j;\nonumber\\
&(6)\widehat{R}(\frac{\partial}{\partial t}, \frac{\partial}{\partial t},  V_{i}, W_{j})=\widehat{R}(V_{i}, W_{j}, \frac{\partial}{\partial t}, \frac{\partial}{\partial t})=\widehat{R}(\frac{\partial}{\partial t}, V_{i}, W_{j}, \frac{\partial}{\partial t})=\widehat{R}(\frac{\partial}{\partial t}, V_{i}, \frac{\partial}{\partial t}, W_{j})=0,\nonumber\\
&~if~i\neq j;\nonumber\\
&(7)\widehat{R}(\frac{\partial}{\partial t}, V_{i}, W_{j}, U_{k})=\widehat{R}(W_{j}, U_{k}, \frac{\partial}{\partial t}, V_{i})=0,~if~i=j=k;\nonumber\\
&(8)\widehat{R}(W_{j}, U_{k}, V_{i}, \frac{\partial}{\partial t})=-\phi^{2p_{j}}(\mathcal{K}_{F_{j}}(W_{j}, V_{j}, U_{j})-\mathcal{K}_{F_{j}}(U_{j}, V_{j}, W_{j})),~if~i=j=k;\nonumber\\
&(9)\widehat{R}(\frac{\partial}{\partial t}, V_{i}, W_{j}, U_{k})=\widehat{R}(W_{j}, U_{k}, \frac{\partial}{\partial t}, V_{i})=\widehat{R}(W_{j}, U_{k}, V_{i}, \frac{\partial}{\partial t})=0,~other~cases;\nonumber\\
&(10)\widehat{R}(U_{k}, V_{i}, W_{j}, Q_{s})=\phi^{2p_{j}}R_{F_{j}}(U_{j}, V_{j}, W_{j}, Q_{j})+\phi^{2p_{j}}g^{*}_{I}(d\phi^{p_{j}}, d\phi^{p_{j}})(g_{F_{j}}(U_{j}, W_{j})g_{F_{j}}(V_{j}, Q_{j})\nonumber\\
&-g_{F_{j}}(V_{j}, W_{j})g_{F_{j}}(U_{j}, Q_{j})),~if~i=j=k=s;\nonumber\\
&(11)\widehat{R}(U_{k}, V_{i}, W_{j}, Q_{s})=\phi^{p_{i}}\phi^{p_{j}}g^{*}_{I}(d\phi^{p_{i}}, d\phi^{p_{j}})g_{F_{i}}(V_{i}, Q_{i})g_{F_{j}}(U_{j}, W_{j}),~if~k=j\neq i=s;\nonumber\\
&(12)\widehat{R}(U_{k}, V_{i}, W_{j}, Q_{s})=-\phi^{p_{j}}\phi^{p_{k}}g^{*}_{I}(d\phi^{p_{j}}, d\phi^{p_{k}})g_{F_{j}}(V_{j}, W_{j})g_{F_{k}}(U_{k}, Q_{k}),~if~k=s\neq j=i;\nonumber\\
&(13)\widehat{R}(U_{k}, V_{i}, W_{j}, Q_{s})=0,~other~cases.\nonumber
\end{align}
\end{thm}

\begin{thm}
Let $M_{3}=I\times_{\phi^{p_{1}}}F_{1}\times_{\phi^{p_{2}}}F_{2}\times_{\phi^{p_{3}}}F_{3}$ be a degenerate multiply warped product, let the vector fields $\partial/ \partial t\in\Gamma(TI)$ and $U_{j}, V_{j}, W_{j}, Q_{j}\in\Gamma(TF_{j}).$ Let $P\in\Gamma(TF_{l}),$ then
\begin{align}
&(1)\widehat{R}(\frac{\partial}{\partial t}, \frac{\partial}{\partial t}, \frac{\partial}{\partial t}, \frac{\partial}{\partial t})=R_{I}(\frac{\partial}{\partial t}, \frac{\partial}{\partial t}, \frac{\partial}{\partial t}, \frac{\partial}{\partial t})=0;\nonumber\\
&(2)\widehat{R}(\frac{\partial}{\partial t}, \frac{\partial}{\partial t}, \frac{\partial}{\partial t}, W_{j})=\widehat{R}(\frac{\partial}{\partial t}, \frac{\partial}{\partial t}, W_{j}, \frac{\partial}{\partial t})=\widehat{R}(\frac{\partial}{\partial t}, W_{j}, \frac{\partial}{\partial t}, \frac{\partial}{\partial t})=\widehat{R}(W_{j}, \frac{\partial}{\partial t}, \frac{\partial}{\partial t}, \frac{\partial}{\partial t})=0;\nonumber\\
&(3)\widehat{R}(\frac{\partial}{\partial t}, \frac{\partial}{\partial t},  V_{i}, W_{j})=\widehat{R}(V_{i}, W_{j}, \frac{\partial}{\partial t}, \frac{\partial}{\partial t})=0;\nonumber\\
&(4)\widehat{R}(V_{i}, \frac{\partial}{\partial t}, W_{j}, \frac{\partial}{\partial t})=-\widehat{R}(\frac{\partial}{\partial t}, V_{i}, W_{j}, \frac{\partial}{\partial t})=\phi^{p_{j}}\frac{\partial^{2} \phi^{p_{j}}}{\partial t^{2}}g_{F_{j}}(V_{j}, W_{j})-\phi^{2p_{j}}V_{j}(g_{F_{j}}(W_{j}, P)),\nonumber\\
&~if~i=j=l;\nonumber\\
&(5)\widehat{R}(V_{i}, \frac{\partial}{\partial t}, W_{j}, \frac{\partial}{\partial t})=-\widehat{R}(\frac{\partial}{\partial t}, V_{i}, W_{j}, \frac{\partial}{\partial t})=\phi^{p_{j}}\frac{\partial^{2} \phi^{p_{j}}}{\partial t^{2}}g_{F_{j}}(V_{j}, W_{j}),~if~i=j\neq l;\nonumber\\
&(6)\widehat{R}(V_{i}, \frac{\partial}{\partial t}, \frac{\partial}{\partial t}, W_{j})=-\widehat{R}(\frac{\partial}{\partial t}, V_{i}, \frac{\partial}{\partial t}, W_{j})=-\phi^{p_{j}}\frac{\partial^{2} \phi^{p_{j}}}{\partial t^{2}}g_{F_{j}}(V_{j}, W_{j}),~if~i=j=l\nonumber\\
&~or~i=j\neq l;\nonumber\\
&(7)\widehat{R}(V_{i}, \frac{\partial}{\partial t}, W_{j}, \frac{\partial}{\partial t})=\widehat{R}(V_{i}, \frac{\partial}{\partial t}, \frac{\partial}{\partial t}, W_{j})=0,~other~cases;\nonumber\\
&(8)\widehat{R}(\frac{\partial}{\partial t}, V_{i}, W_{j}, U_{k})=2\phi^{3p_{j}}\frac{\partial \phi^{p_{j}}}{\partial t}(g_{F_{j}}(V_{j}, U_{j})g_{F_{j}}(W_{j}, P)+g_{F_{j}}(V_{j}, W_{j})g_{F_{j}}(U_{j}, P)),\nonumber\\
&~if~i=j=k=l;\nonumber\\
&(9)\widehat{R}(\frac{\partial}{\partial t}, V_{i}, W_{j}, U_{k})=\widehat{R}(\frac{\partial}{\partial t}, V_{i}, U_{k}, W_{j})=2\phi^{p_{j}}\phi^{2p_{k}}\frac{\partial \phi^{p_{j}}}{\partial t}g_{F_{j}}(V_{j}, W_{j})g_{F_{k}}(U_{k}, P),\nonumber\\
&~if~i=j\neq k=l;\nonumber\\
&(10)\widehat{R}(\frac{\partial}{\partial t}, V_{i}, W_{j}, U_{k})=\widehat{R}(W_{j}, U_{k}, \frac{\partial}{\partial t}, V_{i})=\widehat{R}(W_{j}, U_{k}, V_{i}, \frac{\partial}{\partial t})=0,~other~cases;\nonumber\\
&(11)\widehat{R}(U_{k}, V_{i}, W_{j}, Q_{s})=\phi^{2p_{j}}R_{F_{j}}(U_{j}, V_{j}, W_{j}, Q_{j})+\phi^{2p_{j}}g^{*}_{I}(d\phi^{p_{j}}, d\phi^{p_{j}})(g_{F_{j}}(U_{j}, W_{j})g_{F_{j}}(V_{j}, Q_{j})\nonumber\\
&-g_{F_{j}}(V_{j}, W_{j})g_{F_{j}}(U_{j}, Q_{j}))+\phi^{4p_{j}}g_{F_{j}}(Q_{j}, P)(\mathcal{K}_{F_{j}}(U_{j}, W_{j}, V_{j})-\mathcal{K}_{F_{j}}(V_{j}, W_{j}, U_{j}))\nonumber\\
&+\phi^{4p_{j}}U_{j}(g_{F_{j}}(W_{j}, P))g_{F_{j}}(V_{j}, Q_{j})-\phi^{4p_{j}}V_{j}(g_{F_{j}}(W_{j}, P))g_{F_{j}}(U_{j}, Q_{j}),~if~i=j=k=s=l;\nonumber\\
&(12)\widehat{R}(U_{k}, V_{i}, W_{j}, Q_{s})=\phi^{2p_{j}}R_{F_{j}}(U_{j}, V_{j}, W_{j}, Q_{j})+\phi^{2p_{j}}g^{*}_{I}(db_{j}, db_{j})(g_{F_{j}}(U_{j}, W_{j})g_{F_{j}}(V_{j}, Q_{j})\nonumber\\
&-g_{F_{j}}(V_{j}, W_{j})g_{F_{j}}(U_{j}, Q_{j}))~if~i=j=k=s\neq l;\nonumber\\
&(13)\widehat{R}(U_{k}, V_{i}, W_{j}, Q_{s})=\phi^{p_{i}}\phi^{p_{j}}g^{*}_{I}(d\phi^{p_{i}}, d\phi^{p_{j}})g_{F_{i}}(V_{i}, Q_{i})g_{F_{j}}(U_{j}, W_{j}),~if~j=k\neq i=s=l\nonumber\\
&or~i=s, j=k, l~are~different;\nonumber\\
&(14)\widehat{R}(U_{k}, V_{i}, W_{j}, Q_{s})=\phi^{p_{i}}\phi^{p_{j}}g^{*}_{I}(d\phi^{p_{i}}, d\phi^{p_{j}})g_{F_{i}}(V_{i}, Q_{i})g_{F_{j}}(U_{j}, W_{j})+\phi^{2p_{i}}\phi^{2p_{j}}U_{j}(g_{F_{j}}(W_{j}, P))\nonumber\\
&\times g_{F_{i}}(V_{i}, Q_{i}),~if~l=j=k\neq i=s;\nonumber\\
&(15)\widehat{R}(U_{k}, V_{i}, W_{j}, Q_{s})=-\phi^{p_{j}}\phi^{p_{k}}g^{*}_{I}(d\phi^{p_{j}}, d\phi^{p_{k}})g_{F_{j}}(V_{j}, W_{j})g_{F_{k}}(U_{k}, Q_{k})-\phi^{2p_{j}}\phi^{2p_{k}}V_{j}(g_{F_{j}}(W_{j}, P))\nonumber\\
&\times g_{F_{k}}(U_{k}, Q_{k}),~if~k=s\neq i=j=l;\nonumber
\end{align}
\begin{align}
&(16)\widehat{R}(U_{k}, V_{i}, W_{j}, Q_{s})=-\phi^{p_{j}}\phi^{p_{k}}g^{*}_{I}(d\phi^{p_{j}}, d\phi^{p_{k}})g_{F_{j}}(V_{j}, W_{j})g_{F_{k}}(U_{k}, Q_{k}),~if~l=k=s\neq i=j\nonumber\\
&or~i=j, k=s, l~are~different;\nonumber\\
&(17)\widehat{R}(U_{k}, V_{i}, W_{j}, Q_{s})=\phi^{2p_{j}}\phi^{2p_{k}}g_{F_{k}}(U_{k}, P)(\mathcal{K}_{F_{j}}(V_{j}, Q_{j}, W_{j})-\mathcal{K}_{F_{j}}(W_{j}, Q_{j}, V_{j})),\nonumber\\
&~if~l=k\neq i=j=s;\nonumber\\
&(18)\widehat{R}(U_{k}, V_{i}, W_{j}, Q_{s})=0,~other~cases.\nonumber
\end{align}
\end{thm}

\begin{prop}
Let $M_{3}=I\times_{\phi^{p_{1}}}F_{1}\times_{\phi^{p_{2}}}F_{2}\times_{\phi^{p_{3}}}F_{3}$ be a degenerate multiply warped product, let the vector fields $\partial/ \partial t\in\Gamma(TI)$ and $U_{j}, V_{j}, W_{j}\in\Gamma(TF_{j}),$ for $j\in \{1, 2, 3\}.$ Let $\widetilde{\mathcal{K}}$ be the almost product Koszul form on $I\times_{\phi^{p_{1}}}F_{1}\times_{\phi^{p_{2}}}F_{2}\times_{\phi^{p_{3}}}F_{3}$ and $\widetilde{\mathcal{K}}_{I}$, $\widetilde{\mathcal{K}}_{F_{j}}$ be the lifts of the almost product Koszul form on $I,$ $F_{j},$ respectively. Then
\begin{align}
&(1)\widetilde{\mathcal{K}}(\frac{\partial}{\partial t}, \frac{\partial}{\partial t}, \frac{\partial}{\partial t})=\widetilde{\mathcal{K}}_{I}(\frac{\partial}{\partial t}, \frac{\partial}{\partial t}, \frac{\partial}{\partial t})=0;\nonumber\\
&(2)\widetilde{\mathcal{K}}(\frac{\partial}{\partial t}, \frac{\partial}{\partial t}, W_{j})=\widetilde{\mathcal{K}}(\frac{\partial}{\partial t}, W_{j}, \frac{\partial}{\partial t})=\widetilde{\mathcal{K}}(W_{j}, \frac{\partial}{\partial t}, \frac{\partial}{\partial t})=0;\nonumber\\
&(3)\widetilde{\mathcal{K}}(\frac{\partial}{\partial t}, V_{i}, W_{j})=\phi^{p_{j}}\frac{\partial \phi^{p_{j}}}{\partial t}g_{F_{j}}(V_{j}, W_{j}),~if~i=j;\nonumber\\
&(4)\widetilde{\mathcal{K}}(V_{i}, \frac{\partial}{\partial t}, W_{j})=-\widetilde{\mathcal{K}}(V_{i}, W_{j}, \frac{\partial}{\partial t})=\frac{1}{2}(\phi^{p_{j}}\frac{\partial \phi^{p_{j}}}{\partial t}g_{F_{j}}(V_{j}, W_{j})+\phi^{p_{j}}J_{I}(\frac{\partial \phi^{p_{j}}}{\partial t})g_{F_{j}}(V_{j}, J_{F_{j}}W_{j})),\nonumber\\
&~if~i=j;\nonumber\\
&(5)\widetilde{\mathcal{K}}(\frac{\partial}{\partial t}, V_{i}, W_{j})=\widetilde{\mathcal{K}}(V_{i}, \frac{\partial}{\partial t}, W_{j})=\widetilde{\mathcal{K}}(V_{i}, W_{j}, \frac{\partial}{\partial t})=0,~if~i\neq j;\nonumber\\
&(6)\widetilde{\mathcal{K}}(U_{i}, V_{j}, W_{k})=\phi^{2p_{j}}\widetilde{\mathcal{K}}_{F_{j}}(U_{j}, V_{j}, W_{j}),~if~i=j=k;\nonumber\\
&(7)\widetilde{\mathcal{K}}(U_{i}, V_{j}, W_{k})=\widetilde{\mathcal{K}}(U_{i}, W_{k}, V_{j})=\widetilde{\mathcal{K}}(W_{k}, U_{i}, V_{j})=0,~if~i=j\neq k;\nonumber\\
&(8)\widetilde{\mathcal{K}}(U_{i}, V_{j}, W_{k})=0,~where~i, j, k~are~different.\nonumber
\end{align}
\end{prop}

\begin{thm}
Let $M_{3}=I\times_{\phi^{p_{1}}}F_{1}\times_{\phi^{p_{2}}}F_{2}\times_{\phi^{p_{3}}}F_{3}$ be a degenerate multiply warped product, let the vector fields $\partial/ \partial t\in\Gamma(TI)$ and $U_{j}, V_{j}, W_{j}, Q_{j}\in\Gamma(TF_{j}),$ then
\begin{align}
&(1)\widetilde{R}(\frac{\partial}{\partial t}, \frac{\partial}{\partial t}, \frac{\partial}{\partial t}, \frac{\partial}{\partial t})=\widetilde{R}_{I}(\frac{\partial}{\partial t}, \frac{\partial}{\partial t}, \frac{\partial}{\partial t}, \frac{\partial}{\partial t})=0;\nonumber\\
&(2)\widetilde{R}(\frac{\partial}{\partial t}, \frac{\partial}{\partial t}, \frac{\partial}{\partial t}, W_{j})=\widetilde{R}(\frac{\partial}{\partial t}, W_{j}, \frac{\partial}{\partial t}, \frac{\partial}{\partial t})=0;\nonumber\\
&(3)\widetilde{R}(\frac{\partial}{\partial t}, \frac{\partial}{\partial t},  V_{i}, W_{j})=\widetilde{R}(V_{i}, W_{j}, \frac{\partial}{\partial t}, \frac{\partial}{\partial t})=0;\nonumber\\
&(4)\widetilde{R}(\frac{\partial}{\partial t}, V_{i}, \frac{\partial}{\partial t}, W_{j})=\frac{1}{2}\phi^{p_{j}}\frac{\partial^{2} \phi^{p_{j}}}{\partial t^{2}}g_{F_{j}}(V_{j}, W_{j})+\frac{1}{2}\phi^{p_{j}}\frac{\partial }{\partial t}(J_{I}(\frac{\partial \phi^{p_{j}}}{\partial t}))g_{F_{j}}(V_{j}, J_{F_{j}}W_{j}),~if~i=j;\nonumber\\
&(5)\widetilde{R}(\frac{\partial}{\partial t}, V_{i}, \frac{\partial}{\partial t}, W_{j})=0,~if~i\neq j;\nonumber\\
&(6)\widetilde{R}(\frac{\partial}{\partial t}, V_{i}, W_{j}, U_{k})=0;\nonumber
\end{align}
\begin{align}
&(7)\widetilde{R}(U_{k}, V_{i}, \frac{\partial}{\partial t}, W_{j})=\frac{1}{4}\phi^{p_{j}}\frac{\partial \phi^{p_{j}}}{\partial t}(-\mathcal{K}_{F_{j}}(V_{j}, W_{j}, U_{j})+\mathcal{K}_{F_{j}}(U_{j}, W_{j}, V_{j})+\mathcal{K}_{F_{j}}(V_{j}, J_{F_{j}}W_{j}, J_{F_{j}}U_{j})\nonumber\\
&-\mathcal{K}_{F_{j}}(U_{j}, J_{F_{j}}W_{j}, J_{F_{j}}V_{j}))-\frac{1}{4}\phi^{p_{j}}J_{I}(\frac{\partial \phi^{p_{j}}}{\partial t})(-\mathcal{K}_{F_{j}}(V_{j}, W_{j}, J_{F_{j}}U_{j})+\mathcal{K}_{F_{j}}(U_{j}, W_{j}, J_{F_{j}}V_{j})\nonumber\\
&+\mathcal{K}_{F_{j}}(V_{j}, J_{F_{j}}W_{j}, U_{j})-\mathcal{K}_{F_{j}}(U_{j}, J_{F_{j}}W_{j}, V_{j})),~if~i=j=k;\nonumber\\
&(8)\widetilde{R}(U_{k}, V_{i}, \frac{\partial}{\partial t}, W_{j})=0,~other~cases;\nonumber\\
&(9)\widetilde{R}(U_{k}, V_{i}, W_{j}, Q_{s})=\phi^{2p_{j}}\widetilde{R}_{F_{j}}(U_{j}, V_{j}, W_{j}, Q_{j})+\frac{1}{4}\phi^{2p_{j}}[g^{*}_{I}(d\phi^{p_{j}}, d\phi^{p_{j}})(g_{F_{j}}(U_{j}, W_{j})g_{F_{j}}(V_{j}, Q_{j})\nonumber\\
&-g_{F_{j}}(V_{j}, W_{j})g_{F_{j}}(U_{j}, Q_{j}))+g^{*}_{I}(d\phi^{p_{j}}\circ J_{I}, d\phi^{p_{j}})(g_{F_{j}}(U_{j}, J_{F_{j}}W_{j})g_{F_{j}}(V_{j}, Q_{j})\nonumber\\
&-g_{F_{j}}(V_{j}, J_{F_{j}}W_{j})g_{F_{j}}(U_{j}, Q_{j})+g_{F_{j}}(U_{j}, W_{j})g_{F_{j}}(V_{j}, J_{F_{j}}Q_{j})-g_{F_{j}}(V_{j}, W_{j})g_{F_{j}}(U_{j}, J_{F_{j}}Q_{j}))\nonumber\\
&+g^{*}_{I}(d\phi^{p_{j}}\circ J_{I}, d\phi^{p_{j}}\circ J_{I})(g_{F_{j}}(U_{j}, J_{F_{j}}W_{j})g_{F_{j}}(V_{j}, J_{F_{j}}Q_{j})-g_{F_{j}}(V_{j}, J_{F_{j}}W_{j})g_{F_{j}}(U_{j}, J_{F_{j}}Q_{j}))],\nonumber\\
&~if~i=j=k=s;\nonumber\\
&(10)\widetilde{R}(U_{k}, V_{i}, W_{j}, Q_{s})=\frac{1}{4}\phi^{p_{i}}\phi^{p_{j}}(g^{*}_{I}(d\phi^{p_{i}}, d\phi^{p_{j}})g_{F_{i}}(V_{i}, Q_{i})g_{F_{j}}(U_{j}, W_{j})+g^{*}_{I}(d\phi^{p_{i}}\circ J_{I}, d\phi^{p_{j}})\nonumber\\
&\times g_{F_{i}}(V_{i}, J_{F_{i}}Q_{i})g_{F_{j}}(U_{j}, W_{j})+g^{*}_{I}(d\phi^{p_{i}}, d\phi^{p_{j}}\circ J_{I})g_{F_{i}}(V_{i}, Q_{i})g_{F_{j}}(U_{j}, J_{F_{j}}W_{j})\nonumber\\
&+g^{*}_{I}(d\phi^{p_{i}}\circ J_{I}, d\phi^{p_{j}}\circ J_{I})g_{F_{i}}(V_{i}, J_{F_{j}}Q_{i})g_{F_{j}}(U_{j}, J_{F_{j}}W_{j})),~if~j=k\neq i=s;\nonumber\\
&(11)\widetilde{R}(U_{k}, V_{i}, W_{j}, Q_{s})=0,~other~cases.\nonumber
\end{align}
\end{thm}

Let $M_{4}=(0, 1)\times_{b}F$ be a degenerate twisted product with the metric tensor $g=-dt^{2}\oplus b^{2}g_{F}$ where $t\in(0, 1),$ $b: (0, 1)\times F\rightarrow \mathbb{R}.$

\begin{prop}
Let $M_{4}=(0, 1)\times_{b}F$ be a degenerate twisted product, let the vector fields $\partial/ \partial t\in\Gamma(T(0, 1))$ and $U, V, W\in\Gamma(TF).$ Let $\overline{\mathcal{K}}$ be the semi-symmetric metric Koszul form on $(0, 1)\times_{b}F$ and $\overline{\mathcal{K}}_{(0, 1)},$ $\overline{\mathcal{K}}_{F}$ be the lifts of the semi-symmetric metric Koszul form on $(0, 1),$ $F,$ respectively. Let $P=\partial/ \partial t,$ then
\begin{align}
&(1)\overline{\mathcal{K}}(\frac{\partial}{\partial t}, \frac{\partial}{\partial t}, \frac{\partial}{\partial t})=\overline{\mathcal{K}}_{(0, 1)}(\frac{\partial}{\partial t}, \frac{\partial}{\partial t}, \frac{\partial}{\partial t})=0;\nonumber\\
&(2)\overline{\mathcal{K}}(\frac{\partial}{\partial t}, \frac{\partial}{\partial t}, W)=\overline{\mathcal{K}}(\frac{\partial}{\partial t}, W, \frac{\partial}{\partial t})=\overline{\mathcal{K}}(W, \frac{\partial}{\partial t}, \frac{\partial}{\partial t})=0;\nonumber\\
&(3)\overline{\mathcal{K}}(\frac{\partial}{\partial t}, V, W)=b\frac{\partial b}{\partial t}g_{F}(V, W);\nonumber\\
&(4)\overline{\mathcal{K}}(V, \frac{\partial}{\partial t}, W)=-\overline{\mathcal{K}}(V, W, \frac{\partial}{\partial t})=b\frac{\partial b}{\partial t}g_{F}(V, W)-b^{2}g_{F}(V, W);\nonumber\\
&(5)\overline{\mathcal{K}}(U, V, W)=bU(b)g_{F}(V, W)+bV(b)g_{F}(W, U)-bW(b)g_{F}(U, V)+b^{2}\mathcal{K}_{F}(U, V, W).\nonumber
\end{align}
\end{prop}

\begin{prop}
Let $M_{4}=(0, 1)\times_{b}F$ be a degenerate twisted product, let the vector fields $\partial/ \partial t\in\Gamma(T(0, 1))$ and $U, V, W\in\Gamma(TF).$ Let $\overline{\mathcal{K}}$ be the semi-symmetric metric Koszul form on $(0, 1)\times_{b}F$ and $\overline{\mathcal{K}}_{(0, 1)},$ $\overline{\mathcal{K}}_{F}$ be the lifts of the semi-symmetric metric Koszul form on $(0, 1),$ $F,$ respectively. Let $P\in\Gamma(TF),$ then
\begin{align}
&(1)\overline{\mathcal{K}}(\frac{\partial}{\partial t}, \frac{\partial}{\partial t}, \frac{\partial}{\partial t})=\mathcal{K}_{(0, 1)}(\frac{\partial}{\partial t}, \frac{\partial}{\partial t}, \frac{\partial}{\partial t})=0;\nonumber\\
&(2)\overline{\mathcal{K}}(\frac{\partial}{\partial t}, \frac{\partial}{\partial t}, W)=-\overline{\mathcal{K}}(\frac{\partial}{\partial t}, W, \frac{\partial}{\partial t})=b^{2}g_{F}(P, W);\nonumber\\
&(3)\overline{\mathcal{K}}(W, \frac{\partial}{\partial t}, \frac{\partial}{\partial t})=0;\nonumber\\
&(4)\overline{\mathcal{K}}(\frac{\partial}{\partial t}, V, W)=\overline{\mathcal{K}}(V, \frac{\partial}{\partial t}, W)=-\overline{\mathcal{K}}(V, W, \frac{\partial}{\partial t})=b\frac{\partial b}{\partial t}g_{F}(V, W);\nonumber\\
&(5)\overline{\mathcal{K}}(U, V, W)=bU(b)g_{F}(V, W)+bV(b)g_{F}(W, U)-bW(b)g_{F}(U, V)+b^{2}\mathcal{K}_{F}(U, V, W)\nonumber\\
&+b^{4}g_{F}(V, P)g_{F}(U, W)-b^{4}g_{F}(W, P)g_{F}(U, V).\nonumber
\end{align}
\end{prop}

\begin{thm}
Let $M_{4}=(0, 1)\times_{b}F$ be a degenerate twisted product, let the vector fields $\partial/ \partial t\in\Gamma(T(0, 1))$and $U, V, W, Q\in\Gamma(TF).$ Let $P=\partial/ \partial t,$ then
\begin{align}
&(1)\overline{R}(\frac{\partial}{\partial t}, \frac{\partial}{\partial t}, \frac{\partial}{\partial t}, \frac{\partial}{\partial t})=\overline{R}_{(0, 1)}(\frac{\partial}{\partial t}, \frac{\partial}{\partial t}, \frac{\partial}{\partial t}, \frac{\partial}{\partial t})=0;\nonumber\\
&(2)\overline{R}(\frac{\partial}{\partial t}, \frac{\partial}{\partial t}, \frac{\partial}{\partial t}, W)=\overline{R}(\frac{\partial}{\partial t}, W, \frac{\partial}{\partial t}, \frac{\partial}{\partial t})=0;\nonumber\\
&(3)\overline{R}(\frac{\partial}{\partial t}, \frac{\partial}{\partial t},  V, W)=\overline{R}(V, W, \frac{\partial}{\partial t}, \frac{\partial}{\partial t})=0;\nonumber\\
&(4)\overline{R}(\frac{\partial}{\partial t}, V, \frac{\partial}{\partial t}, W)=b\frac{\partial^{2} b}{\partial t^{2}}g_{F}(V, W)-b\frac{\partial b}{\partial t}g_{F}(V, W);\nonumber\\
&(5)\overline{R}(\frac{\partial}{\partial t}, V, U, W)=\overline{R}(U, W, \frac{\partial}{\partial t}, V)=bU(\frac{\partial b}{\partial t})g_{F}(W, V)-U(b)\frac{\partial b}{\partial t}g_{F}(W, V)\nonumber\\
&-bW(\frac{\partial b}{\partial t})g_{F}(U, V)+W(b)\frac{\partial b}{\partial t}g_{F}(U, V);\nonumber\\
&(6)\overline{R}(U, V, W, Q)=b^{2}R_{F}(U, V, W, Q)+(b^{2}g^{*}_{(0, 1)}(db, db)-g^{*}_{F}(db, db)+2b^{3}\frac{\partial b}{\partial t}-b^{4})\nonumber\\
&\times (g_{F}(U, W)g_{F}(V, Q)-g_{F}(V, W)g_{F}(U, Q))
+(bUW(b)-2U(b)W(b)-bg^{*}_{F}(\nabla^{\flat}_{U}W, db))\nonumber\\
&\times g_{F}(V, Q)+(bVQ(b)-2V(b)Q(b)-bg^{*}_{F}(\nabla^{\flat}_{V}Q, db))g_{F}(U, W)
-(bVW(b)-2V(b)W(b)\nonumber\\&-bg^{*}_{F}(\nabla^{\flat}_{V}W, db))g_{F}(U, Q)
-(bUQ(b)-2U(b)Q(b)-bg^{*}_{F}(\nabla^{\flat}_{U}Q, db))g_{F}(V, W).\nonumber
\end{align}
\end{thm}

\begin{thm}
Let $M_{4}=(0, 1)\times_{b}F$ be a degenerate twisted product, let the vector fields $\partial/ \partial t\in\Gamma(T(0, 1))$and $U, V, W, Q\in\Gamma(TF).$ Let $P\in\Gamma(TF),$ then
\begin{align}
&(1)\overline{R}(\frac{\partial}{\partial t}, \frac{\partial}{\partial t}, \frac{\partial}{\partial t}, \frac{\partial}{\partial t})=0;\nonumber\\
&(2)\overline{R}(\frac{\partial}{\partial t}, \frac{\partial}{\partial t}, \frac{\partial}{\partial t}, W)=\overline{R}(\frac{\partial}{\partial t}, W, \frac{\partial}{\partial t}, \frac{\partial}{\partial t})=0;\nonumber
\end{align}
\begin{align}
&(3)\overline{R}(\frac{\partial}{\partial t}, \frac{\partial}{\partial t},  V, W)=\overline{R}(V, W, \frac{\partial}{\partial t}, \frac{\partial}{\partial t})=0;\nonumber\\
&(4)\overline{R}(\frac{\partial}{\partial t}, V, \frac{\partial}{\partial t}, W)=b\frac{\partial^{2} b}{\partial t^{2}}g_{F}(V, W)+b^{4}g_{F}(V, P)g_{F}(P, W)-b^{4}g_{F}(V, W)g_{F}(P, P)\nonumber\\
&-bV(b)g_{F}(P, W)-bP(b)g_{F}(W, V)+bW(b)g_{F}(V, P)-b^{2}\mathcal{K}_{F}(V, P, W);\nonumber\\
&(5)\overline{R}(\frac{\partial}{\partial t}, V, U, W)=bU(\frac{\partial b}{\partial t})g_{F}(W, V)-U(b)\frac{\partial b}{\partial t}g_{F}(W, V)-bW(\frac{\partial b}{\partial t})g_{F}(U, V)\nonumber\\
&+W(b)\frac{\partial b}{\partial t}g_{F}(U, V)+b^{3}\frac{\partial b}{\partial t}g_{F}(U, P)g_{F}(W, V)-b^{3}\frac{\partial b}{\partial t}g_{F}(W, P)g_{F}(U, V);\nonumber\\
&(6)\overline{R}(U, W, \frac{\partial}{\partial t}, V)=bU(\frac{\partial b}{\partial t})g_{F}(W, V)-U(b)\frac{\partial b}{\partial t}g_{F}(W, V)-bW(\frac{\partial b}{\partial t})g_{F}(U, V)\nonumber\\
&+W(b)\frac{\partial b}{\partial t}g_{F}(U, V)-b^{3}\frac{\partial b}{\partial t}g_{F}(U, P)g_{F}(W, V)+b^{3}\frac{\partial b}{\partial t}g_{F}(W, P)g_{F}(U, V);\nonumber\\
&(7)\overline{R}(U, V, W, Q)=b^{2}R_{F}(U, V, W, Q)+(b^{2}g^{*}_{(0, 1)}(db, db)+g^{*}_{F}(db, db)+2b^{3}P(b)+b^{6}g_{F}(P, P))\nonumber\\
&\times (g_{F}(U, W)g_{F}(V, Q)-g_{F}(V, W)g_{F}(U, Q))
+(bUW(b)-2U(b)W(b)+b^{4}\mathcal{K}_{F}(U, P, W)\nonumber\\
&-bg^{*}_{F}(\nabla^{\flat}_{U}W, db))g_{F}(V, Q)+(bVQ(b)-2V(b)Q(b)+b^{4}\mathcal{K}_{F}(V, P, Q)\nonumber\\
&-bg^{*}_{F}(\nabla^{\flat}_{V}Q, db))g_{F}(U, W)-(bVW(b)-2V(b)W(b)+b^{4}\mathcal{K}_{F}(V, P, W)\nonumber\\
&-bg^{*}_{F}(\nabla^{\flat}_{V}W, db))g_{F}(U, Q)-(bUQ(b)-2U(b)Q(b)+b^{4}\mathcal{K}_{F}(U, P, Q)\nonumber\\
&-bg^{*}_{F}(\nabla^{\flat}_{U}Q, db))g_{F}(V, W)+b^{3}Q(b)(g_{F}(U, P)g_{F}(V, W)-g_{F}(V, P)g_{F}(U, W))\nonumber\\
&+(b^{6}g_{F}(U, P)-b^{3}U(b))(g_{F}(V, W)g_{F}(P, Q)-g_{F}(P, W)g_{F}(V, Q))\nonumber\\
&-(b^{6}g_{F}(V, P)-b^{3}V(b))(g_{F}(U, W)g_{F}(P, Q)-g_{F}(P, W)g_{F}(U, Q))\nonumber\\
&-b^{3}W(b)(g_{F}(U, P)g_{F}(V, Q)-g_{F}(V, P)g_{F}(U, Q)).\nonumber
\end{align}
\end{thm}

\begin{prop}
Let $M_{4}=(0, 1)\times_{b}F$ be a degenerate twisted product, let the vector fields $\partial/ \partial t\in\Gamma(T(0, 1))$ and $U, V, W\in\Gamma(TF).$ Let $\widehat{\mathcal{K}}$ be the semi-symmetric non-metric Koszul form on $(0, 1)\times_{b}F$ and $\widehat{\mathcal{K}}_{(0, 1)}$, $\widehat{\mathcal{K}}_{F}$ be the lifts of the semi-symmetric non-metric Koszul form on $(0, 1),$ $F,$ respectively. Let $P=\partial/ \partial t,$ then
\begin{align}
&(1)\widehat{\mathcal{K}}(\frac{\partial}{\partial t}, \frac{\partial}{\partial t}, \frac{\partial}{\partial t})=\widehat{\mathcal{K}}_{(0, 1)}(\frac{\partial}{\partial t}, \frac{\partial}{\partial t}, \frac{\partial}{\partial t})=1;\nonumber\\
&(2)\widehat{\mathcal{K}}(\frac{\partial}{\partial t}, \frac{\partial}{\partial t}, W)=\widehat{\mathcal{K}}(\frac{\partial}{\partial t}, W, \frac{\partial}{\partial t})=\widehat{\mathcal{K}}(W, \frac{\partial}{\partial t}, \frac{\partial}{\partial t})=0;\nonumber\\
&(3)\widehat{\mathcal{K}}(\frac{\partial}{\partial t}, V, W)=-\widehat{\mathcal{K}}(V, W, \frac{\partial}{\partial t})=b\frac{\partial b}{\partial t}g_{F}(V, W);\nonumber\\
&(4)\widehat{\mathcal{K}}(V, \frac{\partial}{\partial t}, W)=b\frac{\partial b}{\partial t}g_{F}(V, W)-b^{2}g_{F}(V, W);\nonumber\\
&(5)\widehat{\mathcal{K}}(U, V, W)=bU(b)g_{F}(V, W)+bV(b)g_{F}(W, U)-bW(b)g_{F}(U, V)+b^{2}\mathcal{K}_{F}(U, V, W).\nonumber
\end{align}
\end{prop}

\begin{prop}
Let $M_{4}=(0, 1)\times_{b}F$ be a degenerate twisted product, let the vector fields $\partial/ \partial t\in\Gamma(T(0, 1))$ and $U, V, W\in\Gamma(TF).$ Let $\widehat{\mathcal{K}}$ be the semi-symmetric non-metric Koszul form on $(0, 1)\times_{b}F$ and $\widehat{\mathcal{K}}_{(0, 1)}$, $\widehat{\mathcal{K}}_{F}$ be the lifts of the semi-symmetric non-metric Koszul form on $(0, 1),$ $F,$ respectively. Let $P\in\Gamma(TF),$ then
\begin{align}
&(1)\widehat{\mathcal{K}}(\frac{\partial}{\partial t}, \frac{\partial}{\partial t}, \frac{\partial}{\partial t})=\mathcal{K}_{(0, 1)}(\frac{\partial}{\partial t}, \frac{\partial}{\partial t}, \frac{\partial}{\partial t})=0;\nonumber\\
&(2)\widehat{\mathcal{K}}(\frac{\partial}{\partial t}, \frac{\partial}{\partial t}, W)=\widehat{\mathcal{K}}(W, \frac{\partial}{\partial t}, \frac{\partial}{\partial t})=0;\nonumber\\
&(3)\widehat{\mathcal{K}}(\frac{\partial}{\partial t}, W, \frac{\partial}{\partial t})=-b^{2}g_{F}(W, P);\nonumber\\
&(4)\widehat{\mathcal{K}}(\frac{\partial}{\partial t}, V, W)=\widehat{\mathcal{K}}(V, \frac{\partial}{\partial t}, W)=-\widehat{\mathcal{K}}(V, W, \frac{\partial}{\partial t})=b\frac{\partial b}{\partial t}g_{F}(V, W);\nonumber\\
&(5)\widehat{\mathcal{K}}(U, V, W)=bU(b)g_{F}(V, W)+bV(b)g_{F}(W, U)-bW(b)g_{F}(U, V)+b^{2}\mathcal{K}_{F}(U, V, W)\nonumber\\
&+b^{4}g_{F}(V, P)g_{F}(U, W).\nonumber
\end{align}
\end{prop}

\begin{thm}
Let $M_{4}=(0, 1)\times_{b}F$ be a degenerate twisted product, let the vector fields $\partial/ \partial t\in\Gamma(T(0, 1))$and $U, V, W, Q\in\Gamma(TF).$ Let $P=\partial/ \partial t,$ then
\begin{align}
&(1)\widehat{R}(\frac{\partial}{\partial t}, \frac{\partial}{\partial t}, \frac{\partial}{\partial t}, \frac{\partial}{\partial t})=\widehat{R}_{(0, 1)}(\frac{\partial}{\partial t}, \frac{\partial}{\partial t}, \frac{\partial}{\partial t}, \frac{\partial}{\partial t})=0;\nonumber\\
&(2)\widehat{R}(\frac{\partial}{\partial t}, \frac{\partial}{\partial t}, \frac{\partial}{\partial t}, W)=\widehat{R}(\frac{\partial}{\partial t}, \frac{\partial}{\partial t}, W, \frac{\partial}{\partial t})=\widehat{R}(\frac{\partial}{\partial t}, W, \frac{\partial}{\partial t}, \frac{\partial}{\partial t})=0;\nonumber\\
&(3)\widehat{R}(\frac{\partial}{\partial t}, \frac{\partial}{\partial t}, V, W)=\widehat{R}(V, W, \frac{\partial}{\partial t}, \frac{\partial}{\partial t})=0;\nonumber\\
&(4)\widehat{R}(V, \frac{\partial}{\partial t}, W, \frac{\partial}{\partial t})=-\widehat{R}(\frac{\partial}{\partial t}, V, W, \frac{\partial}{\partial t})=b\frac{\partial^{2} b}{\partial t^{2}}g_{F}(V, W)+2b\frac{\partial b}{\partial t}g_{F}(V, W);\nonumber\\
&(5)\widehat{R}(V, \frac{\partial}{\partial t}, \frac{\partial}{\partial t}, W)=-\widehat{R}(\frac{\partial}{\partial t}, V, \frac{\partial}{\partial t}, W)=-b\frac{\partial^{2} b}{\partial t^{2}}g_{F}(V, W);\nonumber\\
&(6)\widehat{R}(\frac{\partial}{\partial t}, U, V, W)=\widehat{R}(V, W, \frac{\partial}{\partial t}, U)=bV(\frac{\partial b}{\partial t})g_{F}(W, U)-V(b)\frac{\partial b}{\partial t}g_{F}(W, U)\nonumber\\
&-bW(\frac{\partial b}{\partial t})g_{F}(U, V)+W(b)\frac{\partial b}{\partial t}g_{F}(U, V);\nonumber\\
&(7)\widehat{R}(V, W, U, \frac{\partial}{\partial t})=-bV(\frac{\partial b}{\partial t})g_{F}(W, U)+V(b)\frac{\partial b}{\partial t}g_{F}(W, U)+bW(\frac{\partial b}{\partial t})g_{F}(U, V)\nonumber\\
&-W(b)\frac{\partial b}{\partial t}g_{F}(U, V)+2bW(b)g_{F}(U, V)-2bV(b)g_{F}(U, W)-b^{2}(\mathcal{K}_{F}(W, U, V)-\mathcal{K}_{F}(V, U, W));\nonumber\\
&(8)\widehat{R}(U, V, W, Q)=b^{2}R_{F}(U, V, W, Q)+(b^{2}g^{*}_{(0, 1)}(db, db)+g^{*}_{F}(db, db))(g_{F}(U, W)g_{F}(V, Q)\nonumber\\
&-g_{F}(V, W)g_{F}(U, Q))+(bUW(b)-2U(b)W(b)-bg^{*}_{F}(\nabla^{\flat}_{U}W, db))g_{F}(V, Q)\nonumber\\
&+(bVQ(b)-2V(b)Q(b)-bg^{*}_{F}(\nabla^{\flat}_{V}Q, db))g_{F}(U, W)-(bVW(b)-2V(b)W(b)\nonumber\\
&-bg^{*}_{F}(\nabla^{\flat}_{V}W, db))g_{F}(U, Q)-(bUQ(b)-2U(b)Q(b)-bg^{*}_{F}(\nabla^{\flat}_{U}Q, db))g_{F}(V, W).\nonumber
\end{align}
\end{thm}

\begin{thm}
Let $M_{4}=(0, 1)\times_{b}F$ be a degenerate twisted product, let the vector fields $\partial/ \partial t\in\Gamma(T(0, 1))$and $U, V, W, Q\in\Gamma(TF).$ Let $P\in\Gamma(TF),$ then
\begin{align}
&(1)\widehat{R}(\frac{\partial}{\partial t}, \frac{\partial}{\partial t}, \frac{\partial}{\partial t}, \frac{\partial}{\partial t})=R_{(0, 1)}(\frac{\partial}{\partial t}, \frac{\partial}{\partial t}, \frac{\partial}{\partial t}, \frac{\partial}{\partial t})=0;\nonumber\\
&(2)\widehat{R}(\frac{\partial}{\partial t}, \frac{\partial}{\partial t}, \frac{\partial}{\partial t}, W)=\widehat{R}(\frac{\partial}{\partial t}, \frac{\partial}{\partial t}, W, \frac{\partial}{\partial t})=\widehat{R}(\frac{\partial}{\partial t}, W, \frac{\partial}{\partial t}, \frac{\partial}{\partial t})=0;\nonumber\\
&(3)\widehat{R}(\frac{\partial}{\partial t}, \frac{\partial}{\partial t}, V, W)=\widehat{R}(V, W, \frac{\partial}{\partial t}, \frac{\partial}{\partial t})=0;\nonumber\\
&(4)\widehat{R}(\frac{\partial}{\partial t}, V, W, \frac{\partial}{\partial t})=-\widehat{R}(V, \frac{\partial}{\partial t}, W, \frac{\partial}{\partial t})=-b\frac{\partial^{2} b}{\partial t}g_{F}(V, W)+2bV(b)g_{F}(W, P)+b^{2}V(g_{F}(W, P));\nonumber\\
&(5)\widehat{R}(\frac{\partial}{\partial t}, V, \frac{\partial}{\partial t}, W)=-\widehat{R}(V, \frac{\partial}{\partial t}, \frac{\partial}{\partial t}, W)=b\frac{\partial^{2} b}{\partial t^{2}}g_{F}(V, W);\nonumber\\
&(6)\widehat{R}(\frac{\partial}{\partial t}, V, U, W)=-\widehat{R}(V, \frac{\partial}{\partial t}, U, W)=bU(\frac{\partial b}{\partial t})g_{F}(W, V)-U(b)\frac{\partial b}{\partial t}g_{F}(W, V)\nonumber\\
&-bW(\frac{\partial b}{\partial t})g_{F}(V, U)+W(b)\frac{\partial b}{\partial t}g_{F}(V, U)+2b^{3}\frac{\partial b}{\partial t}(g_{F}(U, P)g_{F}(V, W)\nonumber\\
&+g_{F}(W, P)g_{F}(U, V));\nonumber\\
&(7)\widehat{R}(U, W, \frac{\partial}{\partial t}, V)=-\widehat{R}(U, W, V, \frac{\partial}{\partial t})=bU(\frac{\partial b}{\partial t})g_{F}(W, V)-U(b)\frac{\partial b}{\partial t}g_{F}(W, V)\nonumber\\
&-bW(\frac{\partial b}{\partial t})g_{F}(V, U)+W(b)\frac{\partial b}{\partial t}g_{F}(V, U);\nonumber\\
&(8)\widehat{R}(U, V, W, Q)=b^{2}R_{F}(U, V, W, Q)+(b^{2}g^{*}_{(0, 1)}(db, db)+g^{*}_{F}(db, db))(g_{F}(U, W)g_{F}(V, Q)\nonumber\\
&-g_{F}(V, W)g_{F}(U, Q))+(bUW(b)-2U(b)W(b)-bg^{*}_{F}(\nabla^{\flat}_{U}W, db))g_{F}(V, Q)\nonumber\\
&+(bVQ(b)-2V(b)Q(b)-bg^{*}_{F}(\nabla^{\flat}_{V}Q, db))g_{F}(U, W)-(bVW(b)-2V(b)W(b)\nonumber\\
&-bg^{*}_{F}(\nabla^{\flat}_{V}W, db))g_{F}(U, Q)-(bUQ(b)-2U(b)Q(b)-bg^{*}_{F}(\nabla^{\flat}_{U}Q, db))g_{F}(V, W)\nonumber\\
&+2b^{3}U(b)(g_{F}(W, P)g_{F}(V, Q)+g_{F}(Q, P)g_{F}(V, W))-2b^{3}V(b)(g_{F}(W, P)g_{F}(U, Q)\nonumber\\
&+g_{F}(Q, P)g_{F}(U, W))+b^{4}U(g_{F}(W, P))g_{F}(V, Q)-b^{4}V(g_{F}(W, P))g_{F}(U, Q)\nonumber\\
&+b^{4}g_{F}(P, Q)(\mathcal{K}_{F}(U, W, V)-\mathcal{K}_{F}(V, W, U)).\nonumber
\end{align}
\end{thm}

\begin{prop}
Let $M_{4}=(0, 1)\times_{b}F$ be a degenerate twisted product, let the vector fields $\partial/ \partial t\in\Gamma(T(0, 1))$ and $U, V, W\in\Gamma(TF).$  Let $\widetilde{\mathcal{K}}$ be the almost product Koszul form on $(0, 1)\times_{b}F$ and $\widetilde{\mathcal{K}}_{(0, 1)}$, $\widetilde{\mathcal{K}}_{F}$ be the lifts of the almost product Koszul form on $(0, 1),$ $F,$ respectively. Then
\begin{align}
&(1)\widetilde{\mathcal{K}}(\frac{\partial}{\partial t}, \frac{\partial}{\partial t}, \frac{\partial}{\partial t})=\widetilde{\mathcal{K}}_{(0, 1)}(\frac{\partial}{\partial t}, \frac{\partial}{\partial t}, \frac{\partial}{\partial t})=0;\nonumber\\
&(2)\widetilde{\mathcal{K}}(\frac{\partial}{\partial t}, \frac{\partial}{\partial t}, W)=\widetilde{\mathcal{K}}(\frac{\partial}{\partial t}, W, \frac{\partial}{\partial t})=\widetilde{\mathcal{K}}(W, \frac{\partial}{\partial t}, \frac{\partial}{\partial t})=0;\nonumber
\end{align}
\begin{align}
&(3)\widetilde{\mathcal{K}}(\frac{\partial}{\partial t}, V, W)=b\frac{\partial b}{\partial t}g_{F}(V, W);\nonumber\\
&(4)\widetilde{\mathcal{K}}(V, \frac{\partial}{\partial t}, W)=-\widetilde{\mathcal{K}}(V, W, \frac{\partial}{\partial t})=\frac{1}{2}(b\frac{\partial b}{\partial t}g_{F}(V, W)+bJ_{(0, 1)}(\frac{\partial b}{\partial t})g_{F}(V, J_{F}W));\nonumber\\
&(5)\widetilde{\mathcal{K}}(U, V, W)=bU(b)g_{F}(V, W)+\frac{1}{2}bV(b)g_{F}(U, W)-\frac{1}{2}bW(b)g_{F}(U, V)\nonumber\\
&+\frac{1}{2}bJ_{F}V(b)g_{F}(U, J_{F}W)-\frac{1}{2}bJ_{F}W(b)g_{F}(U, J_{F}V)+b^{2}\widetilde{\mathcal{K}}_{F}(U, V, W).\nonumber
\end{align}
\end{prop}

\begin{thm}
Let $M_{4}=(0, 1)\times_{b}F$ be a degenerate twisted product, let the vector fields $\partial/ \partial t\in\Gamma(T(0, 1))$and $U, V, W, Q\in\Gamma(TF),$ then
\begin{align}
&(1)\widetilde{R}(\frac{\partial}{\partial t}, \frac{\partial}{\partial t}, \frac{\partial}{\partial t}, \frac{\partial}{\partial t})=\widetilde{R}_{(0, 1)}(\frac{\partial}{\partial t}, \frac{\partial}{\partial t}, \frac{\partial}{\partial t}, \frac{\partial}{\partial t})=0;\nonumber\\
&(2)\widetilde{R}(\frac{\partial}{\partial t}, \frac{\partial}{\partial t}, \frac{\partial}{\partial t}, W)=\widetilde{R}(\frac{\partial}{\partial t}, W, \frac{\partial}{\partial t}, \frac{\partial}{\partial t})=0;\nonumber\\
&(3)\widetilde{R}(\frac{\partial}{\partial t}, \frac{\partial}{\partial t},  V, W)=\widetilde{R}(V, W, \frac{\partial}{\partial t}, \frac{\partial}{\partial t})=0;\nonumber\\
&(4)\widetilde{R}(\frac{\partial}{\partial t}, V, \frac{\partial}{\partial t}, W)=\frac{b}{2}\frac{\partial^{2} b}{\partial t^{2}}g_{F}(V, W)+\frac{b}{2}\frac{\partial}{\partial t}(J_{(0, 1)}(\frac{\partial b}{\partial t}))g_{F}(V, J_{F}W);\nonumber\\
&(5)\widetilde{R}(\frac{\partial}{\partial t}, V, W, U)=-\frac{1}{2}(\frac{\partial b}{\partial t}W(b)-bW(\frac{\partial b}{\partial t}))g_{F}(V, U)-\frac{1}{2}(\frac{\partial b}{\partial t}J_{F}W(b)-bJ_{F}W(\frac{\partial b}{\partial t}))\nonumber\\
&\times g_{F}(V, J_{F}U)+\frac{1}{2}(\frac{\partial b}{\partial t}U(b)-bU(\frac{\partial b}{\partial t}))g_{F}(V, W)+\frac{1}{2}(\frac{\partial b}{\partial t}J_{F}U(b)-bJ_{F}U(\frac{\partial b}{\partial t}))g_{F}(V, J_{F}W);\nonumber\\
&(6)\widetilde{R}(W, U, \frac{\partial}{\partial t}, V)=\frac{1}{4}(2bW(\frac{\partial b}{\partial t})-W(b)\frac{\partial b}{\partial t}-J_{(0, 1)}(\frac{\partial b}{\partial t})J_{F}W(b))g_{F}(U, V)+\frac{1}{4}(2bWJ_{(0, 1)}(\frac{\partial b}{\partial t})\nonumber\\
&-J_{F}W(b)\frac{\partial b}{\partial t}-J_{(0, 1)}(\frac{\partial b}{\partial t})W(b))g_{F}(U, J_{F}V)-\frac{1}{4}(2bU(\frac{\partial b}{\partial t})-U(b)\frac{\partial b}{\partial t}-J_{(0, 1)}(\frac{\partial b}{\partial t})J_{F}U(b))\nonumber\\
&\times g_{F}(W, V)-\frac{1}{4}(2bUJ_{(0, 1)}(\frac{\partial b}{\partial t})-J_{F}U(b)\frac{\partial b}{\partial t}-J_{(0, 1)}(\frac{\partial b}{\partial t})U(b))g_{F}(W, J_{F}V)\nonumber\\
&-\frac{b}{4}\frac{\partial b}{\partial t}(\mathcal{K}_{F}(U, V, W)-\mathcal{K}_{F}(W, V, U)-\mathcal{K}_{F}(U, J_{F}V, J_{F}W)+\mathcal{K}_{F}(W, J_{F}V, J_{F}U))\nonumber\\
&+\frac{b}{4}J_{(0, 1)}(\frac{\partial b}{\partial t})(\mathcal{K}_{F}(U, V, J_{F}W)-\mathcal{K}_{F}(W, V, J_{F}U)-\mathcal{K}_{F}(U, J_{F}V, W)+\mathcal{K}_{F}(W, J_{F}V, U));\nonumber\\
&(7)\widetilde{R}(U, V, W, Q)=b^{2}\widetilde{R}_{F}(U, V, W, Q)+\frac{1}{4}(b^{2}g^{*}_{(0, 1)}(db, db)+g^{*}_{F}(db, db))(g_{F}(U, W)g_{F}(V, Q)\nonumber\\
&-g_{F}(V, W)g_{F}(U, Q))+\frac{1}{4}(b^{2}g^{*}_{(0, 1)}(db\circ J_{(0, 1)}, db\circ J_{(0, 1)})+g^{*}_{F}(db\circ J_{F}, db\circ J_{F}))(g_{F}(U, J_{F}W)\nonumber\\
&\times g_{F}(V, J_{F}Q)-g_{F}(V, J_{F}W)g_{F}(U, J_{F}Q))+\frac{1}{4}(b^{2}g^{*}_{(0, 1)}(db\circ J_{(0, 1)}, db)+g^{*}_{F}(db\circ J_{F}, db))\nonumber
\end{align}
\begin{align}
&\times (g_{F}(U, J_{F}W)g_{F}(V, Q)-g_{F}(V, J_{F}W)g_{F}(U, Q)+g_{F}(U, W)g_{F}(V, J_{F}Q)-g_{F}(V, W)\nonumber\\
&\times g_{F}(U, J_{F}Q))-\frac{1}{4}(3U(b)W(b)-2bUW(b)+J_{F}U(b)J_{F}W(b)+2bg^{*}_{F}(\widetilde{\nabla}^{\flat}_{U}W, db))g_{F}(V, Q)\nonumber\\
&-\frac{1}{4}(3U(b)J_{F}W(b)-2bUJ_{F}W(b)+J_{F}U(b)W(b)+2bg^{*}_{F}(\widetilde{\nabla}^{\flat}_{U}W, db\circ J_{F}))g_{F}(V, J_{F}Q)\nonumber\\
&-\frac{1}{4}(3V(b)Q(b)-2bVQ(b)+J_{F}V(b)J_{F}Q(b)+2bg^{*}_{F}(\widetilde{\nabla}^{\flat}_{V}Q, db))g_{F}(U, W)\nonumber\\
&-\frac{1}{4}(3V(b)J_{F}Q(b)-2bVJ_{F}Q(b)+J_{F}V(b)Q(b)+2bg^{*}_{F}(\widetilde{\nabla}^{\flat}_{V}Q, db\circ J_{F}))g_{F}(U, J_{F}W)\nonumber\\
&+\frac{1}{4}(3V(b)W(b)-2bVW(b)+J_{F}V(b)J_{F}W(b)+2bg^{*}_{F}(\widetilde{\nabla}^{\flat}_{V}W, db))g_{F}(U, Q)\nonumber\\
&+\frac{1}{4}(3V(b)J_{F}W(b)-2bVJ_{F}W(b)+J_{F}V(b)W(b)+2bg^{*}_{F}(\widetilde{\nabla}^{\flat}_{V}W, db\circ J_{F}))g_{F}(U, J_{F}Q)\nonumber\\
&+\frac{1}{4}(3U(b)Q(b)-2bUQ(b)+J_{F}U(b)J_{F}Q(b)+2bg^{*}_{F}(\widetilde{\nabla}^{\flat}_{U}Q, db))g_{F}(V, W)\nonumber\\
&+\frac{1}{4}(3U(b)J_{F}Q(b)-2bUJ_{F}Q(b)+J_{F}U(b)Q(b)+2bg^{*}_{F}(\widetilde{\nabla}^{\flat}_{U}Q, db\circ J_{F}))g_{F}(V, J_{F}W)\nonumber\\
&+\frac{1}{4}bQ(b)(\mathcal{K}_{F}(V, W, U)-\mathcal{K}_{F}(U, W, V)-\mathcal{K}_{F}(V, J_{F}W, J_{F}U)+\mathcal{K}_{F}(U, J_{F}W, J_{F}V))\nonumber\\
&-\frac{1}{4}bJ_{F}Q(b)(\mathcal{K}_{F}(V, W, J_{F}U)-\mathcal{K}_{F}(U, W, J_{F}V)-\mathcal{K}_{F}(V, J_{F}W, U)+\mathcal{K}_{F}(U, J_{F}W, V))\nonumber\\
&+\frac{1}{4}bW(b)(\mathcal{K}_{F}(V, Q, U)-\mathcal{K}_{F}(U, Q, V)-\mathcal{K}_{F}(V, J_{F}Q, J_{F}U)+\mathcal{K}_{F}(U, J_{F}Q, J_{F}V))\nonumber\\
&-\frac{1}{4}bJ_{F}W(b)(\mathcal{K}_{F}(V, Q, J_{F}U)-\mathcal{K}_{F}(U, Q, J_{F}V)-\mathcal{K}_{F}(V, J_{F}Q, U)+\mathcal{K}_{F}(U, J_{F}Q, V)).\nonumber
\end{align}
\end{thm}

\vskip 1 true cm
\section{Acknowledgements}

The author was supported in part by  NSFC No.11771070. The author thanks the referee for his (or her) careful reading and helpful comments.

\vskip 1 true cm


\bigskip
\bigskip

\noindent {\footnotesize {\it S. Liu} \\
{School of Mathematics and Statistics, Northeast Normal University, Changchun 130024, China}\\
{Email: liusy719@nenu.edu.cn}

\noindent {\footnotesize {\it T. Wu} \\
{School of Mathematics and Statistics, Northeast Normal University, Changchun 130024, China}\\
{Email: wut977@nenu.edu.cn}

\noindent {\footnotesize {\it Y. Wang} \\
{School of Mathematics and Statistics, Northeast Normal University, Changchun 130024, China}\\
{Email: wangy581@nenu.edu.cn}

\end{document}